%% file: spi.tex
\newif\ifms
\msfalse
\ifms
    \documentclass[11pt]{article} 
    \usepackage[margin=1in]{geometry}
    \usepackage{setspace}
    \usepackage{lineno}           
\else
    \documentclass[11pt, letterpaper]{article}
    \usepackage[margin=1in]{geometry}
\fi

  \usepackage[centertags]{amsmath}
  \usepackage{amsfonts}
  \usepackage{amsthm}
  \usepackage{amssymb}
  \usepackage{bm}
  \usepackage{latexsym}
  \usepackage{graphics}
  \usepackage{graphicx}
  \usepackage{subcaption} 
  \usepackage{tabularx}
  \usepackage{caption}
  \usepackage{placeins}

  \usepackage[hidelinks]{hyperref}
  \ifms
  \else
  \hypersetup{
    pdftitle={The Shadow Price of Intelligence},
    pdfauthor={Elioth Sanabria},
    pdfsubject={Research article},
    pdfkeywords={LLM inference; capacity allocation; service degradation; queueing; duality},
    pdfstartview={FitH}
}
\fi
\usepackage{algorithm} 
\usepackage{algorithmicx}
\usepackage{algpseudocode}
\usepackage[style=authoryear, maxcitenames=2]{biblatex}
\renewbibmacro{in:}{%
  \ifentrytype{article}{}{\printtext{\bibstring{in}\intitlepunct}}%
}
\renewbibmacro*{volume+number+eid}{%
  \printfield{volume}%
  \setunit*{\addnbspace}%
  \printfield{number}%
  \setunit{\addcomma\space}%
  \printfield{eid}}
\DeclareFieldFormat[article]{number}{\mkbibparens{#1}}
\usepackage{tabularx}
\usepackage{array} 
\usepackage{booktabs}
\usepackage{multirow}
\newcolumntype{L}[1]{>{\raggedright\arraybackslash}p{#1}}
\newcolumntype{R}[1]{>{\raggedleft\arraybackslash}p{#1}}
\newcolumntype{C}[1]{>{\centering\arraybackslash}p{#1}}

\usepackage{tikz}
\usetikzlibrary{graphs, positioning, calc, arrows.meta, 3d, fit, backgrounds, matrix, decorations.pathreplacing}

\usepackage{pgfplots}
\pgfplotsset{compat=1.17}
\usepgfplotslibrary{fillbetween, groupplots}
\usepackage{pgfplotstable}

\usepackage{forest}

\input{Figures/latex_out/sec5_numbers}

\usepackage{fancyhdr}
\ifms
\else
    
\fi

\fancypagestyle{plain}{
    \fancyhf{}
    \fancyfoot[C]{\small\thepage}
    
}

\newcommand{\pr}{\mbox{\sf P}}
\newcommand{\ex}{{\bf\sf E}}               
\newcommand{\var}{\mbox{\sf Var}}

\def\ssk{\medskip\noindent}

\newcommand{\bN}{{\bf N}}

\newcommand{\cala}{{\cal A}}

\newcommand{\calh}{{\cal H}}

\newcommand{\calx}{{\cal X}}

\newcommand{\tS}{{\tilde S}}
\newcommand{\wh}{\widehat}

\def\ind{{\bf 1}}


\theoremstyle{definition} 
\newtheorem{thm}{Theorem}
\newtheorem{lem}[thm]{Lemma}
\newtheorem{pro}[thm]{Proposition}
\newtheorem{cor}[thm]{Corollary}

\newtheorem{rem}{Remark}

\ifms
    \newtheorem{exm}{Example} 
\else
    \newtheorem{exm}{Example}[section] 
\fi

\title{\vspace{-1.5cm} The Shadow Price of Intelligence\\[0.25cm]{\large Quality Degradation and Retry Feedback in LLM Inference}}

\ifms
    \author{Anonymous Authors}
    \date{\small \today}
\else
    \author{
        Elioth Sanabria \\[0.2cm]
        \small Department of Decision and Technology Analytics \\ 
        \small Lehigh University - College of Business\\
        \small \texttt{els626@lehigh.edu}
    }
    \date{\small \today}
\fi

\begin{document}

\maketitle
\begin{abstract}
  \noindent Large language model providers are compute constrained, and a common response to congestion is to degrade service: route queries to smaller models, cut reasoning budgets, or shorten context windows. A user who receives an unsatisfactory answer either retries, adding load while the system is busy, or abandons, possibly as churn. We model inference serving as a control problem in which service quality is the provider's decision. A two-regime fluid approximation with retries folded into the service rate gives closed-form trajectories and crossing times, so the dynamic program over quality tiers needs no simulation. We characterize a region in which a weaker tier uses less energy per task but more capacity per task, and show that in the fluid model reactive throttling to such a tier can keep the queue from draining. With heterogeneous customers the problem is a transportation problem with retry-inflated capacity coefficients; optimal routing follows a $c\mu$-type index, and the dual of the capacity constraint gives the marginal cost of one more query by class and time. We illustrate the method on five instances calibrated to public benchmark data.

\vspace{0.4cm}\noindent\textbf{Keywords:} LLM inference; capacity allocation; service degradation; queueing; duality.
\end{abstract}
\vspace{0.5cm}
\setlength{\parskip}{0.4em}
\section{Introduction}\label{sec:intro}

Improvements in the performance of Large Language Models (LLM) are disrupting every sector of the information economy, and scaling data centers and hardware to meet the resulting demand is a central problem for the largest inference providers. A 2026 Anthropic careers posting reads:
\begin{quote}
``Anthropic is compute-constrained, and how we allocate that compute is one of the highest-leverage decisions we make as a company. Today, allocation choices are only loosely tied to the user outcomes we ultimately care about: retention, lifetime value, and the experience of people relying on Claude.''
\end{quote}
When a language model provider is congested, it does not turn customers away. It degrades them: it routes queries to smaller or quantized models, cuts the compute budget for reasoning, and truncates the context the model is allowed to read. A degraded query consumes less energy, less memory, and less server time, so per-query cost accounting reports a saving.

The degradation is not free: a user or agent that receives a degraded answer that is not fully satisfactory may re-ask or exit, potentially as churn. Given limited compute, the question is who should receive degraded service, and when. Clearly, certain types of users will feel degradation sharply, while others working on simpler tasks will not. We compute the marginal cost of a query by class and by hour: the shadow price of intelligence. Resolving the tradeoff between the per-query saving and the churn requires a framework that is tractable but accurate enough for the allocation problem.

At first sight the problem appears intractable: multi-class dynamics with retries and non-homogeneous arrivals explode the state space, making simulation costly, and the noise of sampled paths compounds with the noise of DP transitions. We compose a two-regime closed-form fluid approximation of the dynamics with the retries folded into the service time, and characterize the static and dynamic shape of the optimal policies in this setting.

{\bf Contributions:} Our contributions are the following.

\begin{enumerate}
  \item We give conditions under which a weaker tier uses less energy per initiated task and more slot-time per initiated task (Proposition~\ref{pro:trap}).
  \item We derive a closed-form two-regime fluid recursion with endogenous retries (Lemma~\ref{lem:fluid}) and characterize when degradation fails to drain a saturated queue (Proposition~\ref{pro:spiral}), when a threshold rule fails to release (Corollary~\ref{cor:latch}), and when anticipatory degradation dominates reactive uniform degradation (Proposition~\ref{pro:anticipation}).
  \item With heterogeneous customers, the one-period problem is a transportation problem with retry-inflated capacity coefficients; its optimum follows an index rule (Proposition~\ref{pro:index}), targeted routing weakly dominates uniform routing (Proposition~\ref{pro:dominance}), and the dual of the capacity constraint is the marginal cost of a query by class and time (Section~\ref{sec:shadow}).
  \item We implement the dynamic program with the closed-form transition, compare it with Euler integration and with simulation on five calibrated instances, and distill the policy into decision trees (Sections~\ref{sec:practice} and~\ref{sec:poc}).
\end{enumerate}

\subsection{Related Literature}\label{sec:lit}

Four streams of literature are related. In each, service quality is a primitive; here it is the decision.

\ssk\emph{Retrials, abandonment, and rational queueing.} In retrial queues, blocked customers enter an orbit \parencite{FalinTempleton1997, Artalejo2008}; in the Erlang-A abandonment tradition, waiting customers renege, and staffing rules control the fraction lost \parencite{GarnettMandelbaumReiman2002, ZeltynMandelbaum2005}. Our customers do both, but at a different point: they defect \emph{after} service, in response to its quality, and the quality is the firm's control. Where Erlang-A asks how many servers hold abandonment below a target, we ask which quality menu holds the retry feedback below self-sustainability. The strategic queueing literature beginning with \textcite{Naor1969} and synthesized by \textcite{HassinHaviv2003} endogenizes customer behavior through equilibrium reasoning about delay; our customers respond to realized quality rather than anticipated delay, which removes the fixed point of equilibrium arrival rates but introduces a completion-driven feedback whose gain the firm sets tier by tier. Two antecedents are closer. Queues with Bernoulli feedback \parencite{Takacs1963} recycle a completed job with a fixed probability, and \textcite{deVericourtZhou2005} route calls under a resolution probability, unresolved customers calling back after service. Our extension makes the resolution probability quality-dependent, so retry-induced load and expected retention loss respond to the same control.

\ssk\emph{Quality-speed trade-offs and demand-service interactions.} A line of work models servers who choose a speed-quality point, with congestion penalizing slowness and errors penalizing haste \parencite{AnandPacVeeraraghavan2011, HoppIravaniYuen2007, AlizamirVanBerghVaraktarakis2013, AtaShneorson2006}; a complementary line models service quality feeding back into future demand through loyalty and word of mouth \parencite{Hall2000, Gans2002, AflakiPopescu2014}. We combine the two channels and add a third that the LLM setting makes unavoidable: failed service returns to the queue \emph{immediately}, so quality degradation moves the arrival rate on the same time scale as the backlog it was meant to relieve. The effective-throughput inversion (Proposition~\ref{pro:trap}) and instability threshold (Proposition~\ref{pro:spiral}) have, to our knowledge, no analogue in either line once the feedback gain is a decision variable.

\ssk\emph{Amplification in supply networks.} The bullwhip effect shows how individually rational hedging amplifies demand variability as it propagates through a supply network \parencite{LeePadmanabhanWhang1997, ChenDreznerRyanSimchiLevi2000}; experimental evidence traces much of the amplification to the managers' own response to the feedback \parencite{Sterman1989, CrosonDonohue2006}. The mechanism here is structural: the retry feedback is individually rational cost-cutting that amplifies the demand it faces, and it arises in the fluid model with known parameters and continuous control, without misperception.

\ssk\emph{Server allocation, scheduling, and LLM serving.} The autoscaling literature asks how many servers to run against time-varying load, trading energy against latency \parencite{GandhiHarcholBalter2012, LinWiermanAndrewThereska2013}; square-root staffing and fluid approximations for transient many-server systems supply its analytical backbone \parencite{HalfinWhitt1981, MandelbaumMassey1995, Whitt2002}. The systems literature on LLM inference optimizes the serving layer itself: batching, paged attention, memory management \parencite{KwonEtAl2023, YuEtAl2022}. We hold the fleet and the serving stack fixed and ask the question orthogonal to both: what quality to serve each class, and at what implicit price. The index policy of Section~\ref{sec:who} is a relative of the $c\mu$ rule \parencite{VanMieghem1995}; the difference is that the capacity purchased by degradation is itself deflated by the retry multiplier of the class being degraded. For the interpretable-policy question of Section~\ref{sec:tree} we draw on optimal classification trees \parencite{BertsimasDunn2017}, tree extraction from black-box models \parencite{BastaniKimBastani2018}, and tree-structured Markov-decision-policy representations \parencite{SanabriaYaoLam2021}; for delayed retries, a natural extension, see the nonlinear-Markov formulation of \textcite{DiekerHackmanWangYan2026}.

Section~\ref{sec:model} presents the model. Section~\ref{sec:analysis} contains the statics, the dynamics, the matching problem, and the shadow price. Section~\ref{sec:practice} addresses the questions a deployment must answer: how the program is solved, what variance does to the fluid conclusions, and how to make the policy legible. Section~\ref{sec:poc} is the numerical exercise, and Section~\ref{sec:conclusion} concludes.

\section{The Model}\label{sec:model}

In this section we present our model, adding customer heterogeneity in Section \ref{sec:who}. The full notation is presented in Appendix~\ref{app:notation}. 

\ssk\textbf{LLM inference serving and its cost.} A provider operates $k$ servers, each hosting $m$ concurrent inference slots, for example a discrete GPU cluster, for a capacity of $mk$ concurrent jobs. With $N$ jobs in system and the fleet serving at quality tier $j=0,1,\dots J$, $j=0$ the strongest model and $j=J$ the weakest, the operating cost per hour is (later extended to a combination of tiers):
\begin{equation}\label{eq:cost}
c(k,N) \;=\; \gamma\bigl[k\,w_0 + \min\{N,\,mk\}\,\Delta w_j\bigr] \;+\; \gamma\, c_m N^{\kappa} \;+\; C_{\mathrm{SLA}}\!\left(\frac{(N-mk)^{+}}{mk}\right)^{\!\beta_{\mathrm{SLA}}}.
\end{equation}
The three terms approximate the physics of inference serving. The first is electricity at price $\gamma$: an idle floor $w_0$ per powered server plus an active draw $\Delta w_j$ per busy slot. The second is the memory overhead of holding $N$ attention contexts resident, superlinear with exponent $\kappa>1$ because their KV caches compete for a fixed pool of GPU memory (HBM). The third is a service-level penalty at scale $C_{\mathrm{SLA}}$, a constant with units of dollars per hour: the penalty rate incurred when the excess backlog equals capacity, convex in the relative backlog with elasticity $\beta_{\mathrm{SLA}}=2$. While simplified, these terms approximate the energy tradeoff of an LLM inference system with $N$ customers and $k$ servers.

\ssk\textbf{The catalog of models.} The LLM inference provider has access to a catalog of models noted as $j=0,\dots,J$, tier $0$ the strongest and $J$ the weakest. A routing mix $\varphi$ across tiers is also possible. Each model has: a power draw per active slot $\Delta w_j$, non-increasing in $j$ because quantized and distilled models draw less power; a mean service time $\ex[S_j]$, rate $\mu_j = 1/\ex[S_j]$, non-decreasing in $j$ because smaller models have higher token throughput; and a dissatisfaction probability $d_j\in[0,1)$, \emph{non-decreasing} in $j$, the probability that the answer fails the user, potentially an agent: a smaller model fails more often than a larger one with deeper reasoning. All tiers have $d_0\ge 0$. Section~\ref{sec:data} discusses estimation.

\ssk\textbf{The feedback loop.} What does an unsatisfied user do? With probability $\rho$ they re-ask, and the query re-enters the arrival stream indistinguishable from fresh demand. With probability $1-\rho$ they abandon, and the abandoning customer churns with probability $p_r$, taking an expected lifetime margin of $\ell$ dollars: an \emph{expected LTV loss} of $p_r\ell$. Each attempt fails and returns with probability $d_j\rho$, independently across attempts, so the number of attempts per initiated task is geometric with expected value:
\begin{equation}\label{eq:multiplier}
M_j \;=\; \frac{1}{1-d_j\rho}.
\end{equation}
Two events of interest are related to $M_j$ and are used throughout. The probability that an initiated task eventually ends in a satisfied answer is $\pr_{j}(\mathrm{success})=(1-d_j)/(1-d_j\rho)$; aggregate attempts per successful task are $M_j/\pr_j(\mathrm{success})=1/(1-d_j)$, while conditional on eventual success the attempt count has mean $M_j$. Capacity is consumed per initiated task because $\lambda_x$ counts fresh arrivals (Section~\ref{sec:who}).

Define the \emph{effective service time} and the \emph{effective throughput} of a tier,
\begin{equation}\label{eq:effective}
\tS_j \;:=\; M_j\,\ex[S_j],
\qquad\qquad
\theta_j \;:=\; \frac{mk}{\tS_j} \;=\; (1-d_j\rho)\,k\mu_j m,
\end{equation}
the slot-time consumed per initiated task and the rate of terminal departures (satisfied or abandoned) per unit time at saturation.

\ssk\textbf{The demand.} Arrivals form a stochastic process $\{D_t\}$; $r_t=\ex[D_t\mid\calh_t]$ is the expected arrivals at hour $t$ given the observable history.\footnote{\label{fn:demand}One bookkeeping consequence of $d_0>0$: logged arrivals already contain the incumbent tier's retries, so the fresh-demand rate is the deconvolution $r_t=(1-d_{j(t)}\rho)\,r_t^{\mathrm{logged}}$ under the incumbent tier $j(t)$; the feedback term of \eqref{eq:fluid} then supplies all retries, with no double-counting.}

\ssk\textbf{The static case: provisioning is a Newsvendor.} Let $p-c$ be the margin on a served query, $c_o$ the hourly cost of an idle slot, and $p_r^{\mathrm{ns}}\ell$ the stockout charge on an unserved one.\footnote{The stockout charge $p_r^{\mathrm{ns}}\ell$ is churn on non-service, a distinct behavioral channel from the post-service churn of the feedback loop; the stockout probability carries its own symbol $p_r^{\mathrm{ns}}$ to keep the channels apart, and the numerical examples assume $p_r^{\mathrm{ns}} = p_r$.} Capacity $q$ against random offered load $L=D\,\tS_0$, demand in effective slot-time, earns $\ex[(L\wedge q)(p-c)-(q-L)^{+}c_o-(L-q)^{+}p_r^{\mathrm{ns}}\ell]$, and the first-order condition yields:
\begin{equation}\label{eq:newsvendor}
q^{\star}=F_L^{-1}\!\left(\frac{p-c+p_r^{\mathrm{ns}}\ell}{\,p-c+p_r^{\mathrm{ns}}\ell+c_o\,}\right)
\qquad\leadsto\qquad
k^{\star}_t=\frac{1}{m}\Bigl[r_t\,\tS_0+\sqrt{\var(L_t)}\;\Phi^{-1}(\mathrm{CR})\Bigr],
\end{equation}
where the Gaussian form on the right is square-root staffing at the critical ratio $\mathrm{CR}$ \parencite{HalfinWhitt1981}. The buffer term trades the cost of servers against the churn cost $p_r^{\mathrm{ns}}\ell$ of refused fresh demand.

\section{Analysis}\label{sec:analysis}

We start with the static analysis, which shows the serving tradeoffs; we then expand to the transient dynamics, introduce customer heterogeneity, and derive duality results that yield the shadow price of intelligence.

\subsection{The Two Frontiers}\label{sec:frontiers}

The question of this section is whether a weaker tier model is actually cheaper to serve. The naive comparison is per raw query; the correct one is per \emph{task}. Per initiated task, tier $j$ consumes on average:
\begin{equation}\label{eq:ledger}
\underbrace{\tS_j\,\Delta w_j\,\gamma}_{\text{energy}}\,,\qquad
\underbrace{\tS_j}_{\text{slot-time}}\,,\qquad
\underbrace{M_j\,d_j\,(1-\rho)\,p_r\ell}_{\text{lifetime value loss}}\,,
\end{equation}
the first two terms are the energy and time consumed by the task. For the third: an initiated task drags $M_j$ attempts behind it, and the terminal attempt, dissatisfactory with probability $d_j$, does not retry $(1-\rho)$ and loses the customer with probability $p_r$, at an LTV $\ell$. Comparing tier $j$ against tier $0$ column by column: tier $j$ saves energy per initiated task iff the first inequality below holds, and relieves capacity iff the second does,
\begin{equation}\label{eq:frontiers}
\underbrace{\tS_j\,\Delta w_j \;<\; \tS_0\,\Delta w_0}_{\text{energy frontier}}\,,
\qquad\qquad
\underbrace{\tS_j \;<\; \tS_0}_{\text{capacity frontier}}\,.
\end{equation}
These are two different conditions, and because a weak tier by construction draws less power, the energy frontier is strictly easier to cross: there is a region in which one frontier has been crossed and the other has not.

\begin{pro}[The trap]\label{pro:trap}
Fix tiers $0$ and $j$ with $\Delta w_j<\Delta w_0$. Tier $j$ saves energy per initiated task but consumes strictly more slot-time per initiated task if and only if
\[
1 - \frac{\ex[S_j]}{\tS_0} \;<\; d_j\rho \;<\; 1 - \frac{\ex[S_j]\,\Delta w_j}{\tS_0\,\Delta w_0}.
\]
\end{pro}
\begin{proof}
The energy frontier of \eqref{eq:frontiers} is $\tS_j\Delta w_j<\tS_0\Delta w_0$ and the capacity frontier is $\tS_j>\tS_0$. Substituting $\tS_j=\ex[S_j]/(1-d_j\rho)$ and solving each inequality for $d_j\rho$: the energy frontier gives $d_j\rho<1-\ex[S_j]\Delta w_j/(\tS_0\Delta w_0)$, and the capacity frontier is violated, strictly more slot-time per initiated task, exactly when $d_j\rho>1-\ex[S_j]/\tS_0$. The interval is nonempty when $\ex[S_j]/\tS_0>\ex[S_j]\Delta w_j/(\tS_0\Delta w_0)$, that is, $\Delta w_j<\Delta w_0$.
\end{proof}

The following example illustrates the tradeoff explicitly.

\begin{exm}[Two tiers of the same model]\label{ex:twotiers}
A LLM provider serves both a strong model, tier $0$, drawing $\Delta w_0=4$\,kW per active slot at $\ex[S_0]=100$ milliseconds, and a distilled sibling, tier $1$, drawing $\Delta w_1=2$\,kW at $\ex[S_1]=60$ milliseconds. With $d_0=0$, $d_1=0.6$, and a failed user re-asking with probability $\rho=0.8$, $\tS_0=\ex[S_0]=100$. Is the switch profitable?

The multiplier is $d_1\rho=0.48$, so $M_1=1/0.52\approx 1.92$ and each initiated task costs nearly two attempts, $\tS_1=1.92\cdot 60=115.4$ slot-milliseconds. The service ratio $\tS_1/\tS_0=1.15$ is inside the trap region $(1,\,\Delta w_0/\Delta w_1)=(1,2)$. Per initiated task, tier 1 spends $\tS_1\Delta w_1=230.8$ kW-ms of energy, 230.8 joules against tier 0's $400$, a 42\% saving on electricity, while occupying $115.4$ slot-milliseconds against tier 0's $100$: the cheap tier consumes 15\% \emph{more} capacity per initiated task. During a surge, when capacity rather than electricity binds, this is the wrong trade, and per-query accounting reports it as a saving. For this example the trap interval of Proposition~\ref{pro:trap} is $d_1\rho\in(0.40,\,0.70)$; Figure \ref{fig:frontiers} is the schematic.
\end{exm}

\begin{figure}[t]
\centering
\begin{tikzpicture}[scale=0.9]
\draw[-{Stealth}] (0,0) -- (7.2,0) node[below left, font=\footnotesize] {intelligence $\to$};
\draw[-{Stealth}] (0,0) -- (0,4.6);
\node[rotate=90, font=\footnotesize] at (-0.35,2.3) {cost per answer};
\draw[thick, gray, dashed] plot[smooth] coordinates {(1,0.7) (2.5,1.1) (4,1.8) (5.5,2.7) (6.5,3.6)};
\node[gray, font=\footnotesize, anchor=west] at (4.3,1.3) {posted};
\draw[very thick, blue!55!black] plot[smooth] coordinates {(1,4.1) (2.0,2.6) (3.0,1.9) (4,1.9) (5.5,2.8) (6.5,3.7)};
\node[blue!55!black, font=\footnotesize, anchor=south west] at (1.5,3.1) {realized};
\fill[gray] (1,0.7) circle (1.8pt);
\fill[blue!55!black] (1,4.1) circle (2pt);
\fill[blue!55!black] (6.5,3.7) circle (2pt);
\node[font=\scriptsize, anchor=south east] at (6.85,3.85) {strong model ($j{=}0$)};
\node[font=\scriptsize, anchor=west] at (1.25,0.6) {distilled model ($j{=}J$)};
\draw[-{Stealth}, thick, red!60!black] (1,0.95) -- (1,3.85);
\end{tikzpicture}
\caption{The two frontiers (schematic). The posted frontier, per raw query, is monotone: weaker is cheaper. The realized frontier, per initiated task, carries the multiplier and the churn charge of \eqref{eq:ledger} and bends back up as $d_j\rho$ grows; past the energy break-even the ordering of the tiers inverts. The red arrow marks the inversion at the distilled model: the vertical gap between its posted and realized cost is the retry-induced load plus the churn.}
\label{fig:frontiers}
\end{figure}

\begin{rem}[The third frontier]\label{rem:memory}
The frontiers \eqref{eq:frontiers} compare energy and slot-time, but the cost \eqref{eq:cost} carries a third, occupancy-dependent term: memory, $\gamma c_m N^{\kappa}$. Routing traffic of rate $\lambda$ to tier $j$ shifts the number in server by $\Delta N=\lambda\bigl(\tS_0-\tS_j\bigr)$, worth $\gamma c_m\kappa N^{\kappa-1}\Delta N$ per hour at level $N$. Comparing the energy against the tier-to-tier churn difference $[M_j d_j-M_0 d_0](1-\rho)\,p_r\ell$ yields a closed-form break-even expected LTV loss $(p_r\ell)^{\star}(N)$, increasing in $N$ when $\tS_j<\tS_0$: below $N^{\star}$ churn is the more costly, above it the opposite. The derivation is in Appendix~\ref{app:memory}.
\end{rem}

\subsection{Transient Dynamics}\label{sec:dynamics}

We now study the problem dynamically, posing the throttling timing as a control problem, and develop a fluid approximation of the queue dynamics. A finite-server system lives in one of two regimes. Above capacity, all $mk$ slots are busy and jobs complete at throughput $k\mu_j m$, the maximum the fleet can attain with every slot serving at rate $\mu_j$; the backlog changes linearly, accumulating at $r>k\mu_j m$ and depleting at $r<k\mu_j m$. Below capacity, service scales with occupancy and the system relaxes exponentially toward the Little's law equilibrium $r/\mu_j$ \parencite{Whitt2002, EickMasseyWhitt1993}, $r\tS_j$ once the feedback below is taken into account. One modification separates our system from the textbook one: the arrival rate is endogenous, because completions at a degraded tier feed retries back into the stream, and a resident job's total residence is a geometric number of exponential services, hence exponential with rate $\mu^{\mathrm{eff}}_j$. The lemma below gives the closed-form fluid solution of the instantaneous-feedback model of Section~\ref{sec:model}; the multi-class extension of Section~\ref{sec:dp} adds a frozen-share approximation.

\begin{lem}[Two-regime fluid dynamics with endogenous arrivals]\label{lem:fluid}
Let $q_j := d_j\rho$ and define the \emph{effective service rate} and equilibrium
\[
\mu^{\mathrm{eff}}_j \;:=\; (1-q_j)\,\mu_j \;=\; \frac{1}{\tS_j},
\qquad\qquad
\bar N \;:=\; \frac{r}{\mu^{\mathrm{eff}}_j} \;=\; r\,\tS_j .
\]
The fluid ODE $dN/dt = r - (1-q_j)\,\mu_j\,\min\{N,\,mk\}$ has the exact solution
\begin{equation}\label{eq:fluid}
N_{t+\Delta t} \;=\;
\begin{cases}
N_t + (r-\theta_j)\,\Delta t & \text{if } N_t \ge mk,\\[4pt]
N_t\, e^{-\mu^{\mathrm{eff}}_j \Delta t} + \bar N\bigl(1-e^{-\mu^{\mathrm{eff}}_j\Delta t}\bigr) & \text{if } N_t < mk,
\end{cases}
\end{equation}
with crossing times $t^\star=(mk-N_t)/(r-\theta_j)$ and
$t^\star=(\mu^{\mathrm{eff}}_j)^{-1}\ln\bigl[(\bar N-N_t)/(\bar N-mk)\bigr]$. If $N$ crosses $mk$ within a step, the active branch is applied up to the crossing time and the other branch from there.
\end{lem}
\begin{proof}
Assume $r$ and the tier $j$ are constant over the step $[t,t+\Delta t]$ and that $N_t$ does not cross $mk$ within it; no further smoothness on $r_t$ is needed because the endogeneity is folded into the rate.
\emph{Recovering branch ($N<mk$).} The ODE is the linear equation $dN/dt = r-\mu^{\mathrm{eff}}_j N$, whose solution over $\Delta t$ is $N_{t+\Delta t}=\bar N+(N_t-\bar N)\,e^{-\mu^{\mathrm{eff}}_j\Delta t}$, the displayed exponential form.
\emph{Saturated branch ($N\ge mk$).} The drift is the constant $r-(1-q_j)k\mu_j m = r-\theta_j$, integrated directly to $N_{t+\Delta t}=N_t+(r-\theta_j)\Delta t$.
The crossing times follow by solving each branch for $N=mk$: the linear branch directly, the exponential branch by isolating the exponential and taking logarithms.
\end{proof}

\begin{cor}[Recovery time scale, fluid model]\label{cor:slowing}
The degraded system mean-reverts at rate $\mu^{\mathrm{eff}}_j = 1/\tS_j$: its recovery timescale is the effective service time, $M_j$ times the posted one. Define the \emph{fluid stability boundary} $q^{\star}_j := 1 - r/(mk\mu_j)$, the value of $q_j=d_j\rho$ at which the saturated drift $r-\theta_j$ vanishes. As $q_j = d_j\rho \to 1$, the below-capacity relaxation timescale $M_j\ex[S_j]$ and the equilibrium $r\,\tS_j$ diverge; the boundary $q^{\star}_j<1$ is reached earlier, at $r\,\tS_j=mk$. A controller that estimates recovery time from posted service times underestimates it by the multiplier, and the misestimate grows as the system nears the boundary. Corollary~\ref{cor:latch} shows the complementary situation.
\end{cor}

Note that $j$ indexes the tier currently serving, the strong tier $0$ included; the regime is determined by the state $N$ alone, not by the tier. Between regime switches the trajectory is a formula and the switch epochs are the crossing times of the lemma (the 139 milliseconds of Example~\ref{ex:surge} below is the logarithm), so evaluating an entire horizon costs a handful of elementary operations: no time grid, no probability mass function, and, as Section~\ref{sec:shadow} exploits, policy-conditioned sensitivities by finite differences at the same price.

\begin{rem}[Scope of the lemma]\label{rem:mean}
The lemma solves the fluid model, not the stochastic process; Section~\ref{sec:poc} compares it with simulation. The closed form depends on two modeling choices: instantaneous retries, which fold the feedback into the service-rate modification $\mu_j\mapsto(1-q_j)\mu_j$; and memorylessness, which makes the drift piecewise linear in the state. With delayed retries the closed form is lost. The fluid model discards the variance of the branching process, an overdispersion that grows without bound as $d_j\rho \to 1$; Section~\ref{sec:variance} quantifies it.
\end{rem}

We illustrate the formulas with an example:

\begin{exm}[Throttling at the surge]\label{ex:surge}
A datacenter runs $k=150$ clusters of $m=20$ slots, so the total capacity is $mk=3$ thousand concurrent jobs, with tier-0 service time $\ex[S_0]=100$ milliseconds as in Example~\ref{ex:twotiers}: $\mu_0=10$ per second, and with $d_0=0$ the departure rate is $\theta_0=30$ thousand jobs per second. Demand sits at rest at $r_{\mathrm{calm}}=20$ thousand queries per second, the tier-0 resting level $N(0)=r_{\mathrm{calm}}\tS_0=2$ thousand jobs. A surge arrives at $r_{\mathrm{surge}}=33.3$ thousand queries per second. The controller uses a common reactive rule: when $N$ crosses a trigger level $\wh N_{\mathrm{fire}}$, degrade everyone to tier 1, with $\ex[S_1]=60$ milliseconds ($\mu_1\approx 16.7$ per second, raw throughput 50 thousand per second) and, for this traffic, $d_1=0.6$ and $\rho=0.8$ as in Example~\ref{ex:twotiers}. What happens?

The system starts below capacity, relaxing toward $r_{\mathrm{surge}}/\mu_0=33.3/10\approx 3.33$ thousand. Since $3.33>3$, it crosses the capacity line first, at $t$ solving $2e^{-10t}+3.33(1-e^{-10t})=3$, that is $e^{-10t}=0.25$, or $t=\ln 4/10\approx 0.139$ seconds in. From there the system saturates and grows linearly at $r_{\mathrm{surge}}-\theta_0=3.3$ thousand jobs per second. Suppose the trigger sits at $\wh N_{\mathrm{fire}}=3.5$ thousand, so the rule fires 150 milliseconds later; the reactive policy analyzed in Proposition~\ref{pro:anticipation} below instead fires at $\wh N_{\mathrm{fire}}=mk$, i.e.\ at $\tau=t_1+0.139$ s. On paper the throttle looks decisive: raw capacity jumps from 30 to 50 thousand per second against demand of 33.3, so the queue should drain at 16.7 thousand per second.

Now account for the feedback loop. In the fluid model, in saturation, completions run at 50 thousand per second, and a fraction $d_1\rho=0.48$ of them come back: 24 thousand retries per second, so $r^{\mathrm{eff}}=33.3+24=57.3>50$. The queue does not drain. It grows at 7.3 thousand jobs per second, more than twice its pre-throttle rate, while the service-level term of \eqref{eq:cost} costs quadratically in the queue length. In plain words: once the retries it generates are subtracted from its raw output, tier 1 delivers $\theta_1=(1-d_1\rho)\times 50=26$ thousand fresh answers per second against tier 0's $\theta_0=30$. Each second of the unstable feedback also adds $50\times 0.6\times 0.2=6$ thousand abandonments, which enter the churn cost and not the operating cost \eqref{eq:cost}. Appendix~\ref{app:boomerang}, Figure~\ref{fig:boomerang} sketches the path.

What would work instead? Not this throttle at another time: with a single class the instance sits inside the trap of Proposition~\ref{pro:trap} ($\tS_1=115.385>100=\tS_0$), so degrading everyone, at any moment, \emph{adds} effective load. The anticipatory path requires heterogeneity: suppose $30\%$ of the traffic is insensitive to the weak tier, say $d_1=0.05$ and $\rho=0.3$ for that fraction. The same arithmetic: $d_1\rho=0.015$, so $M_1=1/0.985$ and $\tS_1=60/0.985=60.914$ milliseconds for the insensitive class, against $\tS_1=60/0.52=115.385$ milliseconds for the sensitive class. The offered load, arrivals times effective service time class by class: the sensitive $70\%$ stays on tier 0 and occupies $0.7\times 33.3\times 0.100=2.331$ thousand slots, the insensitive $30\%$ moves to tier 1 and occupies $0.3\times 33.3\times 0.060914=0.609$ thousand, for a total of $2.940<3$ thousand. Degrading the insensitive fraction \emph{before} the surge holds the system below capacity: it relaxes toward $2.940$ thousand, never saturates, and the feedback never becomes self-sustaining, the dashed path of Figure~\ref{fig:boomerang}. Who is insensitive, and how to identify them, is the subject of Section~\ref{sec:who}.
\end{exm}

We formalize this behavior:

\begin{pro}[Self-sustaining feedback threshold]\label{pro:spiral}
In the fluid model: suppose the system is saturated, $N_t\ge mk$, and the fleet switches to tier $j$. The queue drains back below $mk$ if and only if $r<\theta_j$, the tier's effective throughput \eqref{eq:effective}. Above this threshold the retry feedback is self-sustaining: the drift is positive, $N$ never returns below $mk$, the backlog and, whenever $d_j(1-\rho)>0$, the churn cost grow without bound.
\end{pro}
\begin{proof}
With $N\ge mk$, all $mk$ slots are busy, completions run at a rate $k\mu_j m$, and a fraction $d_j\rho$ re-enter, so $r^{\mathrm{eff}}=r+d_j\rho\,k\mu_j m$ and the net change of $N$ per unit time is
\[
\frac{\Delta N}{\Delta t}\;=\;r^{\mathrm{eff}}-k\mu_j m\;=\;r-(1-d_j\rho)\,k\mu_j m\;=\;r-\theta_j.
\]
The queue drains if and only if $r-\theta_j<0$. If positive, $N$ stays above $mk$, the saturated branch of \eqref{eq:fluid} continues to apply, and the same positive rate persists.
\end{proof}

\begin{cor}[The latch]\label{cor:latch}
In the fluid model, with $r_{\mathrm{calm}}<\theta_1$ and $0<\wh N_{\mathrm{rel}}\le mk$: consider the two-thresholds reactive rule, which degrades to tier $1$ when $N$ exceeds the trigger $\wh N_{\mathrm{fire}}$ of Example~\ref{ex:surge}, and releases back to tier $0$ when $N$ falls below a release level $\wh N_{\mathrm{rel}}\le\wh N_{\mathrm{fire}}$. Start the clock when the surge ends: from time $0$ on, the arrival rate is a constant $r_{\mathrm{calm}}<r_{\mathrm{surge}}$, the system still degraded, and $N_0>\wh N_{\mathrm{rel}}$. The rule returns to the strong tier if and only if
\[
r_{\mathrm{calm}}\,\tS_1 \;<\; \wh N_{\mathrm{rel}}.
\]
The left side is where the degraded system comes to rest: by Little's law, arrivals at rate $r_{\mathrm{calm}}$ each holding a slot for an effective time $\tS_1$, retries included, keep $r_{\mathrm{calm}}\,\tS_1$ jobs in the system, and the release fires only if that resting level lies below the release line. When it holds, the release is final: back on tier $0$ the same logic applies with $\tS_0<\tS_1$ in the trap region, so the system settles at $r_{\mathrm{calm}}\,\tS_0$, below the line it just crossed, and the trigger never re-fires. When it fails, the throttle never releases though the surge is over: the trajectory stays degraded for the rest of the horizon, and churn accrues at $r_{\mathrm{calm}}\,M_1 d_1(1-\rho)\,p_r\ell$ per hour. A release level below $r_{\mathrm{calm}}\,\tS_1$ never fires in the fluid model.
\end{cor}
\begin{proof}
For $t>0$ the arrival rate is $r_{\mathrm{calm}}$ and the tier is $1$, so by \eqref{eq:fluid} the trajectory decreases monotonically toward $\bar N_1:=r_{\mathrm{calm}}\,\tS_1$, linearly while saturated, then along the exponential branch below $mk$. If $\bar N_1<\wh N_{\mathrm{rel}}$, the trajectory crosses the release line in finite time, by the crossing formulas of Lemma~\ref{lem:fluid}. If $\bar N_1\ge\wh N_{\mathrm{rel}}$, the trajectory approaches $\bar N_1$ from above and stays above the release line: $N_t>\wh N_{\mathrm{rel}}$ for all $t$, and the rule never releases. Durability after release: once the tier is $0$, the same argument drives $N$ toward $\bar N_0:=r_{\mathrm{calm}}\,\tS_0$. In the trap region $\tS_0<\tS_1$ gives $\bar N_0<\bar N_1<\wh N_{\mathrm{rel}}\le\wh N_{\mathrm{fire}}$, so the trajectory never re-crosses the trigger. Outside the trap, $\bar N_0>\bar N_1$ and the trajectory climbs back toward $\bar N_0$; the release holds iff $\bar N_0\le\wh N_{\mathrm{fire}}$, and when it does not, the rule re-fires and cycles indefinitely.
\end{proof}

\textbf{Notation for Proposition~\ref{pro:anticipation}.} Classes $x$ carry fresh rates $\lambda_x^{\mathrm{calm}}$ off $[t_1,t_2]$ and $\lambda_x^{\mathrm{surge}}\ge\lambda_x^{\mathrm{calm}}$ on it; $L^{\mathrm{calm}}(a)=\sum_x\lambda_x^{\mathrm{calm}}\,\tS_{a(x)}(x)$ and $L^{\mathrm{surge}}(a)$ likewise; ``$0$'', ``$1$'', $a^{\star}$ denote the all-strong, all-tier-1, and a class-targeted assignment; $\tS_1^{\min}$, $\tS_1^{\max}$ are the extreme tier-1 effective service times; dynamics and stage cost as in \eqref{eq:value} without the chance constraint.

\begin{pro}[Anticipation versus reaction, fluid model]\label{pro:anticipation}
\textbf{Hypotheses.}
(H1) $L^{\mathrm{calm}}(0)<mk<L^{\mathrm{surge}}(0)$.
(H2) $L^{\mathrm{calm}}(1)<mk<L^{\mathrm{surge}}(1)$.
(H3) With $N_{t_0}(x)=\lambda_x^{\mathrm{calm}}\,\tS_0(x)$ the tier-0 rest at the switch time $t_0<t_1$: $\sum_x\max\{\,N_{t_0}(x),\,\lambda_x^{\mathrm{surge}}\,\tS_{a^{\star}(x)}(x)\,\}<mk$.
(H4) Reactive rule: trigger $\wh N_{\mathrm{fire}}=mk$; release $\wh N_{\mathrm{rel}}<mk\,\tS_1^{\min}/\tS_1^{\max}$; $\tau:=\inf\{t\ge t_1: N_t=mk\}$ along the all-strong trajectory from rest; $t_{\mathrm{lag}}:=mk\,(\tS_1^{\max}-\tS_1^{\min})/(L^{\mathrm{surge}}(1)-mk)$; assume $\tau+t_{\mathrm{lag}}<t_2$. Set $u:=t_2-\tau-t_{\mathrm{lag}}>0$.
(H5) $\eta:=(L^{\mathrm{surge}}(1)-mk)/(mk-L^{\mathrm{calm}}(1))$. Horizon $T:=t_2+\eta u$. Anticipatory policy: $a^{\star}$ on $[t_0,T]$. Reactive policy: tier $0$ until $N$ reaches $\wh N_{\mathrm{fire}}$, then tier $1$ until $N$ falls below $\wh N_{\mathrm{rel}}$, then tier $0$.

\textbf{Define}
\[
b:=\frac{L^{\mathrm{surge}}(1)-mk}{\tS_1^{\max}},\quad
b':=\frac{mk-L^{\mathrm{calm}}(1)}{\tS_1^{\max}},\quad
\Phi:=\frac{C_{\mathrm{SLA}}\,b^{2}u^{3}\,(1+b/b')}{3\,(mk)^{2}},
\]
\[
\Psi':=k\mu_1 m\min_x\bigl[d_1(x)(1-\rho_x)(p_r\ell)_x\bigr],\quad
\Psi:=\Psi'(1+\eta)\,u,\quad
N^{\star}_x:=\max\bigl\{\,N_{t_0}(x),\,\lambda_x^{\mathrm{surge}}\,\tS_{a^{\star}(x)}(x)\,\bigr\},
\]
\[
D^{\mathcal{A}}:=\gamma\Bigl[k\,w_0+\sum_x N^{\star}_x\,\Delta w_{a^{\star}(x)}\Bigr]+\gamma\,c_m\Bigl(\sum_x N^{\star}_x\Bigr)^{\!\kappa}+\sum_x d_{a^{\star}(x)}(x)\,(1-\rho_x)\,\mu_{a^{\star}(x)}\,N^{\star}_x\,(p_r\ell)_x,
\]
and $K:=D^{\mathcal{A}}-\gamma k\,w_0$. Under (H1)--(H5):
\emph{(i)} the reactive policy remains on tier 1 throughout $[\tau,T]$, its backlog satisfies $N_t\ge mk\,\tS_1^{\min}/\tS_1^{\max}$ there, and its integrated cost over $[t_0,T]$ is at least $\gamma k\,w_0\,(T-t_0)+\Phi+\Psi$;
\emph{(ii)} the anticipatory trajectory satisfies $N_t<mk$ for all $t\in[t_0,T]$ and its integrated cost is at most $D^{\mathcal{A}}\,(T-t_0)$;
\emph{(iii)} if $\Phi+\Psi>K\,(T-t_0)$, the anticipatory policy strictly dominates;
\emph{(iv)} holding $t_0,t_1,\tau,t_{\mathrm{lag}}$ and all cost parameters fixed and varying the surge duration through $u$, define $g(u):=A\,u^{3}+\bigl(\Psi'(1+\eta)-C\bigr)u-B$ with $A=C_{\mathrm{SLA}}b^{2}(1+b/b')/(3(mk)^2)$, $B=K(\tau+t_{\mathrm{lag}}-t_0)$, and $C=K(1+\eta)$; condition (iii) reads $g(u)>0$; $g$ is strictly convex on $u\ge 0$ with $g(0)=-B<0$, so it has exactly one positive root $u^{\star}$, and $u>u^{\star}$ is a sufficient condition for anticipation to dominate.
\end{pro}
\begin{proof}
Appendix~\ref{app:anticipation}.
\end{proof}

\medskip\noindent\emph{Remark.} Compute $u^{\star}$ numerically as the unique positive root of $Au^{3}+\bigl(\Psi'(1+\eta)-C\bigr)u-B$; when $\Psi'(1+\eta)>C$ the churn accrued under full degradation already exceeds the anticipatory cost rate $K$ above the idle floor and $u^{\star}$ is small. The threshold is sufficient, from a lower bound on the reactive cost and an upper bound on the anticipatory cost; anticipation may dominate for shorter surges.

Propositions~\ref{pro:spiral} and~\ref{pro:anticipation} and Corollary~\ref{cor:latch} have one mechanism in common. In the fluid model, degradation changes the effective arrival rate on the same time scale as the backlog, so a reactive policy can increase the drift it is meant to reduce.

\subsection{Customer Heterogeneity}\label{sec:who}

In this subsection we show the value of side information distinguishing the users of the LLMs. Up to this point every customer feels degradation identically, which is not the case in practice: a researcher doing context-heavy work feels a single tier of degradation acutely; a user parsing CSV fields into a table does fine on a smaller model. Index customer classes by $x\in\calx$, a label observed by the router; for example $\calx=\{\text{researcher},\text{parser}\}$, two classes each aggregating customers whose workload fits the description. Let each primitive of Section~\ref{sec:model} carry the class: the dissatisfaction probability becomes $d_j(x)$, the retry probability $\rho_x$, the expected LTV loss $(p_r\ell)_x$. The sensitivity curve $d_j(x)$ is steep in $j$ for the first type and nearly flat for the second, and $\rho_x$ and $(p_r\ell)_x$ order the same way: the researcher re-asks more and is worth more. The effective service time and throughput of \eqref{eq:effective} become class-dependent, $\tS_j(x):=M_j(x)\,\ex[S_j]$ and $\theta_j(x):=mk/\tS_j(x)$. For the asymmetry, take a researcher at $d_1=0.6$, $\rho=0.9$ and a parser at $d_1=0.05$, $\rho=0.3$: the multipliers are $M_1=(1-0.54)^{-1}\approx 2.17$ against $(1-0.015)^{-1}\approx 1.02$. Routing the researcher to tier 1 more than doubles their traffic; routing the parser is nearly free.

We consider now a routing control: let $\varphi(x,j)$ be the fraction of class-$x$ traffic sent to tier $j$, and $\lambda_x$ the class-$x$ fresh-demand rate, the class decomposition $r_t=\sum_x\lambda_x$ of the forecast, counting first attempts only because $\tS_j(x)$ already carries each attempt's retries. With $c(x,j):=\gamma\,\tS_j(x)\,\Delta w_j+M_j(x)\,d_j(x)(1-\rho_x)\,(p_r\ell)_x$, the cost per initiated class-$x$ task at tier $j$ (the energy of \eqref{eq:ledger} plus its lifetime value loss, at class $x$'s own multiplier), the one-period problem is a transportation problem with classes as demand nodes and tiers as warehouses. The capacity constraint charges each unit of routed traffic its \emph{retry-inflated} load $\tS_j(x)$, not its posted service time; otherwise the optimization sees tier 1 as cheap capacity and selects the unstable regime of Example~\ref{ex:surge},
\begin{equation}\label{eq:transport}
\min_{\varphi\ge 0}\;\sum_{x,j} c(x,j)\,\lambda_x\,\varphi(x,j)
\quad\text{s.t.}\quad
\sum_j \varphi(x,j)=1 \;\;\forall x\in\calx,
\qquad
\sum_{x,\,j} \tS_j(x)\,\lambda_x\,\varphi(x,j) \;\le\; mk.
\end{equation}
The fleet is a single pool shared across tiers; dedicated per-tier budgets are the special case with the constraint replaced by $\sum_x \tS_j(x)\lambda_x\varphi(x,j) \le mk_j$. The objective sums, over classes and tiers, the dollar cost per initiated task times the volume routed there (electricity plus churn loss). The first constraint says every class is served somewhere. The second measures capacity in retry-inflated slot-time: each unit of class-$x$ traffic routed to tier $j$ occupies $\tS_j(x)$ slot-time, multiplier included, so a heavy-retry class consumes capacity its posted service time conceals, and in the trap region routing it to the weak tier \emph{tightens} the constraint the routing was meant to relax. The linear objective matches the nonlinear cost \eqref{eq:cost} term by term. The idle floor $\gamma k w_0$ is constant across routings and drops from the argmin. The service-level penalty becomes the capacity constraint, and its economics survive as the dual $\nu$. The memory term is convex in occupancy and is linearized at the operating point at marginal rate $\gamma c_m\kappa N^{\kappa-1}$ per unit slot-time, the discount of Remark~\ref{rem:memory}.

\begin{pro}[Index rule]\label{pro:index}
At an LP optimum of \eqref{eq:transport} with multiplier $\nu\ge0$ on the capacity constraint, each class $x$ is routed to a tier in
\[
\operatorname*{arg\,min}_j\;\bigl[\,c(x,j)\;+\;\nu\,\tS_j(x)\,\bigr],
\]
and there exists an optimal extreme-point routing with at most one class split across tiers. For $J=1$ and two classes with positive relief $\tS_0(x)-\tS_1(x)>0$, the class with the smaller index
\[
I(x,1) \;=\; \frac{c(x,1)-c(x,0)}{\;\tS_0(x)-\tS_1(x)\;},
\]
marginal damage per unit of capacity relief, is degraded first; for each class $x$, the effective slot-time $\tS_{j^{\star}(x)}(x)$ of the routed tier is weakly decreasing in $\nu$.
\end{pro}
\begin{proof}
Attach a multiplier $\nu\ge0$ to the capacity constraint of \eqref{eq:transport}. The Lagrangian $\sum_{x,j}\lambda_x\bigl[c(x,j)+\nu\,\tS_j(x)\bigr]\varphi(x,j)-\nu\,mk$ separates by class, so for fixed $\nu$ each class minimizes $\bigl[c(x,j)+\nu\,\tS_j(x)\bigr]$ over the simplex independently; LP duality characterizes every optimum at the optimal $\nu^{\star}$. The feasible polytope has one equality row per class plus the capacity row, so a vertex has at most $|\calx|+1$ positive variables; every class contributes at least one and each split class a second, so two split classes would need $|\calx|+2>|\calx|+1$: an optimum exists with at most one class split. For $J=1$ and two classes, degrading class $x$ buys relief $\lambda_x\bigl(\tS_0(x)-\tS_1(x)\bigr)$ at cost $\lambda_x\,I(x,1)\bigl(\tS_0(x)-\tS_1(x)\bigr)$; when the budget binds, moving a unit of relief from the higher-index class to the lower-index one changes the objective by the negative quantity in the proof of Proposition~\ref{pro:dominance}, so the smaller index is degraded first. The argmin is the lower envelope of the affine functions $c(x,j)+\nu\,\tS_j(x)$ in $\nu$; as $\nu$ grows the envelope selects tiers of smaller $\tS_j(x)$, so $\tS_{j^{\star}(x)}(x)$ is weakly decreasing in $\nu$.
\end{proof}
The index is a relative of the $c\mu$ rule \parencite{VanMieghem1995}, with the denominator in effective slot-time: degrading a heavy-retry class buys less capacity than its service-time discount suggests, because part of the freed capacity is reclaimed by that class's own retries; in the trap region of Proposition~\ref{pro:trap} the denominator is negative, and the class is degraded only if $c(x,0)-c(x,1)>\nu\,\bigl(\tS_1(x)-\tS_0(x)\bigr)$.

\begin{rem}\label{rem:indexplain}
When the fleet is congested, every slot freed by degrading someone is paid for in degraded service, and the exchange rate differs by customer. The index says: degrade first the customers for whom that swap is cheapest, whose retries eat the least of the capacity just freed and whose dissatisfaction and churn cost the least; heavy-retry, high-value customers are degraded last, and never when their relief is nonpositive unless the direct saving exceeds the capacity cost.
\end{rem}

\begin{pro}[Dominance]\label{pro:dominance}
Targeted routing weakly dominates uniform routing. Assume two tiers, $\lambda_x>0$ for all $x$, and two classes $x_1,x_2$ with $\tS_0(x)>\tS_1(x)$ and, after relabeling, $I(x_1,1)<I(x_2,1)$; and assume the uniform optimum has $\varphi^u(1)=\alpha\in(0,1)$. Then dominance is strict.
\end{pro}
\begin{proof}
Any uniform policy is feasible for \eqref{eq:transport} with $\varphi(x,j)$ constant in $x$, so the optimum weakly improves on it. For strictness: move $\epsilon>0$ of class-$x_2$ volume from tier 1 to tier 0 and $\epsilon'=\epsilon\,\bigl(\tS_0(x_2)-\tS_1(x_2)\bigr)/\bigl(\tS_0(x_1)-\tS_1(x_1)\bigr)$ of class-$x_1$ volume from tier 0 to tier 1. This is feasible for small $\epsilon$ because $\alpha\in(0,1)$ and $\lambda_x>0$; the capacity use is unchanged since $\epsilon\bigl(\tS_0(x_2)-\tS_1(x_2)\bigr)=\epsilon'\bigl(\tS_0(x_1)-\tS_1(x_1)\bigr)$; and the objective changes by
\[
\epsilon\,\bigl(\tS_0(x_2)-\tS_1(x_2)\bigr)\,\bigl[I(x_1,1)-I(x_2,1)\bigr]\;<\;0,
\]
using $c(x,1)-c(x,0)=I(x,1)\bigl(\tS_0(x)-\tS_1(x)\bigr)$. The uniform optimum is therefore not optimal, and dominance is strict.
\end{proof}

The uniform policy is the LP optimum restricted to routings with $\varphi(x,j)=\varphi(j)$, the same tier mix for every class; its gap to the true optimum is the value of re-sorting degraded volume across classes in increasing order of the index $I(x,1)$, positive under the hypotheses of Proposition~\ref{pro:dominance}. That gap is the value of class information.

\subsection{The Shadow Price of Intelligence}\label{sec:shadow}

The shadow price of intelligence is the Lagrange multiplier on the retry-inflated slot-time budget of \eqref{eq:transport}: the marginal cost of one additional unit of effective load, the price the system assigns to intelligence at the binding tier.

In the static problem each demand-satisfaction constraint carries a dual variable, and in the dynamic problem the corresponding object is a derivative through the value function,
\begin{equation}\label{eq:shadow}
\pi_x(t) \;:=\; \frac{\partial V^{\star}_t}{\partial \lambda_x(t)},
\end{equation}
the system's marginal cost of one more class-$x$ query at hour $t$. The value function behind \eqref{eq:shadow} is the dynamic version of \eqref{eq:transport}. The state is the class-indexed backlog $\bN_s=(N_s(x))_{x\in\calx}$ with total $N_s:=\sum_x N_s(x)$; an action is an assignment $a_s:\calx\to\{0,\dots,J\}$ of a tier to each class, and $\cala$ denotes the admissible assignment sequences $(a_t,\dots,a_T)$. Write $n_s(x)$ for the class-$x$ jobs \emph{in service}, below capacity and the proportional share above it, $n_s(x):=N_s(x)\min\{1,\,mk/N_s\}$. The stage cost is \eqref{eq:cost} with its active-draw term opened up by class, the multi-tier extension deferred in Section~\ref{sec:model},
\[
c\bigl(k,\bN_s,a_s\bigr) \;:=\; \gamma\Bigl[k\,w_0+\sum_{x\in\calx} n_s(x)\,\Delta w_{a_s(x)}\Bigr] \;+\;\gamma\,c_m N_s^{\kappa} \;+\;C_{\mathrm{SLA}}\!\left(\frac{(N_s-mk)^{+}}{mk}\right)^{\!\beta_{\mathrm{SLA}}},
\]
which reduces to \eqref{eq:cost} when every class shares one tier. Then
\begin{equation}\label{eq:value}
\begin{aligned}
V^{\star}_t \;=\; \min_{a(\cdot)\in\cala}\; &\sum_{s=t}^{T}\Bigl[\,c\bigl(k,\bN_s,a_s\bigr) \;+\;\sum_{x\in\calx} d_{a_s(x)}(x)\,(1{-}\rho_x)\,\mu_{a_s(x)}\,n_s(x)\,(p_r\ell)_x\Bigr] \\
\text{s.t.}\quad &N_{s+1}(x) \;=\; \mathcal{F}\bigl(N_s(x),\,\lambda_x(s),\,a_s(x)\bigr) \quad\forall x\in\calx,\\
&\pr\bigl(W_s>\bar s\bigr)\;\le\;\alpha,
\end{aligned}
\end{equation}
with $\mathcal{F}$ the two-regime recursion \eqref{eq:fluid} applied class by class, each class carrying its own effective quantities, dissatisfaction $d_{a_s(x)}(x)$, retry probability $\rho_x$, its capacity share as the regime boundary. The second sum is the churn cost: class-$x$ completions depart at rate $\mu_{a_s(x)}\,n_s(x)$, each an attempt that fails with probability $d_{a_s(x)}(x)$ and abandons with probability $1-\rho_x$, at the expected LTV loss $(p_r\ell)_x$. The response-time constraint is a chance constraint at threshold $\bar s$, enforced by pruning every state--action pair that violates the heuristic normal surrogate below; it permits transient saturation and charges for it through $C_{\mathrm{SLA}}$. The fresh-demand rate $\lambda_x(s)$ enters only the class-$x$ dynamics, so \eqref{eq:shadow} is well defined class by class; with a single class the assignment reduces to a tier-switching sequence, and \eqref{eq:value} is the problem of Section~\ref{sec:dynamics}. The forward pass makes $\pi_x(t)$ computable at every level: closed form at the newsvendor level, by differentiating \eqref{eq:newsvendor} through the chance constraint; the LP dual of \eqref{eq:transport} in the static problem; and, in the dynamic model, a finite difference through the deterministic forward pass, with no Monte Carlo jitter. The finite-difference objects are \emph{policy-conditioned sensitivities}: the perturbation evaluated under the optimal action sequence held fixed. They agree with \eqref{eq:shadow} when the sequence and its feasibility are unchanged by the perturbation; they may differ at action-switch hours. In the static problem \eqref{eq:transport} the demand-constraint dual satisfies
\[
\pi_x \;=\; \min_j \bigl\{\, c(x,j) + \nu\,\tS_j(x) \,\bigr\},
\]
a class-specific floor $c(x,j^{\star}(x))$ plus a scarcity rent $\nu\,\tS_{j^{\star}(x)}(x)$ that is nearly common across classes because $\nu$ is shared and the $\tS$'s are of similar order. Because the forward pass is closed form between regime switches, the dynamic sensitivity can be recomputed in close to real-time.

The chance constraint is enforced through a \emph{heuristic} normal surrogate for the per-attempt response time. With $q(N):=(N-mk)^{+}$ and $\mathbf{1}_{\mathrm{sat}}:=\ind\{N\ge mk\}$, an arrival finding $N$ resident jobs queues behind $q(N)$ jobs served at aggregate rate $k\mu_j m$, then occupies a slot for one exponential service at rate $\mu_j$; the normal approximation to this Erlang-plus-service response time gives the pruning rule
\[
\bar s \;\ge\; \frac{q(N) + \mathbf{1}_{\mathrm{sat}}}{k\mu_j m} \;+\; \frac{1}{\mu_j} \;+\; \Phi^{-1}(1-\alpha)\,\sqrt{\frac{q(N) + \mathbf{1}_{\mathrm{sat}}}{(k\mu_j m)^2} + \frac{1}{\mu_j^2}},
\]
pruning retains only state--action pairs that pass the surrogate; Section~\ref{sec:poc} re-estimates latency by simulation because the surrogate is a heuristic.

The structure of $\pi_x(t)$ has two parts. First, the class duals differ dynamically, not only at peak load: even with the capacity constraint slack, a marginal specialized query costs its strong tier's multiplied compute and churn, $M_j(x)\bigl(c_j + d_j(x)(1{-}\rho_x)(p_r\ell)_x\bigr)$, while a marginal casual query costs its cheap tier's, so the relative price a uniform policy implies is wrong at every hour. Second, when capacity binds, the shared pool adds a scarcity rent nearly common across classes, because congestion anywhere raises the marginal cost of capacity: dollar differences between the class duals widen at peak load while their \emph{ratios} compress toward one.

\section{Computational Considerations}\label{sec:practice}

In this section we compute the dynamic control problem without simulation, using a backward pass of the value function that exploits the closed form of the transitions.

\subsection{Complexity Analysis of the Closed Form}\label{sec:dp}

The value function \eqref{eq:value} is solved by backward induction: $T$ epochs, a grid of $g$ points per class backlog with multilinear interpolation, and the admissible assignments $\cala_1$ from \eqref{eq:ledger}, which prunes unfeasible states (Section~\ref{sec:poc-instances} reports reductions as large as \spiOpenAIActionsFull$\to$\spiOpenAIActionsLedger$\to$\spiOpenAIActions). In general, solving this problem involves this count of transitions,
\[
T \;\times\; |\cala_1| \;\times\; g^{|\calx|},
\]
so the computational difference between algorithms is the cost of one transition, which Lemma~\ref{lem:fluid} reduces to closed-form arithmetic.

Naive transitions cost $\Theta(\mu_{\max}\,\Delta)$ substeps per transition against $\Theta(1)$ for our closed form, a speedup linear in the service-rate scale; Table~\ref{tab:poc-cert} reports the measured factors and Appendix~\ref{app:transition} shows the complete derivation.

The same saving applies when estimating the dual: after the backward pass, a forward pass costs $T$ transitions, so the finite-difference sensitivity \eqref{eq:shadow} computes the full daily surface $\pi_x(t)$ in $O(|\calx|\,T^2)$ transitions of closed-form arithmetic without Monte Carlo or averaging.

The single-class dynamics of \eqref{eq:fluid} requires no time discretization: between changes of the arrival rate or the tier, the trajectory is a formula, and the only event is the capacity crossing, itself closed form. Approximation enters solely through the multi-class extension, where a shared pool makes each class's capacity a function of every class's current backlog; the solver holds these shares frozen within a leg, re-resolving at the leg boundaries, so the frozen-share transition is approximate, with error first-order in the leg length. It concentrates in deep saturation with unbalanced backlogs, off the optimal trajectory; the DES-estimated policy-value comparison of Table~\ref{tab:poc-cert} spans $-12.2\%$ (OpenAI) to $+2.1\%$ (Google) across the five instances.

\subsection{Variance}\label{sec:variance}

The fluid model \eqref{eq:fluid} discards the fluctuations of the underlying stochastic process. The fluid model is reliable for Propositions~\ref{pro:trap}, \ref{pro:index}, and~\ref{pro:dominance}, which are statements about orderings and regions; for the saturated regime, which is drift-dominated; and for the churn cost, which aggregates over many customers and obeys the law of large numbers. It is not reliable in the tail. The service-level constraint $\pr(W>s)\le\alpha$ is a statement about fluctuations, and a fluid path can only answer zero or one. Because idle servers are a pure cost, the provider operates close to the stability boundary, where the normal surrogate is least accurate.

The retry cascade is a branching process, so the effective arrival stream is overdispersed, its variance growing without bound as $d_j\rho \to 1$. The attempt count $A$ per initiated task is geometric with $\ex[A] = M_j$ and $\var(A) = q_j M_j^2$, $q_j := d_j\rho$, so the offered slot-time per fresh arrival, $W = \sum_{i=1}^{A} S_{j,i}$, is a compound sum with $\var(W) = \ex[A]\,\var(S_j) + \var(A)\,\ex[S_j]^2$; for exponential services this is $\var(W) = \tS_j^{\,2}$, an $M_j^2$ inflation over the posted $\ex[S_j]^2$ against the $M_j$ inflation of the mean. A supply network propagates variability through a multiplier of its own, distinct from and larger than the mean multiplier \parencite{LeePadmanabhanWhang1997, ChenDreznerRyanSimchiLevi2000}. For a fixed tier with $r<\theta_j$ the stochastic queue is stable; fluctuations produce transient excursions above capacity whose cost the fluid model omits.
We conjecture that noise \emph{hurts} the reactive policy, which parks the system at the stability boundary, and \emph{helps} the anticipatory one, which buys distance from it.

With delayed retries the arrival rate becomes a functional of the current population state: the dynamics are distribution-dependent, a nonlinear Markov chain in the sense of \textcite{DiekerHackmanWangYan2026}, the principled extension of this work. The fluid model here is a first-order approximation: it locates the trap, the instability threshold, and the index rule. Exact tail analysis in the critical band, heavy-tailed token distributions, and retries arriving with delay are left for future extensions of this methodology.

\subsection{Distilled decision trees}\label{sec:tree}

The optimal policy of Section~\ref{sec:analysis} is a heatmap over $(t,N,x)$, hard for practitioners to implement; a uniform threshold policy, degrade everyone once $N>\wh N$, is a depth-one stump, constant in $x$. We fit a decision tree to the solved policy. Tree methods for Markov decision processes interleave splitting and optimization because the policy is unknown during construction \parencite{SanabriaYaoLam2021}, whereas here the teacher is already solved. Two steps: run the fluid dynamic program for the optimal action and the occupancy measure $\mu(t,N)$ at every grid state; fit a classification tree to the optimal actions by occupancy-weighted CART. The tree is then rolled out as a policy and its cost and latency are estimated by simulation (Section~\ref{sec:poc-trees}, Table~\ref{tab:des-cis}). The fitting objective is the occupancy-weighted value regret
\begin{equation}\label{eq:regret}
\sum_{t,N,x} \mu(t,N)\,\bigl[Q(t,N,x,\mathrm{tree}(t,N,x))-Q(t,N,x,a^{\star})\bigr],
\end{equation}
nonnegative by definition of $a^{\star}$. Each leaf is a row, ``demand surge, class is parser $\to$ tier 1,'' that a deployment engineer can easily implement.

\subsection{Data requirements}\label{sec:data}

Every input to the model is available to the fleet controller: service times, power draws, serving traces, capacities, and demand. The dissatisfaction probability $d_j(x)$, the retry probability $\rho_x$, and the expected LTV loss $(p_r\ell)_x$ run on the serving logs the provider keeps. Tier assignments are recorded per query; sessions link queries to customers; dissatisfaction is estimated empirically: a re-ask of the same question within minutes, semantically near its predecessor, an explicit regeneration, a thumbs-down, an agentic run abandoned mid-task. The per-answer retry rate at tier $j$ estimates the product $d_j(x)\rho_x$ directly.

\section{Numerical examples}\label{sec:poc}

We conduct numerical experiments showing the full pipeline. The tier physics of each provider, mean service times, quality indices, and cost per task, are read from the Artificial Analysis leaderboard; DeepSeek's concurrency limits, peak windows, and two-to-one peak/off-peak price ratio are taken from its published API documentation; the behavioral primitives $(\rho_x, p_r\ell_x)$ are stylized parameters, as discussed in Section~\ref{sec:data}. The five providers span five distinct regimes of the theory: a steep reasoning-effort ladder (Anthropic), a trap tier alongside a published clock policy (DeepSeek), a fast ladder where degradation buys throughput and little energy (Google), a flat ladder where degradation buys nothing (Kimi), and a fourteen-variant lattice with equal-quality tiers split across the energy and capacity frontiers (OpenAI).

\subsection{Instances, LLM menus, and their comparison}\label{sec:poc-instances}

The parameters of this section are collected in Table~\ref{tab:poc-params} (Appendix~\ref{app:calibration}). We use three customer classes, shared across all five providers, carrying the heterogeneity of Section~\ref{sec:who}: casual (low retry probability, low lifetime value), specialized (high retry, high value), and agentic (retries by construction), with dissatisfaction $d_j(x)$ linear in the relative quality gap of each ladder. Demand follows a shared diurnal profile whose peak sits at $0.90$ of the static-optimal-mix throughput $\theta_{\mathrm{mix}}$.

We first report the models that for each class dominate the others in the \eqref{eq:ledger} sense, weakly dominating in both effective slot-time $\tS_j(x)$ and per-initiated-task total cost; the final list is the last column of Table~\ref{tab:poc-params}. Anthropic's \spiAnthropicActionsFull{} joint actions reduce to \spiAnthropicActions: the agent class is pinned to the single admissible tier Opus~5 (medium), excluded from the strong tier by slot-time and from the weak tiers by its own retry multiplier; we discard the others by the static comparison \eqref{eq:ledger}. DeepSeek's \spiDeepSeekActionsFull{} actions reduce to \spiDeepSeekActions: the trap tier identified by Proposition~\ref{pro:trap} is inadmissible for every class, and the \emph{strong} tier is inadmissible for the majority (casual) class, whose entire effective menu is the Flash sibling. OpenAI's \spiOpenAIActionsFull{} joint actions reduce to \spiOpenAIActionsLedger{} by the static cost comparison, and we then remove the $221$-second tiers Sol~(max) and Terra~(max) and the $172$-second Luna~(max) by the per-attempt screen of Section~\ref{sec:poc-protocol}, for a post-screen count of \spiOpenAIActions{} (Table~\ref{tab:poc-params}): the screened tiers appear only in the always-strong baseline, never in the DP's menu. Kimi's menu reduces to a single action, all classes on the strong tier: the retry multiplier makes every degradation a pure loss. The trap interval of Proposition~\ref{pro:trap}, evaluated class by class and tier by tier on the calibrated menus with the benchmark price per task as the energy proxy ($\Delta w_j\propto c_j/\ex[S_j]$, a ratio identification), places fourteen of the seventy-two tier--class pairs inside the trap: DeepSeek's V4-Pro-0424 (high) for all three classes, Google's Gemini~3.6 Flash for the casual class, both Kimi weak tiers for all three classes, and OpenAI's Terra~(max) and Luna~(max) where admissible.

\subsection{Policies and protocol}\label{sec:poc-protocol}

We compare six policies, each starting from the resting level of its own hour-0 action. P0 serves every query at the strong tier. P1 is the reactive stump: degrade everyone when total backlog crosses a trigger, release below a second threshold, tier and thresholds tuned per instance by rollout search. P2 is the clock policy modeled on DeepSeek's published pricing clock: degrade everyone during the published peak windows regardless of state, tier tuned. P5 is the index rule (hourly): the LP \eqref{eq:transport} is solved each hour for its dual $\nu$, and each class is assigned the single tier in $\operatorname*{arg\,min}_j\,[c(x,j)+\nu\tS_j(x)]$ (Proposition~\ref{pro:index}); this rounding may violate capacity when the LP splits a class. P3 is the fluid dynamic program of Section~\ref{sec:analysis}: hourly epochs, the closed-form fluid transition of Lemma~\ref{lem:fluid} with the arrival rate made endogenous by retries, a $13^3$ interpolated grid over the three class backlogs, and the per-attempt chance constraint $\pr(W_{\mathrm{attempt}}>300\,\mathrm{s})\le 0.05$ enforced by two screens: any fixed-tier action whose mean service time exceeds $300/\ln 20 = 100.14$ seconds is rejected at every state (zero-backlog service time lower-bounds per-attempt latency, and $\pr(\mathrm{Exp}(\mu)>300\,\mathrm{s})>0.05$ exactly when $\ex[S_j]>100.14$ s); and the heuristic normal surrogate of Section~\ref{sec:shadow}, pruning state--action pairs. P4 is the depth-three tree distilled from P3 by the occupancy-weighted procedure of Section~\ref{sec:tree}, then rolled out as a policy in its own right.

Latency is estimated by a time-stepped simulation with exact conditional draws, which we call DES (Appendix~\ref{app:des}): 30 replications per policy and provider, per-class Poisson fresh-arrival streams, exponential services, instantaneous retries, and the fluid model's proportional capacity shares (per-tier pools where published). Appendix~\ref{app:des} gives the implementation and, per policy, the number of hours in which the point estimate exceeds $0.05$ (Table~\ref{tab:des-cis}). A policy is \emph{infeasible} if its DES-estimated $\pr(W_{\mathrm{attempt}}>300\,\mathrm{s})$ exceeds $0.05$ in any hour.

\input{Figures/latex_out/01_table_results}

\subsection{Results}\label{sec:poc-results}

Table~\ref{tab:poc-results} reports the twenty-four-hour DES-estimated economics (30 replications per policy and provider; per-cell 95\% CIs in Appendix~\ref{app:des}); Appendix~\ref{app:trajectories}, Figure~\ref{fig:poc-traj}, the backlog trajectories; Table~\ref{tab:poc-sens} (Appendix~\ref{app:onesens}), the one-at-a-time sensitivity. We note four patterns. First, under capacity provisioned to the optimal mix, not optimizing is costly: the always-strong policy is unstable on Anthropic and DeepSeek, its backlog diverging under its own retry feedback, and where it survives it does so at up to nearly three orders of magnitude above the optimum (Google) or in violation of the per-attempt latency constraint (OpenAI, DES-estimated). Second, the two heuristic baselines fail in characteristic modes. The clock policy is unstable on Anthropic and DeepSeek and on Kimi, where its scheduled degradation steps into a flat ladder and generates retry-induced load; the sawtooth of Appendix~\ref{app:trajectories}, Figure~\ref{fig:poc-traj}, is the fire-and-release cycling of Corollary~\ref{cor:latch}. The tuned stump is DES-feasible only on Kimi (Table~\ref{tab:poc-results}) and is infeasible everywhere else, with a higher churn cost: on Anthropic it costs \spiAnthropicStumpEcon{} (infeasible, \spiAnthropicStumpVh{} violation hours) against the DP's \spiAnthropicDPcost, with DES-estimated churned lifetime value of \spiAnthropicStumpChurn{} against the DP's \spiAnthropicChurn{} on identical demand. Third, the index rule (hourly) of \eqref{eq:transport}, re-solved hourly against the forecast, matches the DP at parity on Kimi and DeepSeek (\spiKimiLPratio{} and \spiDeepSeekLPratio{} the DP), is much worse than the DP on the surge-driven instances (Anthropic at \spiAnthropicLPratio{} the DP, Google at \spiGoogleLPratio, and infeasible besides), and on OpenAI, where the screen removes the strong tiers and the remaining menu is nearly flat, it edges the grid DP at \spiOpenAILPratio{}: under simulation the Euler-solved policy costs $\$1{,}586{,}548$, matching the index rule; the fluid solver's higher simulated cost is consistent with frozen-share error (Section~\ref{sec:dp}). Two effects contribute to this gap and we do not separate them: the rule has no lookahead, and rounding a split class to one tier can violate capacity (Section~\ref{sec:poc-protocol}), which the quadratic service-level term charges. Fourth, under the DES feasibility criterion of Table~\ref{tab:des-cis} (infeasible iff the DES-estimated hourly $\pr(W_{\mathrm{attempt}}>300\,\mathrm{s})$ exceeds $0.05$ in any hour), no tuned baseline survives on Anthropic or Google, so the best/DP column reports NA there; on OpenAI and DeepSeek the index rule (hourly) is feasible and matches or beats the DP (\spiOpenAILPratio{} and \spiDeepSeekLPratio{}), as does always-strong on Kimi (parity, $1.0\times$). The cheapest \emph{infeasible} baseline is still expensive: \spiAnthropicBestInfeasRatio{} the DP on Anthropic and \spiGoogleBestInfeasRatio{} on Google, before the churn cost of their violations is added. On Kimi no policy improves on always-strong, as Proposition~\ref{pro:trap} implies for a menu in which every weak tier lies in the trap region. These patterns are consistent with Propositions~\ref{pro:trap}, \ref{pro:index} and~\ref{pro:dominance} and Corollary~\ref{cor:latch}.

The dollar magnitudes in Table~\ref{tab:poc-results} are dominated by the service-level term of \eqref{eq:cost}, whose scale $C_{\mathrm{SLA}}$ is a calibration choice. Table~\ref{tab:poc-sens} shows the index rule/DP ratio between $0.98\times$ and $1{,}569\times$ across values of $C_{\mathrm{SLA}}$ and the demand peak. The ordering of the policies is the same in every cell; the ratios change.

The Google instance isolates the anticipation mechanism of Proposition~\ref{pro:anticipation}. Its static-optimal mix is all-strong, so off the surge the DP and always-strong agree; the surge exceeds the strong tier's mixed throughput and is survivable only down-ladder, so degradation here buys throughput; the energy saving is small. The DP degrades before the 03:00 surge and stays below capacity; the baselines reach backlogs of hundreds of thousands of jobs and service-level charges in the billions. The same mechanism, run in reverse, appears on Kimi: because its maximum-throughput action \emph{is} the all-strong action, the surge is costly for any policy caught degraded.

Appendix~\ref{app:onesens}, Table~\ref{tab:poc-sens} reports the one-at-a-time sensitivity and Table~\ref{tab:sens-full} the full grid. Four facts summarize the table. First, the DP is DES-feasible in every deterministic perturbation, and under forecast noise in 18 of 20 seeds at both $\sigma$ on Anthropic and 20 of 20 on Google. Second, exactly one cell in the grid has a DES-feasible baseline: Google at the $0.8\,\theta_{\mathrm{mix}}$ peak, where always-strong costs $1.01\times$ the DP. Third, in the two $0.8\,\theta_{\mathrm{mix}}$ cells the index rule (hourly) undercuts the DP in cost ($0.98\times$ on Google, $0.99\times$ on Anthropic) but is DES-infeasible, so it stays out of the best-baseline column. Fourth, the index rule/DP cost ratio ranges from $0.98\times$ (Google, $0.8\,\theta_{\mathrm{mix}}$) to $1{,}569\times$ (Google, $C_{\mathrm{SLA}}\times 2$), largest where the surge binds capacity.

\subsection{Computed sensitivities}\label{sec:poc-shadow}

\input{Figures/latex_out/08_fig_shadow}

Figure~\ref{fig:shadow} gives the marginal query impact by class and by time on all five instances: policy-conditioned sensitivities, finite differences of the twenty-four-hour cost through the closed-form fluid forward pass under the DP's own actions, recomputed in milliseconds. They agree with the optimized-value derivative \eqref{eq:shadow} away from action-switch hours, where they may differ; the kinks at the dotted verticals mark those hours. The two-part structure of Section~\ref{sec:shadow} is visible on every panel: the floors differ at every hour (off peak the specialized sensitivity is \spiAnthropicDualBase{} the casual one on Anthropic), and the rent moves together across classes when load rises. On Google the surge dominates: the marginal \emph{casual} query at 03:00 has a sensitivity of \spiGoogleRentWall{}, almost all of it the service-level term, four orders of magnitude above its off-peak floor. On Anthropic the rent peaks at the window shoulders rather than inside the windows: mid-window the DP has already degraded and capacity is cheap, while at the shoulders it is still running strong tiers into rising load. Three caveats apply. Sensitivities in the last hours are truncated by the finite horizon. The rent reaches the latency constraint through the smooth $C_{\mathrm{SLA}}$ term, the cost used in every rollout in this section. At the capacity boundary the marginal sensitivity is two-sided: adding one query carries the full service-level charge, removing one relieves it, so the add-a-query and remove-a-query estimates bracket the printed value, \spiGoogleRentWall{} being the high side; away from the boundary and from action switches the two agree to three decimals. DeepSeek and Kimi hold the static-optimal action all day, so their sensitivities have no switch structure and their rents follow load.

\subsection{Solver comparison}\label{sec:poc-cert}

Table~\ref{tab:poc-cert} compares the fluid solution with a brute-force alternative: the identical dynamic program, grid, action set, and cost function, with the closed-form fluid transition of Lemma~\ref{lem:fluid} replaced by a converged Euler integration at the per-provider stable step. On the five instances the closed form is 74 to 321 times faster; the simulated cost of the resulting policies differs by between $-12.2\%$ and $+2.1\%$. The solve takes seconds and the sensitivity surface milliseconds, so both can be recomputed during operation. The accuracy comparison is DES-estimated: the table's last column is the relative difference between the two greedy policies' twenty-four-hour costs under the 30-replication simulation. The trajectories carry a second comparison: the closed-form fluid forward pass of \eqref{eq:fluid} matches the simulation's mean hourly backlogs under the DP's own policy to a maximum relative gap of \spiTrajCheckMax{} (mean \spiTrajCheckMeanDP) over every provider and hour; the largest single gap is the surge hour on Google, where the policy sits at the capacity boundary. The trajectories of Appendix~\ref{app:trajectories} (Figure~\ref{fig:poc-traj}) are drawn from that forward pass. The threshold-hysteretic baselines shift firing times between the fluid path and the simulation, timing effects reported in Appendix~\ref{app:des}. The pointwise Q-regret statistics, reported in Appendix~\ref{app:trees}, score the fluid policy's actions state by state under the converged Q-function; their tails concentrate off-trajectory in deep-saturation grid states on the shared-pool, high-rate providers, where the frozen-share transition is weakest.

\input{Figures/latex_out/02_table_certificates}

\subsection{Distilled decision trees (numerical)}\label{sec:poc-trees}

Figures~\ref{fig:trees-anthropic} and~\ref{fig:trees-openai} (Appendix~\ref{app:openaitrees}) show the distilled policy, one depth-three tree per class, rendered in the same form as the baselines; the stump is a single split on total backlog, the clock a single split on the peak flag, always-strong a single leaf, and the DP's tree shows which splits they are missing. Every tree has at most eight leaves, the rollout value gaps to the DP are \spiAnthropicTreeGap{} and \spiOpenAITreeGap{} on the instances shown and \spiGoogleTreeGap{} on Google with non-overlapping DES confidence intervals (the tree smooths the grid-approximate teacher and improves on it; the DP is not the exact optimum). The energy-frontier tier Luna~(max), \$0.05 and $172$ seconds, is statically admissible on the casual menu, and the 221-second flagship Sol~(max) on the specialist menu; the static cost comparison reduces OpenAI from $2744$ joint actions to $72$, and we then remove the $221$-second and $172$-second tiers from the DP's menu entirely by the per-attempt screen of Section~\ref{sec:poc-protocol} (post-screen count \spiOpenAIActions): with mean service above $300/\ln 20 = 100.14$ seconds, $\pr(W_{\mathrm{attempt}}>300\,\mathrm{s})>0.05$ even at zero backlog, so the tier is rejected at every state and the energy saving is not collectible at any backlog. The screened tiers survive only in the always-strong baseline (and hence in the OpenAI trees' model strip, Appendix~\ref{app:openaitrees}, Figure~\ref{fig:trees-openai}, marked as baseline-only). The casual tree therefore splits only on capacity variables. The Anthropic casual tree splits at its root on the demand forecast $\hat r_t$, degrading ahead of load; the specialist tree splits on the peak flag, off-peak at moderate effort and peak hours governed by the two-to-one electricity price; the agent tree is a single leaf, a class whose effective menu is a singleton. The Google, DeepSeek, and Kimi trees are in Appendix~\ref{app:trees}; the DeepSeek specialist and agent trees split on the agentic and casual backlogs, cross-class repricing through the shared per-tier pools, and all three Kimi trees are single leaves with perfect fidelity.

In the trees, $\hat r_t$ is the demand forecast, $t$ the hour of day, $N_{\mathrm{cas}}, N_{\mathrm{spec}}, N_{\mathrm{ag}}, N_{\mathrm{tot}}$ backlogs by class and total, and $C$ pool capacity.

\FloatBarrier

\section{Concluding Remarks}\label{sec:conclusion}
In this paper we modeled LLM inference serving as a control problem: service quality is the provider's decision, retries and abandonment the demand's response. Folding retries into the service rate and using memorylessness gives a two-regime fluid recursion in closed form, which makes the dynamic program over quality tiers solvable without simulation. With heterogeneous customers the one-period problem is a transportation problem whose capacity constraint charges each class its retry-inflated load, and whose dual is the marginal cost of a query by class and time. The data needed are service times, power draws, serving traces, and retry rates by tier and class, which a provider logs.

The fluid model is a first-order approximation. Delayed retries make the arrival rate a functional of the population state, a nonlinear Markov chain in the sense of \textcite{DiekerHackmanWangYan2026}; that formulation is the natural setting for tail analysis in the critical band and for heavy-tailed token distributions, and we leave it for future work.
\printbibliography

\appendix
\FloatBarrier
\section{Supplementary material}\label{app:supp}

\FloatBarrier
\subsection{The Anthropic and OpenAI distilled trees}\label{app:openaitrees}

Figures~\ref{fig:trees-anthropic} and~\ref{fig:trees-openai} render the Anthropic and OpenAI distilled trees in the format of Section~\ref{sec:poc-trees}.
\input{Figures/latex_out/03_trees_anthropic}
\input{Figures/latex_out/03_trees_openai}

\FloatBarrier
\subsection{The transition oracle: explicit integration versus the closed form}\label{app:transition}
Without the lemma, the natural transition oracle is the drift itself: the backlog changes at rate $r^{\mathrm{eff}}_t$ minus the completion rate, stepped forward as the Markov chain or as the fluid ODE. Any explicit scheme must resolve the fastest relaxation, the below-capacity decay at rate $(1-d_j\rho)\mu_j$: stability and accuracy require steps $\delta$ with $(1-d_j\rho)\mu_j\,\delta\le c$, $c\approx 0.2$ in practice, hence $(\max_j\,\mu^{\mathrm{eff}}_j)\,\Delta/c \;\le\; \mu_{\max}\Delta/c$ substeps per epoch of length $\Delta$. The service rate sets the integrator's clock: the fast, cheap tiers that make degradation attractive make integrating it expensive. Lemma~\ref{lem:fluid} is the escape: because the drift is piecewise linear in the state, the fluid ODE solves in closed form on each side of the capacity boundary, the crossing time is itself a formula, and within a leg the trajectory crosses at most once, so one transition costs a constant handful of elementary-function evaluations regardless of $\mu_{\max}$ or $\Delta$. Per transition:
\[
\underbrace{\Theta\!\bigl(\mu_{\max}\,\Delta\bigr)}_{\text{Euler}}
\qquad\text{against}\qquad
\underbrace{\Theta(1)}_{\text{closed form}},
\]
a speedup linear in the service-rate scale, the pattern of Table~\ref{tab:poc-cert}: the measured factors grow with $\mu_{\max}$ and peak on the fastest ladder, sitting below the raw substep ratio only because interpolation and cost evaluation are overhead common to both solvers. The same saving survives the rate correction: one exponential and one logarithm per leg at $\Theta(1)$, independent of the rate scale. Tiers near the stability boundary, $\mu^{\mathrm{eff}}_j\ll\mu_j$, make Euler cheaper while the closed form's cost is unchanged, so the measured speedups shift down but remain linear in the rate scale. Figure~\ref{fig:dp-cost} draws one transition under each solver.
\begin{figure}[!htbp]
\centering
\begin{tikzpicture}[font=\scriptsize, >={Stealth}]
\foreach \xshift/\plabel/\ptitle in {0/(a)/{converged Euler},
                                     5.4/(b)/{closed form (Lemma~\ref{lem:fluid})}}{
\begin{scope}[shift={(\xshift,0)}]
  \fill[blue!10]  (0.9,0)  rectangle (1.35,2.25);
  \fill[green!8]  (1.35,0) rectangle (1.8,2.25);
  \draw[black!25] (0,0) grid[step=0.45] (2.7,2.25);
  \draw[->, black!60] (2.7,2.55) -- (0,2.55);
  \node[above, black!60] at (1.35,2.55) {backward sweep, $T$ epochs};
  \node[below] at (1.125,-0.02) {$s$};
  \node[below] at (1.575,-0.02) {$s{+}1$};
  \node[rotate=90] at (-0.3,1.125) {backlog grid, $g^{|\calx|}$};
  \draw[->, very thick, red!70!black] (1.125,1.575) -- (1.575,0.675);
  \draw[dashed, black!40] (1.125,1.575) -- (0.05,-0.55);
  \draw[dashed, black!40] (1.575,0.675) -- (2.75,-0.55);
  \draw[black!50, rounded corners=2pt] (0.05,-3.1) rectangle (2.75,-0.55);
  \begin{scope}[shift={(0,-2.95)}]
    \draw[->, black!60] (0.2,0.12) -- (2.6,0.12);
    \node[below, inner sep=2pt] at (1.4,0.12) {$t\ \to\ t+\Delta$};
    \draw[dashed, black!50] (0.2,1.0) -- (2.35,1.0)
          node[right, black!60, inner sep=1pt] {$mk$};
  \end{scope}
  \node at (1.4,-3.45) {\plabel\ \ptitle};
\end{scope}
}
\begin{scope}[shift={(0,-2.95)}]
  \draw[black!20] (0.3,1.9) -- (1.06,1.0)
        plot[domain=1.06:2.35, samples=30] (\x,{0.45+0.55*exp(-3.3*(\x-1.06))});
  \foreach \i in {0,...,5}{
    \pgfmathsetmacro\px{0.3+\i*0.152}
    \pgfmathsetmacro\py{1.9-\i*0.18}
    \fill[red!70!black] (\px,\py) circle (1.1pt);}
  \foreach \i in {6,...,14}{
    \pgfmathsetmacro\px{0.3+\i*0.146}
    \pgfmathsetmacro\py{0.45+0.55*exp(-3.3*(\px-1.06))}
    \fill[red!70!black] (\px,\py) circle (1.1pt);}
  \node[red!70!black, anchor=west, align=left] at (1.3,1.75)
    {$(\max_j\mu^{\mathrm{eff}}_j)\Delta/c$ steps,\\ one evaluation each};
\end{scope}
\begin{scope}[shift={(5.4,-2.95)}]
  \draw[very thick, blue!55!black] (0.3,1.9) -- (1.06,1.0);
  \draw[very thick, blue!55!black]
        plot[domain=1.06:2.35, samples=30] (\x,{0.45+0.55*exp(-3.3*(\x-1.06))});
  \fill[blue!55!black] (1.06,1.0) circle (1.6pt);
  \node[blue!55!black, below left, inner sep=1.5pt] at (1.1,0.97) {$t^{\star}$};
  \node[blue!55!black, anchor=west, align=left] at (1.3,1.75)
    {two formulas,\\ one crossing time};
\end{scope}
\end{tikzpicture}
\caption{One transition of the backward induction, magnified. The induction fills the same $T\times|\cala_1|\times g^{|\calx|}$ value surface under either solver; the panels differ inside the transition arrow. (a) An explicit integrator resolves the recovering branch's relaxation at rate $\mu^{\mathrm{eff}}_j$, so each transition takes $(\max_j\mu^{\mathrm{eff}}_j)\Delta/c$ substeps. (b) Lemma~\ref{lem:fluid} gives the same trajectory as a linear leg and an exponential leg joined at the closed-form crossing time $t^{\star}$: constant cost per transition, independent of the service-rate scale.}
\label{fig:dp-cost}
\end{figure}

\FloatBarrier
\subsection{Finite-horizon backlog-growth risk}\label{sec:sweep}

The fluid stability boundary $q^{\star}_j$ of Corollary~\ref{cor:slowing} is a statement about the fluid model; the simulation of this section quantifies what it does not say. Single class, Example~\ref{ex:surge} parameters, $C=mk=3000$ slots, $N(0)=7C/6$, tier 1 from $t=0$, horizon 30 seconds: the outcome is whether the backlog has grown past $17C/6$ ($8500$ jobs) by the horizon, reported as the \emph{finite-horizon backlog-growth probability} over 200 replications per point, with $d_1$ swept from $0.30$ to $0.50$ at spacing $0.02$, refined to spacing $0.0005$ within $\pm 0.02$ of the fluid threshold $d^{\star}=(1-r/(k\mu_1 m))/\rho=0.4175$, and repeated for $C\in\{750, 3000, 12000\}$ with $r$ scaled proportionally.

\input{Figures/latex_out/11_fig_sweep}
 The fluid model itself predicts growth of $(r-\theta_1)\times 30\,\mathrm{s}$ over the horizon, which exceeds $5C/3$ only when $d_1 > d^{\mathrm{fluid}} := \bigl(1-(r-C/18)/(k\mu_1 m)\bigr)/\rho = 0.4217$ for every $C$: at $d_1 = 0.42$ the fluid model's growth is positive but below the $5C/3$ cutoff, and the observed growth probabilities there, $0.285$ $[0.227, 0.351]$ at $C=750$, $0.055$ $[0.031, 0.096]$ at $C=3000$, and $0.000$ $[0.000, 0.019]$ at $C=12000$, are fluctuation-driven growth that shrinks with $C$. This is a finite-horizon growth frequency: the outcome requires absolute growth of $5C/3$ within a fixed horizon, and no irreversibility is claimed or implied.

\FloatBarrier
\subsection{Notation}\label{app:notation}

Table~\ref{tab:notation} collects the notation of Sections~\ref{sec:model}--\ref{sec:analysis}.
\begin{table}[!htbp]
\centering
\caption{Notation. The symbols of the first three blocks are primitives, measurable from benchmarks, serving traces, and logged traffic; the derived quantities of the last block carry the retry feedback and are the objects every result compares.}
\label{tab:notation}
\small
\begin{tabular}{@{}llL{8.2cm}@{}}
\toprule
& symbol & meaning \\
\midrule
\multirow{3}{*}{\rotatebox{90}{fleet}}
& $k$, $m$, $mk$ & servers; concurrent slots per server; total capacity in jobs \\
& $N_t$ & jobs resident in the system at time $t$ \\
& $\gamma$, $w_0$, $c_m$, $\kappa$, $C_{\mathrm{SLA}}$, $\beta_{\mathrm{SLA}}$ & electricity price; idle draw per server; memory cost scale and exponent; service-level penalty scale (dollars per hour) and elasticity \\
\midrule
\multirow{3}{*}{\rotatebox{90}{tiers}}
& $j\in\{0,\dots,J\}$ & quality tier, $0$ the strongest; $\varphi$ a routing mix across tiers \\
& $\Delta w_j$, $\ex[S_j]$, $\mu_j$ & power draw per active slot; mean service time; service rate $1/\ex[S_j]$ \\
& $d_j$ & dissatisfaction probability: the answer fails its user, increasing in $j$ \\
\midrule
\multirow{3}{*}{\rotatebox{90}{customers}}
& $\rho$ & probability an unsatisfied user re-asks (retry); $1-\rho$ abandons \\
& $p_r$, $\ell$ & churn probability upon abandonment; expected lifetime margin, so each abandonment carries the expected LTV loss $p_r\ell$ \\
& $p_r^{\mathrm{ns}}$ & newsvendor stockout (non-service) churn probability; in the numerics $p_r^{\mathrm{ns}} = p_r$ \\
& $r_t$ & fresh-demand arrival rate, the forecast $\ex[D_t\mid\calh_t]$ \\
\midrule
\multirow{5}{*}{\rotatebox{90}{derived}}
& $M_j=\frac{1}{1-d_j\rho}$ & retry multiplier: expected attempts per initiated task \\
& $\tS_j=M_j\,\ex[S_j]$ & effective service time: slot-time per initiated task \\
& $\mu^{\mathrm{eff}}_j=(1-d_j\rho)\mu_j=\frac{1}{\tS_j}$ & effective service rate: relaxation rate below capacity \\
& $\theta_j=(1-d_j\rho)k\mu_j m$ & terminal departures (satisfied or abandoned) per unit time at saturation \\
& $r^{\mathrm{eff}}_t$ & effective arrival rate: fresh demand plus completion-driven retries; saturated regime, where the completion rate is $k\mu_j m$ \\
\bottomrule
\end{tabular}
\end{table}

\FloatBarrier
\subsection{The transportation LP}\label{app:transportlp}

Figure~\ref{fig:transport} draws the transportation problem~\eqref{eq:transport} for two classes and two tiers.
\begin{figure}[!htbp]
\centering
\begin{tikzpicture}[
    >=Stealth,
    class/.style={draw, rounded corners=3pt, minimum width=3cm, minimum height=0.9cm, align=center, font=\small, fill=blue!6},
    tier/.style={draw, rounded corners=3pt, minimum width=3.6cm, minimum height=0.9cm, align=center, font=\small, fill=orange!8},
  ]
  \node[class] (R) at (0, 2) {researcher \quad $\lambda_R$};
  \node[class] (P) at (0, -2) {parser \quad $\lambda_P$};
  \node[tier] (T0) at (7, 2) {tier $0$ (strong) \quad $mk$};
  \node[tier] (T1) at (7, -2) {tier $1$ (weak) \quad $mk$};
  \draw[->, thick] (R.east) -- node[above, font=\scriptsize] {$\lambda_R\,\varphi(R,0)$} (T0.west);
  \draw[->, thick] (P.east) -- node[below, font=\scriptsize] {$\lambda_P\,\varphi(P,1)$} (T1.west);
  \draw[->, thick, densely dotted] (R.east) to[bend right=15] node[sloped, font=\scriptsize, pos=0.22, above=4pt] {$\lambda_R\,\varphi(R,1)$} (T1.west);
  \draw[->, thick, densely dotted] (P.east) to[bend left=15] node[sloped, font=\scriptsize, pos=0.68, below=4pt] {$\lambda_P\,\varphi(P,0)$} (T0.west);
  \node[font=\scriptsize, text=black!50] at (0, 0) {$\textstyle\sum_j \varphi(x,j)=1$};
  \node[font=\scriptsize, text=black!50] at (7, 0) {$\textstyle\sum_{x,j} \tS_j(x)\,\lambda_x\,\varphi(x,j)\;\le\;mk$};
  \node[font=\scriptsize, align=center, text=black!40] at (3.5, -4.2) {edge cost per initiated task:\\$c(x,j)=\gamma\,\tS_j(x)\,\Delta w_j+M_j(x)\,d_j(x)(1{-}\rho_x)\,(p_r\ell)_x$};
\end{tikzpicture}
\caption{The transportation LP \eqref{eq:transport} for two classes and two tiers. Solid arrows are the dominant flows in the calibrated instance; dotted arrows are the cross-assignments the LP considers and rejects. Capacity is charged in effective slot-time $\tS_j(x)=M_j(x)\,\ex[S_j]$, not posted $\ex[S_j]$: a heavy-retry class consumes capacity its service time conceals, and in the trap region the capacity constraint \emph{tightens} under degradation.}
\label{fig:transport}
\end{figure}

\FloatBarrier
\subsection{The response of Example~\ref{ex:surge}}\label{app:boomerang}

Figure~\ref{fig:boomerang} sketches the reactive and anticipatory paths of Example~\ref{ex:surge}.
\begin{figure}[!htbp]
\centering
\begin{tikzpicture}
\begin{axis}[width=0.55\textwidth, height=6.8cm, axis lines=left, xlabel={$t$ (seconds)}, ylabel={$N(t)$ (thousand jobs)}, xtick=\empty, ytick={2,3,3.5}, yticklabels={$2$,$mk{=}3$,$\wh N_{\mathrm{fire}}{=}3.5$}, yticklabel style={font=\scriptsize}, ymin=1.6, ymax=6.4, xmin=0, xmax=1.2, clip=true, legend style={draw=none, font=\small, at={(0.97,0.05)}, anchor=south east}]
\addplot[forget plot, black, dotted, thick] coordinates {(0,3) (1.2,3)};
\addplot[forget plot, gray, dash dot] coordinates {(0,3.5) (1.2,3.5)};
\addplot[red, very thick] coordinates {(0,2) (0.23,3) (0.48,3.5) (1.02,6.4)};
\addplot[blue, thick, dashed] coordinates {(0,2) (0.15,2.5) (0.35,2.85) (0.7,2.95) (1.2,2.95)};
\fill[red!70!black] (axis cs:0.48,3.5) circle (1.8pt);
\node[font=\scriptsize, red!70!black, anchor=west, align=left] at (axis cs:0.04,5.6) {throttle fires at $\wh N_{\mathrm{fire}}$:\\ slope more than doubles};
\draw[-{Stealth}, red!70!black, thin] (axis cs:0.3,5.3) -- (axis cs:0.48,3.6);
\node[font=\scriptsize, blue!60!black, anchor=north east] at (axis cs:1.17,2.86) {relaxes to $\bar N\approx 2.94<mk$};
\legend{reactive: unstable retry feedback, anticipatory: feedback never self-sustaining}\end{axis}
\end{tikzpicture}
\caption{The response of Example~\ref{ex:surge}, in the example's units of thousands of jobs. Three levels anchor the vertical axis: capacity $mk=3$ (dotted), the reactive trigger $\wh N_{\mathrm{fire}}=3.5$ (dash-dotted; the reactive policy of Proposition~\ref{pro:anticipation} instead fires at $\wh N_{\mathrm{fire}}=mk$, where the feedback is already self-sustaining), and the Little's-law equilibrium $\bar N=\textstyle\sum_x\lambda_x\,\tS(x)\approx 2.94$ of the anticipatory mix, which sits below capacity. Reactive throttling fires when $N$ crosses the trigger and the retry feedback becomes self-sustaining; the anticipatory path relaxes toward $\bar N$ and never saturates. The cost rate rises before it falls; the churned lifetime value accrues in the churn cost, which the operating-cost rate does not include.}
\label{fig:boomerang}
\end{figure}

\FloatBarrier
\subsection{Backlog trajectories under each policy}\label{app:trajectories}

Figure~\ref{fig:poc-traj} shows demand and backlog under each policy of Table~\ref{tab:poc-results}.
\input{Figures/latex_out/04_fig_trajectories}

\FloatBarrier
\subsection{One-at-a-time sensitivity}\label{app:onesens}

Table~\ref{tab:poc-sens} reports the one-at-a-time sensitivity summarized in Section~\ref{sec:poc-results}; the full grid is Table~\ref{tab:sens-full}.
\input{Figures/latex_out/09_table_sensitivity}

\FloatBarrier
\subsection{Identifying the behavioral triple}\label{app:dataid}
The product is the right target, because it is what the results consume: the multiplier \eqref{eq:multiplier}, the effective quantities \eqref{eq:effective}, the trap condition of Proposition~\ref{pro:trap}, the instability threshold of Proposition~\ref{pro:spiral}, and the release condition of Corollary~\ref{cor:latch} depend on $(d_j,\rho)$ only through $d_j\rho$, the directly estimable object; fresh demand follows from logged arrivals by the deconvolution of footnote~\ref{fn:demand}. Separating $d_j$ from $\rho$ is needed only for the churn column of \eqref{eq:ledger}, and the abandonment branch identifies it imperfectly: a dissatisfaction fingerprint followed by no retry brackets $d_j(1-\rho)$ from below, and all answers with no follow-up bracket it from above, because a silent satisfied exit and a silent dissatisfied one look alike. The output is an interval, which propagates to the churn column and leaves the capacity side of every frontier untouched: conditional on the fingerprint interpretation, the multiplier, the effective service time, the trap interval, and the fluid stability boundary are point-identified through $d_j\rho$; the churn column needs $d_j$ and $\rho_x$ apart plus $(p_r\ell)_x$, and carries the interval forward.

The comparison across tiers is confounded by construction: routers degrade under congestion, and congestion changes who is asking and how patient they are, so naive per-tier fingerprint rates mix the tier's effect with the crowd's. Identification leans on exogenous variation in tier assignment, past capacity incidents that shift routing for fleet-level reasons, or published clock policies such as DeepSeek's fixed peak windows; conditional on the exclusion assumption that such variation is orthogonal to any individual query's patience, the difference in fingerprint rates across a switched tier estimates the difference in $d_j\rho$. The churn side $(p_r\ell)_x$ anchors to subscription price points and to a survival regression of churn on exposure to degraded episodes, under the corresponding assumptions.

\subsection{Calibration of the numerical instances}\label{app:calibration}

Table~\ref{tab:poc-params} collects the calibration of the five instances of Section~\ref{sec:poc}; markers give each number's provenance.

\input{Figures/latex_out/06_table_parameters}

\FloatBarrier
\section{Proofs}\label{app:proofs}

\FloatBarrier
\subsection{Proof of Proposition~\ref{pro:anticipation}}\label{app:anticipation}

Throughout, classes $x$ carry fresh rates $\lambda_x(t)=\lambda_x^{\mathrm{calm}}$ off $[t_1,t_2]$ and $\lambda_x^{\mathrm{surge}}\ge\lambda_x^{\mathrm{calm}}$ on it; $\tS_j(x)=M_j(x)\,\ex[S_j]$; $L^{\mathrm{calm}}(a)=\sum_x\lambda_x^{\mathrm{calm}}\tS_{a(x)}(x)$ and $L^{\mathrm{surge}}(a)=\sum_x\lambda_x^{\mathrm{surge}}\tS_{a(x)}(x)$ for an assignment $a$; the fluid dynamics are those of \eqref{eq:value}, $dN(x)/dt=\lambda_x(t)-n(x)/\tS_{a(x)}(x)$ with $n(x)=N(x)\min\{1,\,mk/N\}$; and the stage cost is $c(k,\bN,a)$ of \eqref{eq:value} plus its churn sum, the chance constraint not imposed on either policy.

\emph{Step 1 (effective work).} For any fixed assignment $a$, let $V_a(t):=\sum_x N_t(x)\,\tS_{a(x)}(x)$. On any interval on which the assignment is constant, summing the class ODEs and using $\sum_x n(x)=\min\{N_t,\,mk\}$,
\[
\frac{dV_a}{dt}\;=\;\sum_x\lambda_x(t)\,\tS_{a(x)}(x)\;-\;\sum_x n(x)\;=\;L(a)(t)\;-\;\min\{N_t,\,mk\},
\]
independent of the backlog composition. Moreover $N_t\,\max_x\tS_{a(x)}(x)\;\ge\;\sum_x N_t(x)\,\tS_{a(x)}(x)=V_a(t)$, that is, $N_t\;\ge\;V_a(t)/\max_x\tS_{a(x)}(x)$.

\emph{Step 2 (reactive fires).} On $[t_1,\tau)$ the system is on tier $0$ and below $mk$, so classes decouple: $dN_t(x)/dt=\lambda_x^{\mathrm{surge}}-N_t(x)/\tS_0(x)$, whose solution is
\[
N_t(x)\;=\;\lambda_x^{\mathrm{calm}}\,\tS_0(x)\,e^{-(t-t_1)/\tS_0(x)}\;+\;\lambda_x^{\mathrm{surge}}\,\tS_0(x)\,\bigl(1-e^{-(t-t_1)/\tS_0(x)}\bigr),
\]
increasing toward $\lambda_x^{\mathrm{surge}}\tS_0(x)$, hence toward $L^{\mathrm{surge}}(0)>mk$ by (H1). Each $N_t(x)$ is continuous and nondecreasing in $t$, and their sum is strictly increasing since $L^{\mathrm{surge}}(0) > L^{\mathrm{calm}}(0)$, so $\tau:=\inf\{t\ge t_1:\sum_xN_t(x)=mk\}$ is finite and unique, and
\[
\tau-t_1\;\le\;\tS_0^{\max}\ln\frac{L^{\mathrm{surge}}(0)-L^{\mathrm{calm}}(0)}{L^{\mathrm{surge}}(0)-mk},
\]
where $\tS_0^{\max}:=\max_x\tS_0(x)$; at $\tau$ all classes switch to tier 1.

\emph{Step 3 (surge lower bound).} For $t\in[\tau,t_2]$ the reactive policy holds tier 1, so Step 1 applied to $a=1$ gives $dV_1/dt=L^{\mathrm{surge}}(1)-\min\{N_t,mk\}\;\ge\;L^{\mathrm{surge}}(1)-mk>0$ by (H2), and integrating from $V_1(\tau)\;\ge\;\tS_1^{\min}N_\tau=mk\,\tS_1^{\min}$,
\[
V_1(t)\;\ge\;mk\,\tS_1^{\min}+\bigl(L^{\mathrm{surge}}(1)-mk\bigr)(t-\tau).
\]
Since $N_t\;\ge\;V_1(t)/\tS_1^{\max}$, we get $N_t\;\ge\;mk\,\tS_1^{\min}/\tS_1^{\max}>\wh N_{\mathrm{rel}}$ by (H4), so the release does not fire on $[\tau,t_2]$, and
\[
(N_t-mk)^{+}\;\ge\;\frac{\bigl[V_1(t)-mk\,\tS_1^{\max}\bigr]^{+}}{\tS_1^{\max}}\;\ge\;b\,(t-\tau-t_{\mathrm{lag}})^{+},\qquad b:=\frac{L^{\mathrm{surge}}(1)-mk}{\tS_1^{\max}}.
\]
Integrating the service-level term of \eqref{eq:cost} with elasticity $\beta_{\mathrm{SLA}}=2$ over the window $[\tau+t_{\mathrm{lag}},\,t_2]$, of length $u$,
\[
\int_{\tau+t_{\mathrm{lag}}}^{t_2}C_{\mathrm{SLA}}\Bigl(\frac{(N_t-mk)^{+}}{mk}\Bigr)^{2}dt\;\ge\;\frac{C_{\mathrm{SLA}}}{(mk)^{2}}\int_{0}^{u}b^{2}s^{2}\,ds\;=\;\frac{C_{\mathrm{SLA}}b^{2}u^{3}}{3(mk)^{2}}.
\]

\emph{Step 3b (reactive churn lower bound).} On $[\tau+t_{\mathrm{lag}},\,T]$: the bound of Step 3 gives $V_1(t)\ge mk\,\tS_1^{\max}$ at $t=\tau+t_{\mathrm{lag}}$ and thereafter on $[\tau+t_{\mathrm{lag}},t_2]$, and the bound of Step 4 gives $V_1(t)\ge mk\,\tS_1^{\max}+\tS_1^{\max}(b u-b'(t-t_2))\ge mk\,\tS_1^{\max}$ on $[t_2,T]$; hence $N_t\ge V_1(t)/\tS_1^{\max}\ge mk$ throughout $[\tau+t_{\mathrm{lag}},T]$ (the release does not fire on the window by Steps 3 and 4). Tier-1 completions run at $k\mu_1 m$, each a class-$x$ attempt that fails and abandons with probability $d_1(x)(1-\rho_x)$, so the churn rate is at least
\[
\Psi'\;:=\;k\mu_1 m\,\min_x\bigl[d_1(x)(1-\rho_x)(p_r\ell)_x\bigr],
\]
accruing $\Psi:=\Psi'(1+\eta)\,u$ over $[\tau+t_{\mathrm{lag}},\,T]$.

\emph{Step 4 (recovery lower bound).} For $t=t_2+s$, $s\in[0,\eta u]$, the surge is over and the reactive policy still holds tier 1 (the release has not fired, by the same argument: $N_t\ge mk>\wh N_{\mathrm{rel}}$ below). Step 1 gives $dV_1/dt\;\ge\;L^{\mathrm{calm}}(1)-mk=-b'\tS_1^{\max}$, so $(N_t-mk)^{+}\;\ge\;b u-b'\,s$, a lower bound that reaches zero exactly at $s=b u/b'=\eta\,u$, i.e.\ at $T$; throughout this interval $N_t\ge V_1(t)/\tS_1^{\max}\ge mk>\wh N_{\mathrm{rel}}$, so the release does not fire. Integrating as in Step 3,
\[
\int_{t_2}^{T}C_{\mathrm{SLA}}\Bigl(\frac{(N_t-mk)^{+}}{mk}\Bigr)^{2}dt\;\ge\;\frac{C_{\mathrm{SLA}}}{(mk)^{2}}\int_{0}^{\eta u}\bigl(b u-b's\bigr)^{2}ds\;=\;\frac{C_{\mathrm{SLA}}b^{3}u^{3}}{3(mk)^{2}b'}.
\]
Adding Step 3 gives $\Phi=C_{\mathrm{SLA}}b^{2}u^{3}(1+b/b')/(3(mk)^{2})$ and Step 3b gives $\Psi$. The idle floor $\gamma k\,w_0$ accrues on all of $[t_0,T]$, and every other reactive cost term is nonnegative, which proves (i).

\emph{Step 5 (anticipatory stays below capacity).} Let $t':=\inf\{t\ge t_0:\;N_t\ge mk\}$. On $[t_0,t')$ the assignment is $a^{\star}$ and, no class having crossed its own capacity share, classes decouple; since $\lambda_x(t)\le\lambda_x^{\mathrm{surge}}$, each $N_t(x)$ is a convex combination of $N_{t_0}(x)$ and values at most $\lambda_x^{\mathrm{surge}}\tS_{a^{\star}(x)}(x)$, so $N_t(x)\le N^{\star}_x$ and $N_t\le\sum_xN^{\star}_x<mk$ by (H3). By continuity $t'=\infty$. The service-level term of the stage cost is therefore identically zero on $[t_0,T]$, and the energy, memory, and churn rates are bounded by the corresponding terms of $D^{\mathcal{A}}$ because $n(x)\le N(x)\le N^{\star}_x$ and $N\le\sum_xN^{\star}_x$. This proves (ii).

\emph{Step 6.} Claim (iii) is (i) minus (ii). For (iv): $t_2=\tau+t_{\mathrm{lag}}+u$ and $T=t_2+\eta u$ give $T-t_0=(\tau+t_{\mathrm{lag}}-t_0)+(1+\eta)u$, so $K\,(T-t_0)=B+Cu$ with $B=K(\tau+t_{\mathrm{lag}}-t_0)$ and $C=K(1+\eta)$, and the reactive churn of Step 3b adds $\Psi'(1+\eta)u$. Define $g(u):=Au^{3}+\bigl(\Psi'(1+\eta)-C\bigr)u-B$, where $A:=C_{\mathrm{SLA}}b^{2}(1+b/b')/(3(mk)^{2})$ is the coefficient of $\Phi$ above; condition (iii) reads $g(u)>0$; $g$ is strictly convex on $u\ge 0$ with $g(0)=-B<0$, so it has exactly one positive root $u^{\star}$. \hfill$\square$

\FloatBarrier
\subsection{Numerical verification of Proposition~\ref{pro:anticipation} on Example~\ref{ex:surge}}\label{app:prop6check}

\textbf{Inputs.} With the class split of Example~\ref{ex:surge} (70\% sensitive, 30\% insensitive), the loads are $L^{\mathrm{calm}}(0)=2.000k$, $L^{\mathrm{surge}}(0)=3.330k$, $L^{\mathrm{calm}}(1)=0.7\times20\times0.115385+0.3\times20\times0.060914=1.981k$, and $L^{\mathrm{surge}}(1)=0.7\times33.3\times0.115385+0.3\times33.3\times0.060914=3.298k$ (thousands of slots), so $L^{\mathrm{calm}}(a)<mk<L^{\mathrm{surge}}(a)$ for both $a\in\{0,1\}$: hypotheses (H1) and (H2) hold. At the switch the tier-0 rest is $N_{t_0}(x)=\lambda^{\mathrm{calm}}_x\,\tS_0(x)$, giving $\sum_x N^{\star}_x=\max\{1.400,\,2.331\}+\max\{0.600,\,0.609\}=2.940k<3k$: (H3) holds. The firing time along the all-strong trajectory satisfies $\tau-t_1=\tS_0\ln\bigl[(L^{\mathrm{surge}}(0)-L^{\mathrm{calm}}(0))/(L^{\mathrm{surge}}(0)-mk)\bigr]=\ln[(3.33-2)/(3.33-3)]/10=0.139384$\,s, and the lag $t_{\mathrm{lag}}=mk(\tS_1^{\max}-\tS_1^{\min})/(L^{\mathrm{surge}}(1)-mk)=3000\times0.054471/298.14=0.548101$\,s, so (H4) requires a surge of duration $t_2-t_1>0.687485$\,s and a release threshold $\wh N_{\mathrm{rel}}<mk\,\tS_1^{\min}/\tS_1^{\max}=1583.756$. We fix the remaining inputs as follows: lead $t_1-t_0=1$\,s; surge duration $t_2-t_1=2$\,s; $\wh N_{\mathrm{rel}}=1500$; $\gamma=0.10$ dollars per kWh; idle draw $w_0=1.2$\,kW per server; $\Delta w_0=4$\,kW and $\Delta w_1=2$\,kW as in Example~\ref{ex:twotiers}; memory scale $c_m=0.01$ with exponent $\kappa=1.3$; service-level scale $C_{\mathrm{SLA}}=50{,}000$ dollars per hour at one full-capacity excess, with elasticity $\beta_{\mathrm{SLA}}=2$ as in \eqref{eq:cost}; and expected LTV losses $(p_r\ell)_x$ of $\$144$ for the sensitive class and $\$7.20$ for the insensitive class.

\textbf{Evaluation.} For these inputs the quantities of the proposition evaluate as follows: $\eta=0.2925$, $u=t_2-\tau-t_{\mathrm{lag}}=1.3125$\,s, $T-t_0=3.3840$\,s, $b=2583.9$ jobs per second, $b'=8832.5$ jobs per second, $\Phi=\$10.04$, and the anticipatory cost rate $D^{\mathcal{A}}=\$2{,}556.12$ per second, of which the idle floor $\gamma k\,w_0=\$0.005$ per second is the only constant, so $K=D^{\mathcal{A}}-\gamma k\,w_0=\$2{,}556.12$ per second. Without the reactive-churn bound (Step 3b of the proof), the sufficient condition does not certify the $2$-second surge: the critical surge length is $u^{\star}=27.91$\,s, far above $u=1.31$\,s, so $\Phi=\$10.04$ is far below $K\,(T-t_0)=\$8{,}649.83$. With Step 3b the reactive policy accrues churn at rate $\Psi'=k\mu_1 m\min_x[d_1(x)(1-\rho_x)(p_r\ell)_x]=50{,}000\times\min\{0.6\times0.2\times144,\,0.05\times0.7\times7.20\}=\$12{,}600$ per second while saturated, accruing $\Psi=\Psi'(1+\eta)u=\$21{,}376$ over $[\tau+t_{\mathrm{lag}},\,T]$; then $\Phi+\Psi=\$21{,}386>K\,(T-t_0)=\$8{,}650$, so (iii) holds and anticipation strictly dominates, with $u^{\star}=0.33$\,s from $Au^{3}+(\Psi'(1+\eta)-C)u=B$ (the reactive policy's churn under full degradation, $\Psi'(1+\eta)=\$16{,}286$ per second, already exceeds the anticipatory cost rate $K$, the regime the remark after the proposition describes).
\FloatBarrier
\subsection{Derivation of the break-even expected LTV loss (Remark~\ref{rem:memory})}\label{app:memory}
Per unit of rerouted rate, the full hourly cost, energy plus memory minus the tier-to-tier churn difference of \eqref{eq:ledger}, reads
\[
\gamma\bigl(\tS_0\Delta w_0-\tS_j\Delta w_j\bigr)
\;+\;\gamma c_m\kappa N^{\kappa-1}\bigl(\tS_0-\tS_j\bigr)
\;-\;\bigl[M_j\,d_j-M_0\,d_0\bigr](1-\rho)\,p_r\ell,
\]
and setting it to zero yields the closed-form break-even expected LTV loss
\[
(p_r\ell)^{\star}(N)\;=\;\frac{\gamma\bigl(\tS_0\Delta w_0-\tS_j\Delta w_j\bigr)+\gamma c_m\kappa N^{\kappa-1}\bigl(\tS_0-\tS_j\bigr)}{\bigl[M_j\,d_j-M_0\,d_0\bigr](1-\rho)},
\]
increasing in $N$ when $\tS_j<\tS_0$, with minimum at $N=0$, the pure-energy break-even $(p_r\ell)^{\star}:=(p_r\ell)^{\star}(0)$. With $p_r\ell$ below this break-even the weak tier dominates at every congestion level and the optimal policy has no occupancy threshold at all; above it the sign flips at the finite threshold
\[
N^{\star}\;=\;\left[\frac{\bigl[M_j\,d_j-M_0\,d_0\bigr](1-\rho)\,p_r\ell-\gamma\bigl(\tS_0\Delta w_0-\tS_j\Delta w_j\bigr)}{\gamma c_m\kappa\bigl(\tS_0-\tS_j\bigr)}\right]^{1/(\kappa-1)},
\]
below which the churn charge dominates and above which the memory discount does. The threshold $N^{\star}$ is determined by $p_r\ell$ relative to $(p_r\ell)^{\star}$.

\FloatBarrier
\subsection{The two-regime recursion in the queueing literature}\label{app:litqueue}
The standard responses to time-varying load sit at two poles. At one, stationary formulas are applied pointwise, blind to the object this paper controls: how fast the backlog moves between regimes, the transient on which the instability threshold and the release condition live. At the other, exact transient analysis, the fluid and strong-approximation limits for time-varying many-server systems \parencite{MandelbaumMassey1995, Whitt2002} and, further out, the distribution-dependent formulations of \textcite{DiekerHackmanWangYan2026}, describes the trajectory faithfully but as an object to be integrated, at the integrator's price derived above. The two-regime recursion of Lemma~\ref{lem:fluid} sits between the poles: a transient formula, a \emph{transient Little's law} that relaxes to the stationary prescription as the step grows. Its ingredients are classical, the recovering branch being the transient infinite-server decomposition of \textcite{EickMasseyWhitt1993}; what we have not found is the pieced two-regime trajectory itself, crossing times included and arrivals endogenous, nor its use as a constant-cost transition oracle inside an optimization loop.

\FloatBarrier
\section{Additional numerical assets}\label{app:trees}
This appendix collects the remaining per-provider artifacts: the distilled trees for Google, DeepSeek, and Kimi (Figures~\ref{fig:trees-google}--\ref{fig:trees-kimi}) and the pointwise Q-regret statistics (Table~\ref{tab:poc-qregret}).
\input{Figures/latex_out/03_trees_google}
\input{Figures/latex_out/03_trees_deepseek}
\input{Figures/latex_out/03_trees_kimi}
\input{Figures/latex_out/07_table_qregret}

\FloatBarrier
\section{Discrete-event simulation: implementation and estimates}\label{app:des}
The simulation of Section~\ref{sec:poc-protocol} batches time at a per-provider step $\Delta t$ with $\mu_{\max}\Delta t \le 0.1$ (Anthropic $0.996$ s, DeepSeek $1$ s, Google $0.241$ s, Kimi $1$ s, OpenAI $0.679$ s), with exact conditional draws at each step; the 30 replications of a policy run in parallel as a vector dimension. Per-attempt latency is a queue-wait age plus an exponential service; an attempt still waiting at hour end whose age exceeds the threshold counts as a certain violation in that hour, and the hourly denominator pools service starts and such attempts. State per replication: resident jobs per class, in-service counts, and per-class FCFS queues with wait-age buckets at 5-second granularity, an open final bucket holding everything older than the 300-second threshold; service starts draw from the \emph{oldest} buckets, and attempts that never start because the backlog outruns capacity are counted through the still-waiting rule above. Each step draws fresh arrivals per class ($\mathrm{Poisson}(\lambda_x\Delta t)$, $\lambda_x$ the class's fresh rate in the hour), completions per class by per-job thinning (an in-service job completes within $\Delta t$ with probability $1-e^{-\mu\Delta t}$, exact for exponential services). The simulation is time-stepped at $\Delta t$; we use DES as shorthand for it. A slot freed within a step cannot complete again until the next step, understating saturated throughput by a factor $(1-e^{-x})/x$ with $x=\mu_j\Delta t$, about 4.8\% at $x=0.1$; the understatement is largest for the fastest tier, and feasibility labels within a few percent of the threshold could move under an event-driven simulation. The step also draws retries ($\mathrm{Binomial}$ at $d_x\rho_x$) re-entering their class queue immediately at wait age zero, and abandonments ($\mathrm{Binomial}$ at $d_x(1-\rho_x)$) booked at the class's expected LTV loss. Capacity follows the fluid model: per-class proportional shares $n_x=N_x\min\{1,\,mk/N\}$ recomputed every step (per-tier pools where published, e.g.\ DeepSeek's 500-slot Pro pool). Per-attempt latency is recorded at service start: $W_{\mathrm{attempt}} = $ wait $+ \mathrm{Exp}(\mu_x)$, so an attempt starting with wait $w$ violates with probability $1$ if $w\ge 300$\,s and $e^{-\mu_x(300-w)}$ otherwise; hourly point estimates pool all service starts in the hour. Costs use the DES-measured hourly mean backlog and attempt counts in the stage cost \eqref{eq:cost}. Fluid-versus-DES agreement is the maximum relative gap between the closed-form forward pass and the simulation's mean hourly backlogs: under the DP's own policy, \spiTrajCheckMax{} at the worst (Google's surge hour, where the policy acts at the capacity boundary) and \spiTrajCheckMeanDP{} on average across providers. The threshold-hysteretic baselines shift firing times between the deterministic fluid path and the simulation, so their gaps (up to \spiTrajCheckMaxAll{}) are timing effects.

\input{Figures/latex_out/10_table_des_cis}
\FloatBarrier

\section{Full sensitivity grid}\label{app:sens}
Table~\ref{tab:poc-sens} (Appendix~\ref{app:onesens}) summarizes the one-at-a-time grid; this appendix records the per-policy DES-estimated means behind it and the forecast-noise protocol. Forecast noise draws $\hat r_t=\max(0, r_t(1+\varepsilon_t))$, $\varepsilon_t\sim\mathcal N(0,\sigma^2)$, $\sigma\in\{0.10,0.25\}$, 20 seeds; the DP is solved on $\hat r$ and evaluated by DES on true demand $r$, with baselines tuned and evaluated on $r$. The perturbations are $\rho_x\mapsto\min(1, 1.25\rho_x)$ (agentic $0.98\times1.25$ caps at $1.00$) and $0.75\rho_x$; $p_r\ell\times\{0.5,2\}$; $C_{\mathrm{SLA}}\times\{0.5,2\}$; $\beta_x\times0.5$; and demand peaks at $\{0.8,1.0,1.1\}\,\theta_{\mathrm{mix}}$.

\input{Figures/latex_out/13_table_sens_full}
\FloatBarrier

\section{Backlog-growth sweep: per-$C$ detail}\label{app:sweep}
Tables~\ref{tab:sweep-750}--\ref{tab:sweep-12000} report the full grid behind Figure~\ref{fig:sweep}, one table per capacity $C$, including every refined-grid row at spacing $0.0005$ within $\pm 0.02$ of $d^{\star}=0.4175$; the DES step is $\Delta t=5$\,ms, so $\mu_1\Delta t = 0.083 \le 0.1$ (Appendix~\ref{app:des}); the sweep draws completions as $\mathrm{Poisson}(S\,\mu_1\Delta t)$, not by the thinning rule of Appendix~\ref{app:des}.
\FloatBarrier
\input{Figures/latex_out/12_table_sweep_750}
\FloatBarrier
\input{Figures/latex_out/12_table_sweep_3000}
\FloatBarrier
\input{Figures/latex_out/12_table_sweep_12000}
\end{document}

%% file: Figures/latex_out/sec5_numbers.tex
\newcommand{\spiAnthropicDPcost}{\$9.40M}
\newcommand{\spiAnthropicTreeGap}{0.48\%}

\newcommand{\spiAnthropicBestInfeasRatio}{106\ensuremath{\times}}

\newcommand{\spiAnthropicChurn}{\$4.27M}
\newcommand{\spiAnthropicStumpChurn}{\$55.58M}
\newcommand{\spiAnthropicStumpEcon}{\$991.88M}
\newcommand{\spiAnthropicStumpVh}{9}

\newcommand{\spiAnthropicLPratio}{124.23\ensuremath{\times}}

\newcommand{\spiAnthropicActions}{9}
\newcommand{\spiAnthropicActionsFull}{125}

\newcommand{\spiOpenAITreeGap}{0.23\%}

\newcommand{\spiOpenAILPratio}{0.89\ensuremath{\times}}

\newcommand{\spiOpenAIActions}{50}
\newcommand{\spiOpenAIActionsLedger}{72}
\newcommand{\spiOpenAIActionsFull}{2744}

\newcommand{\spiGoogleTreeGap}{-3.70\%}

\newcommand{\spiGoogleBestInfeasRatio}{190\ensuremath{\times}}

\newcommand{\spiGoogleLPratio}{788.36\ensuremath{\times}}

\newcommand{\spiDeepSeekLPratio}{1.00\ensuremath{\times}}

\newcommand{\spiDeepSeekActions}{4}
\newcommand{\spiDeepSeekActionsFull}{27}

\newcommand{\spiKimiLPratio}{1.00\ensuremath{\times}}

\newcommand{\spiTrajCheckMax}{42.5\%}

\newcommand{\spiTrajCheckMeanDP}{11.80\%}
\newcommand{\spiTrajCheckMaxAll}{521\%}

\newcommand{\spiAnthropicDualBase}{3.9\ensuremath{\times}}

\newcommand{\spiGoogleRentWall}{\$145}

%% file: Figures/latex_out/01_table_results.tex
\begin{table}[t]\centering
\caption{Twenty-four-hour DES-estimated economic cost by policy and provider (30 replications; mean; lower is better), with the per-attempt SLA screen of Section~\ref{sec:poc-protocol} applied to the DP's action menu. Demand is scaled to each provider's own throughput, so dollar figures compare across policies within a row, not across providers. Baselines: \emph{always strongest} serves every query at tier 0; the \emph{reactive threshold} degrades everyone when total backlog crosses a trigger, releasing below a second threshold; \emph{scheduled} degrades everyone during the published peak windows; the \emph{index rule (hourly)} re-solves the class-aware transportation problem \eqref{eq:transport} hourly against the forecast. All are tuned per instance as in Section~\ref{sec:poc-protocol}. $^\dagger$\emph{div}: the DES-estimated cost exceeds the $1000\times$ reporting cutoff under the policy's own retry feedback; no finite estimate is reported. $^{k\mathrm{v}}$: the policy's DES-estimated $\pr(W_{\mathrm{attempt}}>300\,\mathrm{s})$ exceeds $0.05$ in $k$ of 24 hours (point estimates), and the policy is marked infeasible. The final column reports the best \emph{feasible} baseline as a multiple of the DP; where no baseline is feasible it reports NA.}\label{tab:poc-results}
\small\begin{tabular}{@{}lrrrrrrr@{}}\toprule
provider & always strongest & reactive threshold & scheduled & index rule & \textbf{fluid DP} & tree & best/DP \\\midrule
Anthropic & {\color{gray}\emph{div}}$^{\dagger}$ & \$991.88M$^{9\mathrm{v}}$ & {\color{gray}\emph{div}}$^{\dagger}$ & \$1.17B$^{3\mathrm{v}}$ & \textbf{\$9.40M} & \$9.44M & NA \\
OpenAI & \$117.70M$^{24\mathrm{v}}$ & \$8.73M$^{21\mathrm{v}}$ & \$61.84M$^{17\mathrm{v}}$ & \$1.59M & \textbf{\$1.78M} & \$1.78M & \textbf{0.89$\times$} \\
Google & \$7.26B$^{3\mathrm{v}}$ & \$1.75B$^{2\mathrm{v}}$ & \$7.26B$^{3\mathrm{v}}$ & \$7.26B$^{3\mathrm{v}}$ & \textbf{\$9.21M} & \$8.87M & NA \\
DeepSeek & {\color{gray}\emph{div}}$^{\dagger}$ & \$27.87M$^{23\mathrm{v}}$ & {\color{gray}\emph{div}}$^{\dagger}$ & \$2.03M & \textbf{\$2.03M} & \$2.03M & \textbf{1.00$\times$} \\
Kimi & \$3.44M & \$3.44M & \$1.62B$^{12\mathrm{v}}$ & \$3.44M & \textbf{\$3.44M} & \$3.44M & \textbf{1.00$\times$} \\
\bottomrule\end{tabular}
\end{table}

%% file: Figures/latex_out/08_fig_shadow.tex
\begin{figure}[t]\centering
\includegraphics[width=\textwidth]{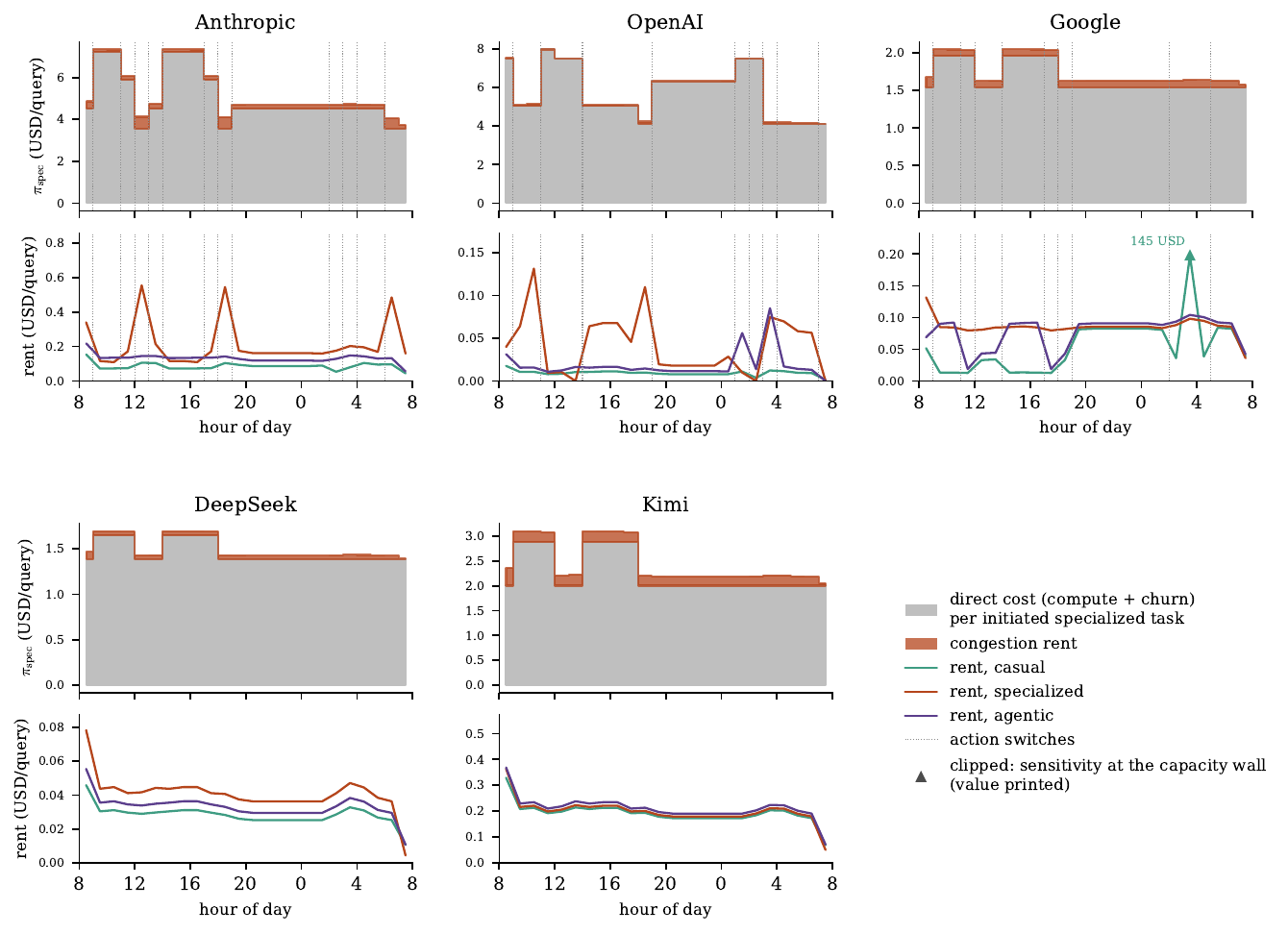}
\caption{Policy-conditioned sensitivities $\hat\pi_x(t)$, five providers: finite differences of the twenty-four-hour cost through the closed-form fluid forward pass under the DP's own actions; dotted verticals mark hours where the optimal action switches, where the sensitivity may kink; the sensitivities agree with the optimized-value derivative \eqref{eq:shadow} away from switch hours. \emph{Top panels}: the specialized class's sensitivity decomposed into the direct cost of one more initiated task at the hour's action, $M_j(x)\bigl(c_j+d_j(x)(1{-}\rho_x)p_r\ell_x\bigr)$ with the peak electricity multiplier, and the congestion rent above it. \emph{Bottom panels}: the rent by class, in dollars; rent panels are clipped where a single hour dominates, the value printed at the marker. Sensitivities in the last hours are biased low by the finite horizon, and the rent reaches the latency constraint through the smooth $C_{\mathrm{SLA}}$ term, with no hard cutoff.}\label{fig:shadow}
\end{figure}

%% file: Figures/latex_out/02_table_certificates.tex
\begin{table}[t]\centering
\caption{Runtime and DES-estimated accuracy of the fluid solver against a brute-force competitor: the identical dynamic program, grid, action set, and cost function, with the closed-form transition replaced by a converged Euler integration \emph{of the same fluid dynamics} at the per-provider stable step. The last column is the DES-estimated policy-value gap, $100\times(\mathrm{Euler}-\mathrm{fluid})/\mathrm{Euler}$ from the two greedy policies' twenty-four-hour costs (30 replications each): negative means the Euler-solved policy is cheaper under simulation. It replaces the fluid-rollout policy-value comparison; pointwise Q-regret statistics are in Appendix~\ref{app:trees}.}\label{tab:poc-cert}
\small\begin{tabular}{@{}lrrrr@{}}\toprule
provider & fluid (s) & Euler (s) & speedup & DES policy-value gap \\\midrule
Anthropic & 0.43 & 34 & 79$\times$ & -3.68\% \\
OpenAI & 2.24 & 275 & 123$\times$ & -12.15\% \\
Google & 1.29 & 414 & 321$\times$ & +2.11\% \\
DeepSeek & 0.18 & 14 & 75$\times$ & +0.00\% \\
Kimi & 0.05 & 4 & 74$\times$ & +0.00\% \\
\bottomrule\end{tabular}
\end{table}

%% file: Figures/latex_out/03_trees_anthropic.tex
\providecommand{\tleaf}[3]{\begin{tabular}{@{}c@{}}{\scriptsize\textbf{\textcolor{#1}{#3}}}\\[-2.5pt]{\tiny\textcolor{black!55}{tier #2}}\end{tabular}}
\providecommand{\tswatch}[2]{\textcolor{#1}{\rule[-0.25ex]{0.9em}{1.5ex}}\,{\scriptsize #2}}
\providecommand{\spifitbox}[1]{\sbox0{#1}\ifdim\wd0>\linewidth\resizebox{\linewidth}{!}{\usebox0}\else\usebox0\fi}
\forestset{spitree/.style={for tree={draw=black!30, rounded corners=1.5pt, edge={black!45,thin}, l sep=7pt, s sep=4pt, inner sep=2.5pt, font=\scriptsize, align=center}}}
\definecolor{anthropicT0}{HTML}{1A3A6B}
\definecolor{anthropicT1}{HTML}{2C6E9B}
\definecolor{anthropicT2}{HTML}{3E9C82}
\definecolor{anthropicT3}{HTML}{C98A2B}
\definecolor{anthropicT4}{HTML}{B5451B}
\begin{figure}[!htbp]\centering
\textbf{Anthropic}\\[3pt]
{\scriptsize models, strongest $\to$ weakest:\quad \tswatch{anthropicT0}{Opus 5 max (tier 0)}\quad \tswatch{anthropicT1}{Opus 5 high (tier 1)}\quad \tswatch{anthropicT2}{Opus 5 medium (tier 2)}\quad \tswatch{anthropicT3}{Opus 5 low (tier 3)}\quad \tswatch{anthropicT4}{Sonnet 5 non-reasoning (tier 4)}}\\[7pt]
{\small\itshape where class = casual is served, at every hour (fluid DP, distilled)}\\[3pt]
\spifitbox{\begin{forest}spitree
[{$\hat r_t\le 162\,301$} [{$N_{\mathrm{ag}}\le 35$} [{$N_{\mathrm{cas}}\le 456$} [\tleaf{anthropicT2}{2}{Opus 5 medium}, fill=anthropicT2!10, draw=anthropicT2!70, line width=0.7pt,edge label={node[midway,left,font=\tiny]{yes}}] [\tleaf{anthropicT4}{4}{Sonnet 5 non-reasoning}, fill=anthropicT4!10, draw=anthropicT4!70, line width=0.7pt,edge label={node[midway,right,font=\tiny]{no}}],edge label={node[midway,left,font=\tiny]{yes}}] [{$N_{\mathrm{spec}}\le 59$} [\tleaf{anthropicT4}{4}{Sonnet 5 non-reasoning}, fill=anthropicT4!10, draw=anthropicT4!70, line width=0.7pt,edge label={node[midway,left,font=\tiny]{yes}}] [\tleaf{anthropicT2}{2}{Opus 5 medium}, fill=anthropicT2!10, draw=anthropicT2!70, line width=0.7pt,edge label={node[midway,right,font=\tiny]{no}}],edge label={node[midway,right,font=\tiny]{no}}],edge label={node[midway,left,font=\tiny]{yes}}] [{$N_{\mathrm{ag}}\le 35$} [\tleaf{anthropicT4}{4}{Sonnet 5 non-reasoning}, fill=anthropicT4!10, draw=anthropicT4!70, line width=0.7pt,edge label={node[midway,left,font=\tiny]{yes}}] [{$N_{\mathrm{spec}}\le 59$} [\tleaf{anthropicT4}{4}{Sonnet 5 non-reasoning}, fill=anthropicT4!10, draw=anthropicT4!70, line width=0.7pt,edge label={node[midway,left,font=\tiny]{yes}}] [\tleaf{anthropicT3}{3}{Opus 5 low}, fill=anthropicT3!10, draw=anthropicT3!70, line width=0.7pt,edge label={node[midway,right,font=\tiny]{no}}],edge label={node[midway,right,font=\tiny]{no}}],edge label={node[midway,right,font=\tiny]{no}}]]
\end{forest}}\\[10pt]
{\small\itshape where class = specialized is served, at every hour (fluid DP, distilled)}\\[3pt]
\spifitbox{\begin{forest}spitree
[{peak window ($t\in$ 09:00--12:00, 14:00--18:00)?} [{before 06:00?} [{$N_{\mathrm{ag}}\le 35$} [\tleaf{anthropicT2}{2}{Opus 5 medium}, fill=anthropicT2!10, draw=anthropicT2!70, line width=0.7pt,edge label={node[midway,left,font=\tiny]{yes}}] [\tleaf{anthropicT1}{1}{Opus 5 high}, fill=anthropicT1!10, draw=anthropicT1!70, line width=0.7pt,edge label={node[midway,right,font=\tiny]{no}}],edge label={node[midway,left,font=\tiny]{yes}}] [\tleaf{anthropicT0}{0}{Opus 5 max}, fill=anthropicT0!10, draw=anthropicT0!70, line width=0.7pt,edge label={node[midway,right,font=\tiny]{no}}],edge label={node[midway,left,font=\tiny]{no}}] [{before 17:00?} [{$N_{\mathrm{spec}}\le 59$} [\tleaf{anthropicT0}{0}{Opus 5 max}, fill=anthropicT0!10, draw=anthropicT0!70, line width=0.7pt,edge label={node[midway,left,font=\tiny]{yes}}] [\tleaf{anthropicT2}{2}{Opus 5 medium}, fill=anthropicT2!10, draw=anthropicT2!70, line width=0.7pt,edge label={node[midway,right,font=\tiny]{no}}],edge label={node[midway,left,font=\tiny]{yes}}] [\tleaf{anthropicT1}{1}{Opus 5 high}, fill=anthropicT1!10, draw=anthropicT1!70, line width=0.7pt,edge label={node[midway,right,font=\tiny]{no}}],edge label={node[midway,right,font=\tiny]{yes}}]]
\end{forest}}\\[10pt]
{\small\itshape where class = agentic is served, at every hour (fluid DP, distilled)}\\[3pt]
\spifitbox{\begin{forest}spitree
[\tleaf{anthropicT2}{2}{Opus 5 medium}, fill=anthropicT2!10, draw=anthropicT2!70, line width=0.7pt]
\end{forest}}\\[10pt]
{\small\itshape baselines: one tree for every class}\\[3pt]
\spifitbox{\begin{tabular}{c@{\hspace{2.5em}}c@{\hspace{2.5em}}c}
{\scriptsize\itshape reactive threshold} & {\scriptsize\itshape scheduled (peak windows)} & {\scriptsize\itshape always strongest} \\[3pt]
\begin{forest}spitree [{$N_{\mathrm{tot}}> 0.8C$?} [\tleaf{anthropicT0}{0}{Opus 5 max}, fill=anthropicT0!10, draw=anthropicT0!70, line width=0.7pt,edge label={node[midway,left,font=\tiny]{no}}] [\tleaf{anthropicT4}{4}{Sonnet 5 non-reasoning}, fill=anthropicT4!10, draw=anthropicT4!70, line width=0.7pt,edge label={node[midway,right,font=\tiny]{yes}}]] \end{forest} &
\begin{forest}spitree [{peak window ($t\in$ 09:00--12:00, 14:00--18:00)?} [\tleaf{anthropicT0}{0}{Opus 5 max}, fill=anthropicT0!10, draw=anthropicT0!70, line width=0.7pt,edge label={node[midway,left,font=\tiny]{no}}] [\tleaf{anthropicT3}{3}{Opus 5 low}, fill=anthropicT3!10, draw=anthropicT3!70, line width=0.7pt,edge label={node[midway,right,font=\tiny]{yes}}]] \end{forest} &
\begin{forest}spitree [\tleaf{anthropicT0}{0}{Opus 5 max}, fill=anthropicT0!10, draw=anthropicT0!70, line width=0.7pt] \end{forest}
\end{tabular}}
\caption{Anthropic: the distilled routing, one tree per class, and the baseline policies in the same form. Baselines: \emph{reactive threshold} degrades everyone above a backlog trigger (release at $0.5\,C$, $C$ the pool capacity); \emph{scheduled} degrades during the published windows; \emph{always strongest} never degrades. Trees may split on other classes' backlogs, since a shared pool reprices capacity everywhere; splits with identical branches are merged. DES-estimated rollout value gap of the tree to the DP: 0.48\%.}\label{fig:trees-anthropic}
\end{figure}

%% file: Figures/latex_out/03_trees_openai.tex
\providecommand{\tleaf}[3]{\begin{tabular}{@{}c@{}}{\scriptsize\textbf{\textcolor{#1}{#3}}}\\[-2.5pt]{\tiny\textcolor{black!55}{tier #2}}\end{tabular}}
\providecommand{\tswatch}[2]{\textcolor{#1}{\rule[-0.25ex]{0.9em}{1.5ex}}\,{\scriptsize #2}}
\providecommand{\spifitbox}[1]{\sbox0{#1}\ifdim\wd0>\linewidth\resizebox{\linewidth}{!}{\usebox0}\else\usebox0\fi}
\forestset{spitree/.style={for tree={draw=black!30, rounded corners=1.5pt, edge={black!45,thin}, l sep=7pt, s sep=4pt, inner sep=2.5pt, font=\scriptsize, align=center}}}
\definecolor{openaiT0}{HTML}{1A3A6B}
\definecolor{openaiT1}{HTML}{2C6E9B}
\definecolor{openaiT2}{HTML}{3E9C82}
\definecolor{openaiT4}{HTML}{C98A2B}
\definecolor{openaiT13}{HTML}{B5451B}
\begin{figure}[!htbp]\centering
\textbf{OpenAI}\\[3pt]
{\scriptsize models, strongest $\to$ weakest:\quad \tswatch{openaiT0}{Sol max (tier 0, baseline only)}\quad \tswatch{openaiT1}{Sol xhigh (tier 1)}\quad \tswatch{openaiT2}{Sol high (tier 2)}\quad \tswatch{openaiT4}{Sol medium (tier 4)}\quad \tswatch{openaiT13}{Terra low (tier 13)}}\\[7pt]
{\small\itshape where class = casual is served, at every hour (fluid DP, distilled)}\\[3pt]
\spifitbox{\begin{forest}spitree
[{before 01:00?} [\tleaf{openaiT4}{4}{Sol medium}, fill=openaiT4!10, draw=openaiT4!70, line width=0.7pt,edge label={node[midway,left,font=\tiny]{yes}}] [{before 03:00?} [{before 02:00?} [\tleaf{openaiT4}{4}{Sol medium}, fill=openaiT4!10, draw=openaiT4!70, line width=0.7pt,edge label={node[midway,left,font=\tiny]{yes}}] [\tleaf{openaiT13}{13}{Terra low}, fill=openaiT13!10, draw=openaiT13!70, line width=0.7pt,edge label={node[midway,right,font=\tiny]{no}}],edge label={node[midway,left,font=\tiny]{yes}}] [\tleaf{openaiT4}{4}{Sol medium}, fill=openaiT4!10, draw=openaiT4!70, line width=0.7pt,edge label={node[midway,right,font=\tiny]{no}}],edge label={node[midway,right,font=\tiny]{no}}]]
\end{forest}}\\[10pt]
{\small\itshape where class = specialized is served, at every hour (fluid DP, distilled)}\\[3pt]
\spifitbox{\begin{forest}spitree
[{$\hat r_t\le 26\,168$} [{before 03:00?} [{before 14:00?} [\tleaf{openaiT4}{4}{Sol medium}, fill=openaiT4!10, draw=openaiT4!70, line width=0.7pt,edge label={node[midway,left,font=\tiny]{yes}}] [\tleaf{openaiT2}{2}{Sol high}, fill=openaiT2!10, draw=openaiT2!70, line width=0.7pt,edge label={node[midway,right,font=\tiny]{no}}],edge label={node[midway,left,font=\tiny]{yes}}] [\tleaf{openaiT1}{1}{Sol xhigh}, fill=openaiT1!10, draw=openaiT1!70, line width=0.7pt,edge label={node[midway,right,font=\tiny]{no}}],edge label={node[midway,left,font=\tiny]{yes}}] [{before 12:00?} [{before 11:00?} [\tleaf{openaiT1}{1}{Sol xhigh}, fill=openaiT1!10, draw=openaiT1!70, line width=0.7pt,edge label={node[midway,left,font=\tiny]{yes}}] [\tleaf{openaiT4}{4}{Sol medium}, fill=openaiT4!10, draw=openaiT4!70, line width=0.7pt,edge label={node[midway,right,font=\tiny]{no}}],edge label={node[midway,left,font=\tiny]{yes}}] [\tleaf{openaiT1}{1}{Sol xhigh}, fill=openaiT1!10, draw=openaiT1!70, line width=0.7pt,edge label={node[midway,right,font=\tiny]{no}}],edge label={node[midway,right,font=\tiny]{no}}]]
\end{forest}}\\[10pt]
{\small\itshape where class = agentic is served, at every hour (fluid DP, distilled)}\\[3pt]
\spifitbox{\begin{forest}spitree
[{$\hat r_t\le 36\,232$} [{before 07:00?} [\tleaf{openaiT4}{4}{Sol medium}, fill=openaiT4!10, draw=openaiT4!70, line width=0.7pt,edge label={node[midway,left,font=\tiny]{yes}}] [\tleaf{openaiT1}{1}{Sol xhigh}, fill=openaiT1!10, draw=openaiT1!70, line width=0.7pt,edge label={node[midway,right,font=\tiny]{no}}],edge label={node[midway,left,font=\tiny]{yes}}] [\tleaf{openaiT1}{1}{Sol xhigh}, fill=openaiT1!10, draw=openaiT1!70, line width=0.7pt,edge label={node[midway,right,font=\tiny]{no}}]]
\end{forest}}\\[10pt]
{\small\itshape baselines: one tree for every class}\\[3pt]
\spifitbox{\begin{tabular}{c@{\hspace{2.5em}}c@{\hspace{2.5em}}c}
{\scriptsize\itshape reactive threshold} & {\scriptsize\itshape scheduled (peak windows)} & {\scriptsize\itshape always strongest} \\[3pt]
\begin{forest}spitree [{$N_{\mathrm{tot}}> 0.8C$?} [\tleaf{openaiT0}{0}{Sol max}, fill=openaiT0!10, draw=openaiT0!70, line width=0.7pt,edge label={node[midway,left,font=\tiny]{no}}] [\tleaf{openaiT4}{4}{Sol medium}, fill=openaiT4!10, draw=openaiT4!70, line width=0.7pt,edge label={node[midway,right,font=\tiny]{yes}}]] \end{forest} &
\begin{forest}spitree [{peak window ($t\in$ 09:00--12:00, 14:00--18:00)?} [\tleaf{openaiT0}{0}{Sol max}, fill=openaiT0!10, draw=openaiT0!70, line width=0.7pt,edge label={node[midway,left,font=\tiny]{no}}] [\tleaf{openaiT1}{1}{Sol xhigh}, fill=openaiT1!10, draw=openaiT1!70, line width=0.7pt,edge label={node[midway,right,font=\tiny]{yes}}]] \end{forest} &
\begin{forest}spitree [\tleaf{openaiT0}{0}{Sol max}, fill=openaiT0!10, draw=openaiT0!70, line width=0.7pt] \end{forest}
\end{tabular}}
\caption{OpenAI: the distilled routing, one tree per class, and the baseline policies in the same form. Baselines: \emph{reactive threshold} degrades everyone above a backlog trigger (release at $0.5\,C$, $C$ the pool capacity); \emph{scheduled} degrades during the published windows; \emph{always strongest} never degrades. Trees may split on other classes' backlogs, since a shared pool reprices capacity everywhere; splits with identical branches are merged. Model names drop the shared prefix ``GPT-5.6''. DES-estimated rollout value gap of the tree to the DP: 0.23\%.}\label{fig:trees-openai}
\end{figure}

%% file: Figures/latex_out/11_fig_sweep.tex
\begin{figure}[!ht]\centering
\includegraphics[width=.62\textwidth]{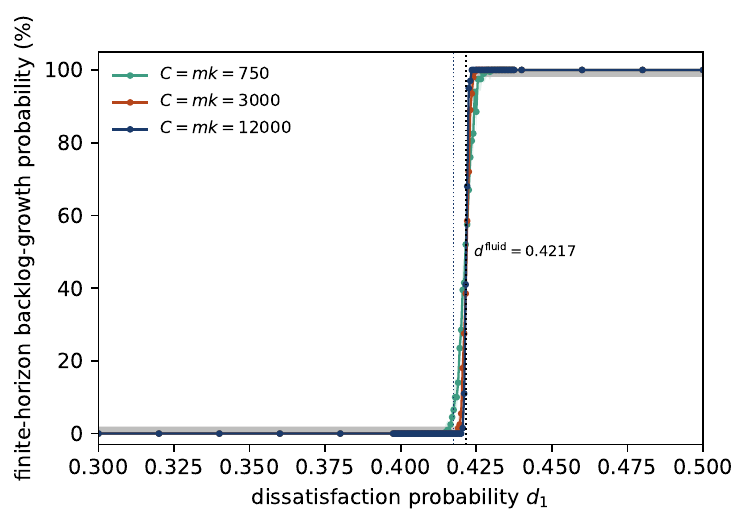}
\caption{Finite-horizon backlog-growth probability at 30\,s against the dissatisfaction probability $d_1$ (single class, Example~\ref{ex:surge} parameters, $N(0)=7C/6$, tier 1 from $t=0$; 200 replications per point, Wilson 95\% CIs). The outcome is $N(30\,\mathrm{s}) > 17C/6$. Dotted verticals mark the fluid stability threshold $d^{\star}=(1-r/(k\mu_1 m))/\rho=0.4175$ and the fluid finite-horizon growth threshold $d^{\mathrm{fluid}}=0.4217$; the stochastic transition is centered near $d^{\mathrm{fluid}}$ and sharpens with $C$, and no irreversibility is implied. Per-$C$ detail in Appendix~\ref{app:sweep}.}\label{fig:sweep}
\end{figure}

%% file: Figures/latex_out/04_fig_trajectories.tex
\begin{figure}[!htbp]\centering
\includegraphics[width=\textwidth]{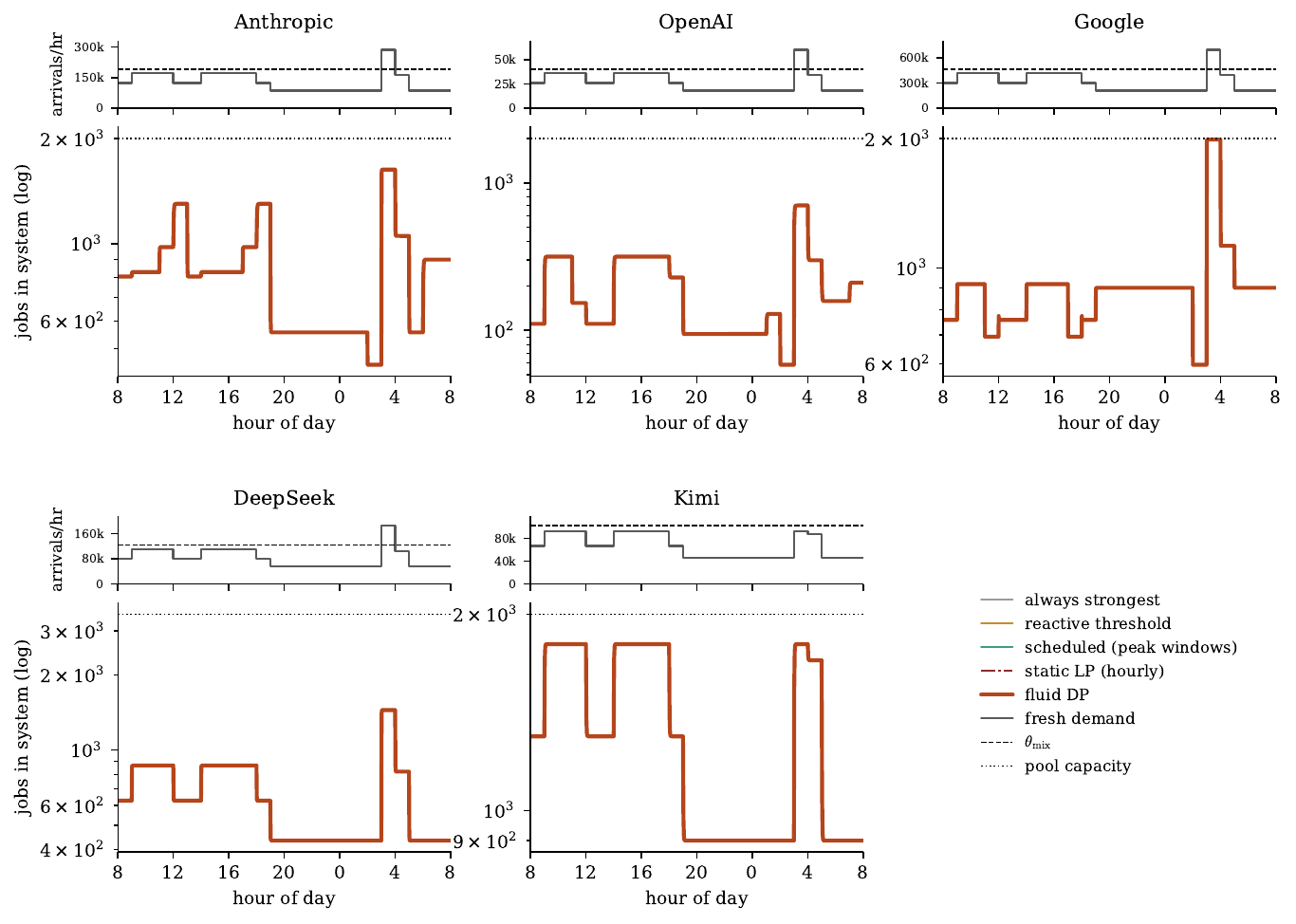}
\caption{Demand and backlog under the policies of Table~\ref{tab:poc-results}, five providers, one shared diurnal profile scaled to each provider's throughput (hour of day; the day is a finite horizon, not a cycle, so the two edges need not match). \emph{Top panels}: fresh demand, with the static-optimal-mix throughput $\theta_{\mathrm{mix}}$ dashed; the morning and afternoon plateaus are the published peak windows, and the overnight spike at 03:00 is the surge. \emph{Bottom panels}: the resulting backlogs (log scale), the closed-form fluid forward pass of \eqref{eq:fluid} under each policy's logged hourly actions; under the DP's own policy the simulation's mean hourly backlogs agree with the forward pass to a maximum relative gap of \spiTrajCheckMax{} over all providers and hours (mean \spiTrajCheckMeanDP). The distilled tree is omitted: visually identical to the fluid DP on every panel. ``static LP (hourly)'' in the legend is the index rule of Section~\ref{sec:poc-protocol}.}\label{fig:poc-traj}
\end{figure}

%% file: Figures/latex_out/09_table_sensitivity.tex
\begin{table}[!htbp]\centering
\caption{One-at-a-time sensitivity of the DES-estimated comparison (Anthropic and Google; demand profile fixed in every row except the peak/$\theta_{\mathrm{mix}}$ provisioning rows). DP cost is the fluid rollout on the perturbed instance (mean and 95\% CI over 20 forecast-noise seeds where applicable); feasibility is the DES criterion of Section~\ref{sec:poc-protocol} (no hour with DES-estimated $\pr(W_{\mathrm{attempt}}>300\,\mathrm{s})>0.05$) evaluated on true demand; best-baseline/DP uses feasible baselines only (NA where none is feasible); index rule (hourly)/DP compares the hourly re-solved index rule against the DP. The index rule was not re-solved on noisy forecasts, so its column reads ``not run'' on the forecast-noise rows ($\sigma$ rows carry no archived index-rule value). The full grid is in Appendix~\ref{app:sens}.}\label{tab:poc-sens}
\footnotesize\begin{tabular}{@{}llrrll@{}}\toprule
provider & perturbation & DP cost & feas. & best baseline/DP & index rule (hourly)/DP \\\midrule
Anthropic & base & \$9.38M & feasible & NA & 124$\times$ \\
Anthropic & $\rho_x\!\uparrow$ (cap 1) & \$5.75M & feasible & NA & 212$\times$ \\
Anthropic & $\rho_x\times 0.75$ & \$16.47M & feasible & NA & 60.9$\times$ \\
Anthropic & $p_r\ell\times 0.5$ & \$7.16M & feasible & NA & 163$\times$ \\
Anthropic & $p_r\ell\times 2$ & \$13.54M & feasible & NA & 86.7$\times$ \\
Anthropic & $C_{\mathrm{SLA}}\times 0.5$ & \$9.38M & feasible & NA & 62.8$\times$ \\
Anthropic & $C_{\mathrm{SLA}}\times 2$ & \$9.38M & feasible & NA & 247$\times$ \\
Anthropic & $\beta_x\times 0.5$ & \$7.47M & feasible & NA & 136$\times$ \\
Anthropic & peak $0.8\,\theta_{\mathrm{mix}}$ & \$5.36M & feasible & NA & 0.99$\times$ \\
Anthropic & peak $1.1\,\theta_{\mathrm{mix}}$ & \$7.16M & feasible & NA & 2.70$\times$ \\
Anthropic & forecast noise $\sigma=0.10$ & \$13.48M [\$7.71M, \$19.25M] & 90\% & NA & not run \\
Anthropic & forecast noise $\sigma=0.25$ & \$15.26M [\$7.11M, \$23.42M] & 90\% & NA & not run \\
Google & base & \$8.90M & feasible & NA & 788$\times$ \\
Google & $\rho_x\!\uparrow$ (cap 1) & \$5.40M & feasible & NA & 1,371$\times$ \\
Google & $\rho_x\times 0.75$ & \$14.88M & feasible & NA & 439$\times$ \\
Google & $p_r\ell\times 0.5$ & \$6.77M & feasible & NA & 1,047$\times$ \\
Google & $p_r\ell\times 2$ & \$12.83M & feasible & NA & 547$\times$ \\
Google & $C_{\mathrm{SLA}}\times 0.5$ & \$8.90M & feasible & NA & 395$\times$ \\
Google & $C_{\mathrm{SLA}}\times 2$ & \$8.90M & feasible & NA & 1,569$\times$ \\
Google & $\beta_x\times 0.5$ & \$8.08M & feasible & NA & 882$\times$ \\
Google & peak $0.8\,\theta_{\mathrm{mix}}$ & \$5.36M & feasible & 1.01$\times$ & 0.98$\times$ \\
Google & peak $1.1\,\theta_{\mathrm{mix}}$ & \$6.86M & feasible & NA & 13.4$\times$ \\
Google & forecast noise $\sigma=0.10$ & \$9.34M [\$9.25M, \$9.43M] & 100\% & NA & not run \\
Google & forecast noise $\sigma=0.25$ & \$9.50M [\$9.35M, \$9.66M] & 100\% & NA & not run \\
\bottomrule\end{tabular}
\end{table}

%% file: Figures/latex_out/06_table_parameters.tex
\begin{table}[t]\centering
\caption{Calibrated instances. Markers give each number's provenance: $^{\mathrm{B}}$ benchmark (Artificial Analysis), $^{\mathrm{P}}$ published by the provider, $^{\mathrm{A}}$ assumption anchored to public prices, $^{\mathrm{C}}$ calibration choice, $^{\mathrm{D}}$ derived, $^{\mathrm{S}}$ full joint-action count $\to$ static-ledger count $\to$ post-screen count; the per-attempt screen of Section~\ref{sec:poc-protocol} removes tiers only on OpenAI (the $221$ s Sol (max), Terra (max) and the $172$ s Luna (max)), so the two stages coincide on the other providers. Symbols as in Table~\ref{tab:notation}, by class $x$ (Section~\ref{sec:who}).}\label{tab:poc-params}
\footnotesize
\begin{tabular}{@{}lrrrL{5.6cm}@{}}\toprule
\multicolumn{5}{@{}l}{\emph{customer classes (shared across providers)}}\\
 & casual & special. & agentic & \\\midrule
share$^{\mathrm{A}}$ & 0.60 & 0.25 & 0.15 & share of fresh demand \\
$\rho_x$$^{\mathrm{A}}$ & 0.30 & 0.85 & 0.98 & probability an unsatisfied user re-asks \\
$d_{0,x}$$^{\mathrm{A}}$ & 0.02 & 0.05 & 0.10 & dissatisfaction at the strongest tier \\
$\beta_x$$^{\mathrm{A}}$ & 0.8 & 2.5 & 3.0 & quality sensitivity: $d_j(x)=\min\{0.95,\,d_{0,x}+\beta_x(1-I_j/I_0)\}$ \\
$p_r$$^{\mathrm{A}}$ & 0.03 & 0.08 & 0.05 & churn probability upon abandonment \\
$\ell$$^{\mathrm{A}}$ & \$240 & \$1,800 & \$2,000 & lifetime margin of a churned customer \\
$p_r\ell$$^{\mathrm{D}}$ & \$7.2 & \$144 & \$100 & churn charge per abandonment, $p_r\ell$ \\
\bottomrule\end{tabular}\\[5pt]
\begin{tabular}{@{}lrrrrL{4.6cm}@{}}\toprule
\multicolumn{6}{@{}l}{\emph{provider instances (one per regime of the theory)}}\\
 & tiers$^{\mathrm{B}}$ & pool$^{\mathrm{C}}$ & $\theta_{\mathrm{mix}}$/hr$^{\mathrm{D}}$ & actions$^{\mathrm{D},\mathrm{S}}$ & regime \\\midrule
Anthropic & 5 & 2\,000 & 190\,943 & 125$\to$9 & steep reasoning-effort ladder \\
OpenAI & 14 & 2\,000 & 40\,258 & 2744$\to$72$\to$50 & 14-variant lattice, tiers split across the two frontiers \\
Google & 4 & 2\,000 & 462\,784 & 64$\to$27 & fast ladder: degradation buys throughput, not energy \\
DeepSeek & 3 & 3\,500$^{\mathrm{P}}$ & 123\,362 & 27$\to$4 & published trap tier, clock policy, per-tier pools \\
Kimi & 3 & 2\,000 & 102\,576 & 27$\to$1 & flat ladder: no lever; the policy space is a point \\
\bottomrule\end{tabular}\\[5pt]
\begin{tabular}{@{}lL{9.6cm}@{}}\toprule
\multicolumn{2}{@{}l}{\emph{shared calibration}}\\\midrule
tier physics$^{\mathrm{B}}$ & service time $\ex[S_j]$, quality index $I_j$, cost per task $c_j$, per tier \\
demand$^{\mathrm{C}}$ & daily profile peaking at $0.90\,\theta_{\mathrm{mix}}$; overnight surge at 03:00, sized $\min(1.5\,\theta_{\mathrm{mix}},\,0.9\,\theta_{\max})$ \\
peak windows$^{\mathrm{P}}$ & 09--12 and 14--18; electricity price $\times 2$ inside them \\
service level$^{\mathrm{C}}$ & $\pr(W_{\mathrm{attempt}}>300\mathrm{s})\le 0.05$ per attempt; necessary screen at $\ex[S_j]\le 300/\ln 20 = 100.14$\,s plus the heuristic surrogate of Section~\ref{sec:shadow}; verified by DES \\
cost scales$^{\mathrm{C}}$ & markup $3\times$; $C_{\mathrm{SLA}}$ = one peak hour of tier-0 revenue; SLA elasticity $\beta_{\mathrm{SLA}}=2$; memory exponent $\kappa=1.3$; idle draw $30\%$ of tier-0 active \\
\bottomrule\end{tabular}
\end{table}

%% file: Figures/latex_out/03_trees_google.tex
\providecommand{\tleaf}[3]{\begin{tabular}{@{}c@{}}{\scriptsize\textbf{\textcolor{#1}{#3}}}\\[-2.5pt]{\tiny\textcolor{black!55}{tier #2}}\end{tabular}}
\providecommand{\tswatch}[2]{\textcolor{#1}{\rule[-0.25ex]{0.9em}{1.5ex}}\,{\scriptsize #2}}
\providecommand{\spifitbox}[1]{\sbox0{#1}\ifdim\wd0>\linewidth\resizebox{\linewidth}{!}{\usebox0}\else\usebox0\fi}
\forestset{spitree/.style={for tree={draw=black!30, rounded corners=1.5pt, edge={black!45,thin}, l sep=7pt, s sep=4pt, inner sep=2.5pt, font=\scriptsize, align=center}}}
\definecolor{googleT0}{HTML}{1A3A6B}
\definecolor{googleT1}{HTML}{C98A2B}
\definecolor{googleT3}{HTML}{5B3E8C}
\begin{figure}[!ht]\centering
\textbf{Google}\\[3pt]
{\scriptsize models, strongest $\to$ weakest:\quad \tswatch{googleT0}{3.7 Flash high (tier 0)}\quad \tswatch{googleT1}{3.7 Flash medium (tier 1)}\quad \tswatch{googleT3}{3.7 Flash low (tier 3)}}\\[7pt]
{\small\itshape where class = casual is served, at every hour (fluid DP, distilled)}\\[3pt]
\spifitbox{\begin{forest}spitree
[{$\hat r_t\le 393\,366$} [{$\hat r_t\le 208\,253$} [{$N_{\mathrm{ag}}\le 35$} [\tleaf{googleT1}{1}{3.7 Flash medium}, fill=googleT1!10, draw=googleT1!70, line width=0.7pt,edge label={node[midway,left,font=\tiny]{yes}}] [\tleaf{googleT0}{0}{3.7 Flash high}, fill=googleT0!10, draw=googleT0!70, line width=0.7pt,edge label={node[midway,right,font=\tiny]{no}}],edge label={node[midway,left,font=\tiny]{yes}}] [\tleaf{googleT1}{1}{3.7 Flash medium}, fill=googleT1!10, draw=googleT1!70, line width=0.7pt,edge label={node[midway,right,font=\tiny]{no}}],edge label={node[midway,left,font=\tiny]{yes}}] [\tleaf{googleT3}{3}{3.7 Flash low}, fill=googleT3!10, draw=googleT3!70, line width=0.7pt,edge label={node[midway,right,font=\tiny]{no}}]]
\end{forest}}\\[10pt]
{\small\itshape where class = specialized is served, at every hour (fluid DP, distilled)}\\[3pt]
\spifitbox{\begin{forest}spitree
[{$N_{\mathrm{spec}}\le 59$} [{$N_{\mathrm{tot}}\le 742$} [{$N_{\mathrm{tot}}\le 550$} [\tleaf{googleT0}{0}{3.7 Flash high}, fill=googleT0!10, draw=googleT0!70, line width=0.7pt,edge label={node[midway,left,font=\tiny]{yes}}] [\tleaf{googleT1}{1}{3.7 Flash medium}, fill=googleT1!10, draw=googleT1!70, line width=0.7pt,edge label={node[midway,right,font=\tiny]{no}}],edge label={node[midway,left,font=\tiny]{yes}}] [\tleaf{googleT3}{3}{3.7 Flash low}, fill=googleT3!10, draw=googleT3!70, line width=0.7pt,edge label={node[midway,right,font=\tiny]{no}}],edge label={node[midway,left,font=\tiny]{yes}}] [{$N_{\mathrm{cas}}\le 909$} [\tleaf{googleT0}{0}{3.7 Flash high}, fill=googleT0!10, draw=googleT0!70, line width=0.7pt,edge label={node[midway,left,font=\tiny]{yes}}] [{$N_{\mathrm{spec}}\le 190$} [\tleaf{googleT1}{1}{3.7 Flash medium}, fill=googleT1!10, draw=googleT1!70, line width=0.7pt,edge label={node[midway,left,font=\tiny]{yes}}] [\tleaf{googleT0}{0}{3.7 Flash high}, fill=googleT0!10, draw=googleT0!70, line width=0.7pt,edge label={node[midway,right,font=\tiny]{no}}],edge label={node[midway,right,font=\tiny]{no}}],edge label={node[midway,right,font=\tiny]{no}}]]
\end{forest}}\\[10pt]
{\small\itshape where class = agentic is served, at every hour (fluid DP, distilled)}\\[3pt]
\spifitbox{\begin{forest}spitree
[{before 18:00?} [{peak window ($t\in$ 09:00--12:00, 14:00--18:00)?} [{$N_{\mathrm{cas}}\le 456$} [\tleaf{googleT1}{1}{3.7 Flash medium}, fill=googleT1!10, draw=googleT1!70, line width=0.7pt,edge label={node[midway,left,font=\tiny]{yes}}] [\tleaf{googleT3}{3}{3.7 Flash low}, fill=googleT3!10, draw=googleT3!70, line width=0.7pt,edge label={node[midway,right,font=\tiny]{no}}],edge label={node[midway,left,font=\tiny]{no}}] [{before 17:00?} [\tleaf{googleT0}{0}{3.7 Flash high}, fill=googleT0!10, draw=googleT0!70, line width=0.7pt,edge label={node[midway,left,font=\tiny]{yes}}] [\tleaf{googleT3}{3}{3.7 Flash low}, fill=googleT3!10, draw=googleT3!70, line width=0.7pt,edge label={node[midway,right,font=\tiny]{no}}],edge label={node[midway,right,font=\tiny]{yes}}],edge label={node[midway,left,font=\tiny]{yes}}] [{$N_{\mathrm{ag}}\le 35$} [{$N_{\mathrm{tot}}\le 1\,003$} [\tleaf{googleT1}{1}{3.7 Flash medium}, fill=googleT1!10, draw=googleT1!70, line width=0.7pt,edge label={node[midway,left,font=\tiny]{yes}}] [\tleaf{googleT3}{3}{3.7 Flash low}, fill=googleT3!10, draw=googleT3!70, line width=0.7pt,edge label={node[midway,right,font=\tiny]{no}}],edge label={node[midway,left,font=\tiny]{yes}}] [{$N_{\mathrm{spec}}\le 59$} [\tleaf{googleT1}{1}{3.7 Flash medium}, fill=googleT1!10, draw=googleT1!70, line width=0.7pt,edge label={node[midway,left,font=\tiny]{yes}}] [\tleaf{googleT0}{0}{3.7 Flash high}, fill=googleT0!10, draw=googleT0!70, line width=0.7pt,edge label={node[midway,right,font=\tiny]{no}}],edge label={node[midway,right,font=\tiny]{no}}],edge label={node[midway,right,font=\tiny]{no}}]]
\end{forest}}\\[10pt]
{\small\itshape baselines: one tree for every class}\\[3pt]
\spifitbox{\begin{tabular}{c@{\hspace{2.5em}}c@{\hspace{2.5em}}c}
{\scriptsize\itshape reactive threshold} & {\scriptsize\itshape scheduled (peak windows)} & {\scriptsize\itshape always strongest} \\[3pt]
\begin{forest}spitree [{$N_{\mathrm{tot}}> 1.0C$?} [\tleaf{googleT0}{0}{3.7 Flash high}, fill=googleT0!10, draw=googleT0!70, line width=0.7pt,edge label={node[midway,left,font=\tiny]{no}}] [\tleaf{googleT3}{3}{3.7 Flash low}, fill=googleT3!10, draw=googleT3!70, line width=0.7pt,edge label={node[midway,right,font=\tiny]{yes}}]] \end{forest} &
\begin{forest}spitree [{peak window ($t\in$ 09:00--12:00, 14:00--18:00)?} [\tleaf{googleT0}{0}{3.7 Flash high}, fill=googleT0!10, draw=googleT0!70, line width=0.7pt,edge label={node[midway,left,font=\tiny]{no}}] [\tleaf{googleT1}{1}{3.7 Flash medium}, fill=googleT1!10, draw=googleT1!70, line width=0.7pt,edge label={node[midway,right,font=\tiny]{yes}}]] \end{forest} &
\begin{forest}spitree [\tleaf{googleT0}{0}{3.7 Flash high}, fill=googleT0!10, draw=googleT0!70, line width=0.7pt] \end{forest}
\end{tabular}}
\caption{Google: the distilled routing, one tree per class, and the baseline policies in the same form. Baselines: \emph{reactive threshold} degrades everyone above a backlog trigger (release at $0.5\,C$, $C$ the pool capacity); \emph{scheduled} degrades during the published windows; \emph{always strongest} never degrades. Trees may split on other classes' backlogs, since a shared pool reprices capacity everywhere; splits with identical branches are merged. Model names drop the shared prefix ``Gemini''. DES-estimated rollout value gap of the tree to the DP: -3.70\%.}\label{fig:trees-google}
\end{figure}

%% file: Figures/latex_out/03_trees_deepseek.tex
\providecommand{\tleaf}[3]{\begin{tabular}{@{}c@{}}{\scriptsize\textbf{\textcolor{#1}{#3}}}\\[-2.5pt]{\tiny\textcolor{black!55}{tier #2}}\end{tabular}}
\providecommand{\tswatch}[2]{\textcolor{#1}{\rule[-0.25ex]{0.9em}{1.5ex}}\,{\scriptsize #2}}
\providecommand{\spifitbox}[1]{\sbox0{#1}\ifdim\wd0>\linewidth\resizebox{\linewidth}{!}{\usebox0}\else\usebox0\fi}
\forestset{spitree/.style={for tree={draw=black!30, rounded corners=1.5pt, edge={black!45,thin}, l sep=7pt, s sep=4pt, inner sep=2.5pt, font=\scriptsize, align=center}}}
\definecolor{deepseekT0}{HTML}{1A3A6B}
\definecolor{deepseekT1}{HTML}{2D6B3F}
\begin{figure}[!ht]\centering
\textbf{DeepSeek}\\[3pt]
{\scriptsize models, strongest $\to$ weakest:\quad \tswatch{deepseekT0}{V4-Pro-0813 thinking (tier 0)}\quad \tswatch{deepseekT1}{V4-Flash-0731 thinking (tier 1)}}\\[7pt]
{\small\itshape where class = casual is served, at every hour (fluid DP, distilled)}\\[3pt]
\spifitbox{\begin{forest}spitree
[\tleaf{deepseekT1}{1}{V4-Flash-0731 thinking}, fill=deepseekT1!10, draw=deepseekT1!70, line width=0.7pt]
\end{forest}}\\[10pt]
{\small\itshape where class = specialized is served, at every hour (fluid DP, distilled)}\\[3pt]
\spifitbox{\begin{forest}spitree
[{$N_{\mathrm{ag}}\le 61$} [{$N_{\mathrm{cas}}\le 799$} [\tleaf{deepseekT0}{0}{V4-Pro-0813 thinking}, fill=deepseekT0!10, draw=deepseekT0!70, line width=0.7pt,edge label={node[midway,left,font=\tiny]{yes}}] [{$\hat r_t\le 80\,185$} [\tleaf{deepseekT0}{0}{V4-Pro-0813 thinking}, fill=deepseekT0!10, draw=deepseekT0!70, line width=0.7pt,edge label={node[midway,left,font=\tiny]{yes}}] [\tleaf{deepseekT1}{1}{V4-Flash-0731 thinking}, fill=deepseekT1!10, draw=deepseekT1!70, line width=0.7pt,edge label={node[midway,right,font=\tiny]{no}}],edge label={node[midway,right,font=\tiny]{no}}],edge label={node[midway,left,font=\tiny]{yes}}] [\tleaf{deepseekT0}{0}{V4-Pro-0813 thinking}, fill=deepseekT0!10, draw=deepseekT0!70, line width=0.7pt,edge label={node[midway,right,font=\tiny]{no}}]]
\end{forest}}\\[10pt]
{\small\itshape where class = agentic is served, at every hour (fluid DP, distilled)}\\[3pt]
\spifitbox{\begin{forest}spitree
[{$N_{\mathrm{ag}}\le 61$} [{$N_{\mathrm{cas}}\le 799$} [\tleaf{deepseekT1}{1}{V4-Flash-0731 thinking}, fill=deepseekT1!10, draw=deepseekT1!70, line width=0.7pt,edge label={node[midway,left,font=\tiny]{yes}}] [{$\hat r_t\le 80\,185$} [\tleaf{deepseekT1}{1}{V4-Flash-0731 thinking}, fill=deepseekT1!10, draw=deepseekT1!70, line width=0.7pt,edge label={node[midway,left,font=\tiny]{yes}}] [\tleaf{deepseekT0}{0}{V4-Pro-0813 thinking}, fill=deepseekT0!10, draw=deepseekT0!70, line width=0.7pt,edge label={node[midway,right,font=\tiny]{no}}],edge label={node[midway,right,font=\tiny]{no}}],edge label={node[midway,left,font=\tiny]{yes}}] [\tleaf{deepseekT1}{1}{V4-Flash-0731 thinking}, fill=deepseekT1!10, draw=deepseekT1!70, line width=0.7pt,edge label={node[midway,right,font=\tiny]{no}}]]
\end{forest}}\\[10pt]
{\small\itshape baselines: one tree for every class}\\[3pt]
\spifitbox{\begin{tabular}{c@{\hspace{2.5em}}c@{\hspace{2.5em}}c}
{\scriptsize\itshape reactive threshold} & {\scriptsize\itshape scheduled (peak windows)} & {\scriptsize\itshape always strongest} \\[3pt]
\begin{forest}spitree [{$N_{\mathrm{tot}}> 0.8C$?} [\tleaf{deepseekT0}{0}{V4-Pro-0813 thinking}, fill=deepseekT0!10, draw=deepseekT0!70, line width=0.7pt,edge label={node[midway,left,font=\tiny]{no}}] [\tleaf{deepseekT1}{1}{V4-Flash-0731 thinking}, fill=deepseekT1!10, draw=deepseekT1!70, line width=0.7pt,edge label={node[midway,right,font=\tiny]{yes}}]] \end{forest} &
\begin{forest}spitree [{peak window ($t\in$ 09:00--12:00, 14:00--18:00)?} [\tleaf{deepseekT0}{0}{V4-Pro-0813 thinking}, fill=deepseekT0!10, draw=deepseekT0!70, line width=0.7pt,edge label={node[midway,left,font=\tiny]{no}}] [\tleaf{deepseekT1}{1}{V4-Flash-0731 thinking}, fill=deepseekT1!10, draw=deepseekT1!70, line width=0.7pt,edge label={node[midway,right,font=\tiny]{yes}}]] \end{forest} &
\begin{forest}spitree [\tleaf{deepseekT0}{0}{V4-Pro-0813 thinking}, fill=deepseekT0!10, draw=deepseekT0!70, line width=0.7pt] \end{forest}
\end{tabular}}
\caption{DeepSeek: the distilled routing, one tree per class, and the baseline policies in the same form. Baselines: \emph{reactive threshold} degrades everyone above a backlog trigger (release at $0.5\,C$, $C$ the pool capacity); \emph{scheduled} degrades during the published windows; \emph{always strongest} never degrades. Trees may split on other classes' backlogs, since a shared pool reprices capacity everywhere; splits with identical branches are merged. DES-estimated rollout value gap of the tree to the DP: 0.00\%.}\label{fig:trees-deepseek}
\end{figure}

%% file: Figures/latex_out/03_trees_kimi.tex
\providecommand{\tleaf}[3]{\begin{tabular}{@{}c@{}}{\scriptsize\textbf{\textcolor{#1}{#3}}}\\[-2.5pt]{\tiny\textcolor{black!55}{tier #2}}\end{tabular}}
\providecommand{\tswatch}[2]{\textcolor{#1}{\rule[-0.25ex]{0.9em}{1.5ex}}\,{\scriptsize #2}}
\providecommand{\spifitbox}[1]{\sbox0{#1}\ifdim\wd0>\linewidth\resizebox{\linewidth}{!}{\usebox0}\else\usebox0\fi}
\forestset{spitree/.style={for tree={draw=black!30, rounded corners=1.5pt, edge={black!45,thin}, l sep=7pt, s sep=4pt, inner sep=2.5pt, font=\scriptsize, align=center}}}
\definecolor{kimiT0}{HTML}{1A3A6B}
\definecolor{kimiT1}{HTML}{2D6B3F}
\begin{figure}[!ht]\centering
\textbf{Kimi}\\[3pt]
{\scriptsize models, strongest $\to$ weakest:\quad \tswatch{kimiT0}{K3 max (tier 0)}\quad \tswatch{kimiT1}{K3 low (tier 1)}}\\[7pt]
{\small\itshape where class = casual is served, at every hour (fluid DP, distilled)}\\[3pt]
\spifitbox{\begin{forest}spitree
[\tleaf{kimiT0}{0}{K3 max}, fill=kimiT0!10, draw=kimiT0!70, line width=0.7pt]
\end{forest}}\\[10pt]
{\small\itshape where class = specialized is served, at every hour (fluid DP, distilled)}\\[3pt]
\spifitbox{\begin{forest}spitree
[\tleaf{kimiT0}{0}{K3 max}, fill=kimiT0!10, draw=kimiT0!70, line width=0.7pt]
\end{forest}}\\[10pt]
{\small\itshape where class = agentic is served, at every hour (fluid DP, distilled)}\\[3pt]
\spifitbox{\begin{forest}spitree
[\tleaf{kimiT0}{0}{K3 max}, fill=kimiT0!10, draw=kimiT0!70, line width=0.7pt]
\end{forest}}\\[10pt]
{\small\itshape baselines: one tree for every class}\\[3pt]
\spifitbox{\begin{tabular}{c@{\hspace{2.5em}}c@{\hspace{2.5em}}c}
{\scriptsize\itshape reactive threshold} & {\scriptsize\itshape scheduled (peak windows)} & {\scriptsize\itshape always strongest} \\[3pt]
\begin{forest}spitree [{$N_{\mathrm{tot}}> 1.0C$?} [\tleaf{kimiT0}{0}{K3 max}, fill=kimiT0!10, draw=kimiT0!70, line width=0.7pt,edge label={node[midway,left,font=\tiny]{no}}] [\tleaf{kimiT1}{1}{K3 low}, fill=kimiT1!10, draw=kimiT1!70, line width=0.7pt,edge label={node[midway,right,font=\tiny]{yes}}]] \end{forest} &
\begin{forest}spitree [{peak window ($t\in$ 09:00--12:00, 14:00--18:00)?} [\tleaf{kimiT0}{0}{K3 max}, fill=kimiT0!10, draw=kimiT0!70, line width=0.7pt,edge label={node[midway,left,font=\tiny]{no}}] [\tleaf{kimiT1}{1}{K3 low}, fill=kimiT1!10, draw=kimiT1!70, line width=0.7pt,edge label={node[midway,right,font=\tiny]{yes}}]] \end{forest} &
\begin{forest}spitree [\tleaf{kimiT0}{0}{K3 max}, fill=kimiT0!10, draw=kimiT0!70, line width=0.7pt] \end{forest}
\end{tabular}}
\caption{Kimi: the distilled routing, one tree per class, and the baseline policies in the same form. Baselines: \emph{reactive threshold} degrades everyone above a backlog trigger (release at $0.5\,C$, $C$ the pool capacity); \emph{scheduled} degrades during the published windows; \emph{always strongest} never degrades. Trees may split on other classes' backlogs, since a shared pool reprices capacity everywhere; splits with identical branches are merged. Model names drop the shared prefix ``Kimi''. DES-estimated rollout value gap of the tree to the DP: 0.00\%.}\label{fig:trees-kimi}
\end{figure}

%% file: Figures/latex_out/07_table_qregret.tex
\begin{table}[t]\centering
\caption{Pointwise Q-regret of the fluid policy under the converged Q-function, on mutually feasible state--action pairs (uniform mean and p95, occupancy-weighted mean), with the feasibility-disagreement rate under the heuristic surrogate. Tails concentrate off-trajectory in deep-saturation grid states on the shared-pool, high-$\mu$ providers; the deployment-relevant comparison is the DES-estimated policy-value gap of Table~\ref{tab:poc-cert}.}\label{tab:poc-qregret}
\small\begin{tabular}{@{}lrrrr@{}}\toprule
provider & mean & p95 & occ.\ mean & feas.\ disagr. \\\midrule
Anthropic & 1.83\% & 7.79\% & 3.33\% & 5.77\% \\
OpenAI & 1.04\% & 3.98\% & 1.56\% & 0.00\% \\
Google & 8.05\% & 39.53\% & 11.66\% & 1.51\% \\
DeepSeek & 1.23\% & 1.83\% & 1.39\% & 0.00\% \\
Kimi & 0.00\% & 0.00\% & 0.00\% & 0.00\% \\
\bottomrule\end{tabular}
\end{table}

%% file: Figures/latex_out/10_table_des_cis.tex
\begin{table}[t]\centering
\caption{DES-estimated 24-hour cost: mean and 95\% CI over 30 replications per policy and provider, with the number of hours in which the hourly point estimate of $\pr(W_{\mathrm{attempt}}>300\,\mathrm{s})$ exceeds $0.05$. $^\dagger$\,div: the DES-estimated cost exceeds the $1000\times$ reporting cutoff under the policy's own retry feedback, so no finite estimate is reported, the same convention as Table~\ref{tab:poc-results}. Implementation: batched discrete-event simulation at per-provider step $\Delta t$ with $\mu_{\max}\Delta t \le 0.1$ (Anthropic $0.996$ s, DeepSeek $1$ s, Google $0.241$ s, Kimi $1$ s, OpenAI $0.679$ s), per-class Poisson fresh-arrival streams, exponential services by per-job thinning, instantaneous retries, per-class FCFS queues with wait-age buckets, and the fluid model's proportional capacity shares (per-tier pools where published), all 30 replications simulated in parallel.}\label{tab:des-cis}
\footnotesize\begin{tabular}{@{}llrrr@{}}\toprule
provider & policy & mean & 95\% CI & hours $>0.05$ \\\midrule
Anthropic & always strongest & {\color{gray}\emph{div}}$^{\dagger}$ & {\color{gray}\emph{div}}$^{\dagger}$ & 24 \\
Anthropic & reactive threshold & \$991.88M & [\$984.08M, \$999.69M] & 9 \\
Anthropic & scheduled (peak windows) & {\color{gray}\emph{div}}$^{\dagger}$ & {\color{gray}\emph{div}}$^{\dagger}$ & 19 \\
Anthropic & index rule (hourly) & \$1.17B & [\$1.16B, \$1.18B] & 3 \\
Anthropic & fluid DP & \$9.40M & [\$9.39M, \$9.41M] & 0 \\
Anthropic & distilled tree & \$9.44M & [\$9.44M, \$9.45M] & 0 \\
OpenAI & always strongest & \$117.70M & [\$116.10M, \$119.30M] & 24 \\
OpenAI & reactive threshold & \$8.73M & [\$8.66M, \$8.79M] & 21 \\
OpenAI & scheduled (peak windows) & \$61.84M & [\$61.36M, \$62.31M] & 17 \\
OpenAI & index rule (hourly) & \$1.59M & [\$1.59M, \$1.59M] & 0 \\
OpenAI & fluid DP & \$1.78M & [\$1.78M, \$1.78M] & 0 \\
OpenAI & distilled tree & \$1.78M & [\$1.78M, \$1.78M] & 0 \\
Google & always strongest & \$7.26B & [\$7.23B, \$7.29B] & 3 \\
Google & reactive threshold & \$1.75B & [\$1.74B, \$1.76B] & 2 \\
Google & scheduled (peak windows) & \$7.26B & [\$7.22B, \$7.29B] & 3 \\
Google & index rule (hourly) & \$7.26B & [\$7.23B, \$7.29B] & 3 \\
Google & fluid DP & \$9.21M & [\$9.10M, \$9.31M] & 0 \\
Google & distilled tree & \$8.87M & [\$8.87M, \$8.87M] & 0 \\
DeepSeek & always strongest & {\color{gray}\emph{div}}$^{\dagger}$ & {\color{gray}\emph{div}}$^{\dagger}$ & 24 \\
DeepSeek & reactive threshold & \$27.87M & [\$27.79M, \$27.95M] & 23 \\
DeepSeek & scheduled (peak windows) & {\color{gray}\emph{div}}$^{\dagger}$ & {\color{gray}\emph{div}}$^{\dagger}$ & 19 \\
DeepSeek & index rule (hourly) & \$2.03M & [\$2.03M, \$2.03M] & 0 \\
DeepSeek & fluid DP & \$2.03M & [\$2.03M, \$2.03M] & 0 \\
DeepSeek & distilled tree & \$2.03M & [\$2.03M, \$2.03M] & 0 \\
Kimi & always strongest & \$3.44M & [\$3.44M, \$3.44M] & 0 \\
Kimi & reactive threshold & \$3.44M & [\$3.44M, \$3.44M] & 0 \\
Kimi & scheduled (peak windows) & \$1.62B & [\$1.60B, \$1.63B] & 12 \\
Kimi & index rule (hourly) & \$3.44M & [\$3.44M, \$3.44M] & 0 \\
Kimi & fluid DP & \$3.44M & [\$3.44M, \$3.44M] & 0 \\
Kimi & distilled tree & \$3.44M & [\$3.44M, \$3.44M] & 0 \\
\bottomrule\end{tabular}
\end{table}

%% file: Figures/latex_out/13_table_sens_full.tex
\begin{table}[t]\centering
\caption{Full one-at-a-time sensitivity grid (Anthropic and Google): DES-estimated 24-hour mean cost by policy, the DP's fluid rollout cost, and the best feasible baseline/DP ratio (feasibility by the Table~\ref{tab:des-cis} criterion: no hour with DES-estimated $\pr(W_{\mathrm{attempt}}>300\,\mathrm{s})>0.05$). The base row is taken from the same 30-replication run as Table~\ref{tab:des-cis}; the deterministic perturbation rows use their own 30-replication DES sets and the two forecast-noise rows use 20 seeds, so base and perturbed cells differ by seeds as well as by the perturbation. $^\dagger$\,As in Table~\ref{tab:des-cis}: the cost exceeds $1000\times$ the DP's mean, so no finite estimate is reported. The index rule was not re-solved on noisy forecasts, so its column reads ``not run'' on the forecast-noise rows.}\label{tab:sens-full}
\scriptsize\resizebox{\textwidth}{!}{\begin{tabular}{@{}llrccccc l@{}}\toprule
provider & cell & DP cost & P0 & P1 & P2 & P5 & DP & best base/DP \\\midrule
Anthropic & base & \$9.38M & {\color{gray}\emph{div}}$^{\dagger}$ & \$991.88M & {\color{gray}\emph{div}}$^{\dagger}$ & \$1.17B & \$9.40M & NA \\
Anthropic & $\rho_x\!\uparrow$ (cap 1) & \$5.75M & {\color{gray}\emph{div}}$^{\dagger}$ & \$1.44B & {\color{gray}\emph{div}}$^{\dagger}$ & \$1.22B & \$5.76M & NA \\
Anthropic & $\rho_x\times 0.75$ & \$16.47M & {\color{gray}\emph{div}}$^{\dagger}$ & \$1.05B & {\color{gray}\emph{div}}$^{\dagger}$ & \$1.00B & \$16.48M & NA \\
Anthropic & $p_r\ell\times 0.5$ & \$7.16M & {\color{gray}\emph{div}}$^{\dagger}$ & \$976.77M & {\color{gray}\emph{div}}$^{\dagger}$ & \$1.17B & \$7.17M & NA \\
Anthropic & $p_r\ell\times 2$ & \$13.54M & {\color{gray}\emph{div}}$^{\dagger}$ & \$1.06B & {\color{gray}\emph{div}}$^{\dagger}$ & \$1.17B & \$13.53M & NA \\
Anthropic & $C_{\mathrm{SLA}}\times 0.5$ & \$9.38M & {\color{gray}\emph{div}}$^{\dagger}$ & \$538.03M & {\color{gray}\emph{div}}$^{\dagger}$ & \$591.51M & \$9.42M & NA \\
Anthropic & $C_{\mathrm{SLA}}\times 2$ & \$9.38M & {\color{gray}\emph{div}}$^{\dagger}$ & \$1.94B & {\color{gray}\emph{div}}$^{\dagger}$ & \$2.33B & \$9.42M & NA \\
Anthropic & $\beta_x\times 0.5$ & \$7.47M & {\color{gray}\emph{div}}$^{\dagger}$ & \$1.07B & {\color{gray}\emph{div}}$^{\dagger}$ & \$1.02B & \$7.48M & NA \\
Anthropic & peak $0.8\,\theta_{\mathrm{mix}}$ & \$5.36M & {\color{gray}\emph{div}}$^{\dagger}$ & \$471.38M & {\color{gray}\emph{div}}$^{\dagger}$ & \$5.31M & \$5.36M & NA \\
Anthropic & peak $1.1\,\theta_{\mathrm{mix}}$ & \$7.16M & {\color{gray}\emph{div}}$^{\dagger}$ & \$2.26B & {\color{gray}\emph{div}}$^{\dagger}$ & \$19.26M & \$7.14M & NA \\
Anthropic & forecast noise $\sigma=0.10$ & \$13.48M [\$7.71M, \$19.25M] & \multicolumn{5}{c}{20 DES seeds; DP feas. 90\%; index rule not run} & NA \\
Anthropic & forecast noise $\sigma=0.25$ & \$15.26M [\$7.11M, \$23.42M] & \multicolumn{5}{c}{20 DES seeds; DP feas. 90\%; index rule not run} & NA \\
Google & base & \$8.90M & \$7.26B & \$1.75B & \$7.26B & \$7.26B & \$9.21M & NA \\
Google & $\rho_x\!\uparrow$ (cap 1) & \$5.40M & \$7.42B & \$1.76B & \$7.41B & \$7.41B & \$5.41M & NA \\
Google & $\rho_x\times 0.75$ & \$14.88M & \$6.86B & \$1.70B & \$6.87B & \$6.86B & \$15.64M & NA \\
Google & $p_r\ell\times 0.5$ & \$6.77M & \$7.24B & \$1.74B & \$7.24B & \$7.24B & \$6.91M & NA \\
Google & $p_r\ell\times 2$ & \$12.83M & \$7.24B & \$1.75B & \$7.25B & \$7.24B & \$13.24M & NA \\
Google & $C_{\mathrm{SLA}}\times 0.5$ & \$8.90M & \$3.64B & \$881.77M & \$3.64B & \$3.64B & \$9.20M & NA \\
Google & $C_{\mathrm{SLA}}\times 2$ & \$8.90M & \$14.45B & \$3.47B & \$14.45B & \$14.44B & \$9.20M & NA \\
Google & $\beta_x\times 0.5$ & \$8.08M & \$7.24B & \$1.74B & \$7.24B & \$7.24B & \$8.20M & NA \\
Google & peak $0.8\,\theta_{\mathrm{mix}}$ & \$5.36M & \$5.42M & \$5.55M & \$6.68M & \$5.26M & \$5.36M & 1.01$\times$ \\
Google & peak $1.1\,\theta_{\mathrm{mix}}$ & \$6.86M & \$92.14M & \$86.50M & \$93.92M & \$92.13M & \$6.86M & NA \\
Google & forecast noise $\sigma=0.10$ & \$9.34M [\$9.25M, \$9.43M] & \multicolumn{5}{c}{20 DES seeds; DP feas. 100\%; index rule not run} & NA \\
Google & forecast noise $\sigma=0.25$ & \$9.50M [\$9.35M, \$9.66M] & \multicolumn{5}{c}{20 DES seeds; DP feas. 100\%; index rule not run} & NA \\
\bottomrule\end{tabular}}
\end{table}

%% file: Figures/latex_out/12_table_sweep_750.tex
\begin{table}[p]\centering
\caption{Finite-horizon backlog-growth probability, $C=mk=750$ (200 replications per point; Wilson 95\% CIs; fluid threshold $d^\star=0.4175$; fluid growth threshold $d^\mathrm{fluid}=0.4217$; DES step $\mu_1\Delta t=0.083$). Grid: coarse spacing $0.02$ plus refined rows at spacing $0.0005$ within $\pm0.02$ of $d^\star$; the grid is split into two side-by-side halves so that every row from $0.3000$ through $0.5000$ prints.}\label{tab:sweep-750}
\scriptsize\begin{tabular}{@{}rrl@{\hspace{2.2em}}rrl@{}}\toprule
$d_1$ & prob. & 95\% CI & $d_1$ & prob. & 95\% CI \\\midrule
0.3000 & 0.000 & [0.000, 0.019] & 0.4175 & 0.065 & [0.038, 0.108] \\
0.3200 & 0.000 & [0.000, 0.019] & 0.4180 & 0.100 & [0.066, 0.149] \\
0.3400 & 0.000 & [0.000, 0.019] & 0.4185 & 0.100 & [0.066, 0.149] \\
0.3600 & 0.000 & [0.000, 0.019] & 0.4190 & 0.140 & [0.099, 0.195] \\
0.3800 & 0.000 & [0.000, 0.019] & 0.4195 & 0.235 & [0.182, 0.298] \\
0.3975 & 0.000 & [0.000, 0.019] & 0.4200 & 0.285 & [0.227, 0.351] \\
0.3980 & 0.000 & [0.000, 0.019] & 0.4205 & 0.395 & [0.330, 0.464] \\
0.3985 & 0.000 & [0.000, 0.019] & 0.4210 & 0.415 & [0.349, 0.484] \\
0.3990 & 0.000 & [0.000, 0.019] & 0.4215 & 0.520 & [0.451, 0.588] \\
0.3995 & 0.000 & [0.000, 0.019] & 0.4220 & 0.575 & [0.506, 0.641] \\
0.4000 & 0.000 & [0.000, 0.019] & 0.4225 & 0.670 & [0.602, 0.731] \\
0.4005 & 0.000 & [0.000, 0.019] & 0.4230 & 0.760 & [0.696, 0.814] \\
0.4010 & 0.000 & [0.000, 0.019] & 0.4235 & 0.805 & [0.745, 0.854] \\
0.4015 & 0.000 & [0.000, 0.019] & 0.4240 & 0.825 & [0.766, 0.871] \\
0.4020 & 0.000 & [0.000, 0.019] & 0.4245 & 0.940 & [0.898, 0.965] \\
0.4025 & 0.000 & [0.000, 0.019] & 0.4250 & 0.885 & [0.833, 0.922] \\
0.4030 & 0.000 & [0.000, 0.019] & 0.4255 & 0.975 & [0.943, 0.989] \\
0.4035 & 0.000 & [0.000, 0.019] & 0.4260 & 0.975 & [0.943, 0.989] \\
0.4040 & 0.000 & [0.000, 0.019] & 0.4265 & 0.975 & [0.943, 0.989] \\
0.4045 & 0.000 & [0.000, 0.019] & 0.4270 & 0.995 & [0.972, 0.999] \\
0.4050 & 0.000 & [0.000, 0.019] & 0.4275 & 0.990 & [0.964, 0.997] \\
0.4055 & 0.000 & [0.000, 0.019] & 0.4280 & 0.995 & [0.972, 0.999] \\
0.4060 & 0.000 & [0.000, 0.019] & 0.4285 & 1.000 & [0.981, 1.000] \\
0.4065 & 0.000 & [0.000, 0.019] & 0.4290 & 1.000 & [0.981, 1.000] \\
0.4070 & 0.000 & [0.000, 0.019] & 0.4295 & 0.995 & [0.972, 0.999] \\
0.4075 & 0.000 & [0.000, 0.019] & 0.4300 & 1.000 & [0.981, 1.000] \\
0.4080 & 0.000 & [0.000, 0.019] & 0.4305 & 1.000 & [0.981, 1.000] \\
0.4085 & 0.000 & [0.000, 0.019] & 0.4310 & 1.000 & [0.981, 1.000] \\
0.4090 & 0.000 & [0.000, 0.019] & 0.4315 & 1.000 & [0.981, 1.000] \\
0.4095 & 0.000 & [0.000, 0.019] & 0.4320 & 1.000 & [0.981, 1.000] \\
0.4100 & 0.000 & [0.000, 0.019] & 0.4325 & 1.000 & [0.981, 1.000] \\
0.4105 & 0.000 & [0.000, 0.019] & 0.4330 & 1.000 & [0.981, 1.000] \\
0.4110 & 0.000 & [0.000, 0.019] & 0.4335 & 1.000 & [0.981, 1.000] \\
0.4115 & 0.000 & [0.000, 0.019] & 0.4340 & 1.000 & [0.981, 1.000] \\
0.4120 & 0.000 & [0.000, 0.019] & 0.4345 & 1.000 & [0.981, 1.000] \\
0.4125 & 0.000 & [0.000, 0.019] & 0.4350 & 1.000 & [0.981, 1.000] \\
0.4130 & 0.000 & [0.000, 0.019] & 0.4355 & 1.000 & [0.981, 1.000] \\
0.4135 & 0.000 & [0.000, 0.019] & 0.4360 & 1.000 & [0.981, 1.000] \\
0.4140 & 0.000 & [0.000, 0.019] & 0.4365 & 1.000 & [0.981, 1.000] \\
0.4145 & 0.000 & [0.000, 0.019] & 0.4370 & 1.000 & [0.981, 1.000] \\
0.4150 & 0.005 & [0.001, 0.028] & 0.4375 & 1.000 & [0.981, 1.000] \\
0.4155 & 0.010 & [0.003, 0.036] & 0.4400 & 1.000 & [0.981, 1.000] \\
0.4160 & 0.005 & [0.001, 0.028] & 0.4600 & 1.000 & [0.981, 1.000] \\
0.4165 & 0.025 & [0.011, 0.057] & 0.4800 & 1.000 & [0.981, 1.000] \\
0.4170 & 0.045 & [0.024, 0.083] & 0.5000 & 1.000 & [0.981, 1.000] \\
\bottomrule\end{tabular}
\end{table}

%% file: Figures/latex_out/12_table_sweep_3000.tex
\begin{table}[p]\centering
\caption{Finite-horizon backlog-growth probability, $C=mk=3000$ (200 replications per point; Wilson 95\% CIs; fluid threshold $d^\star=0.4175$; fluid growth threshold $d^\mathrm{fluid}=0.4217$; DES step $\mu_1\Delta t=0.083$). Grid: coarse spacing $0.02$ plus refined rows at spacing $0.0005$ within $\pm0.02$ of $d^\star$; the grid is split into two side-by-side halves so that every row from $0.3000$ through $0.5000$ prints.}\label{tab:sweep-3000}
\scriptsize\begin{tabular}{@{}rrl@{\hspace{2.2em}}rrl@{}}\toprule
$d_1$ & prob. & 95\% CI & $d_1$ & prob. & 95\% CI \\\midrule
0.3000 & 0.000 & [0.000, 0.019] & 0.4175 & 0.000 & [0.000, 0.019] \\
0.3200 & 0.000 & [0.000, 0.019] & 0.4180 & 0.000 & [0.000, 0.019] \\
0.3400 & 0.000 & [0.000, 0.019] & 0.4185 & 0.000 & [0.000, 0.019] \\
0.3600 & 0.000 & [0.000, 0.019] & 0.4190 & 0.015 & [0.005, 0.043] \\
0.3800 & 0.000 & [0.000, 0.019] & 0.4195 & 0.025 & [0.011, 0.057] \\
0.3975 & 0.000 & [0.000, 0.019] & 0.4200 & 0.055 & [0.031, 0.096] \\
0.3980 & 0.000 & [0.000, 0.019] & 0.4205 & 0.180 & [0.133, 0.239] \\
0.3985 & 0.000 & [0.000, 0.019] & 0.4210 & 0.275 & [0.218, 0.341] \\
0.3990 & 0.000 & [0.000, 0.019] & 0.4215 & 0.385 & [0.320, 0.454] \\
0.3995 & 0.000 & [0.000, 0.019] & 0.4220 & 0.585 & [0.516, 0.651] \\
0.4000 & 0.000 & [0.000, 0.019] & 0.4225 & 0.720 & [0.654, 0.778] \\
0.4005 & 0.000 & [0.000, 0.019] & 0.4230 & 0.890 & [0.839, 0.926] \\
0.4010 & 0.000 & [0.000, 0.019] & 0.4235 & 0.935 & [0.892, 0.962] \\
0.4015 & 0.000 & [0.000, 0.019] & 0.4240 & 0.980 & [0.950, 0.992] \\
0.4020 & 0.000 & [0.000, 0.019] & 0.4245 & 0.995 & [0.972, 0.999] \\
0.4025 & 0.000 & [0.000, 0.019] & 0.4250 & 1.000 & [0.981, 1.000] \\
0.4030 & 0.000 & [0.000, 0.019] & 0.4255 & 1.000 & [0.981, 1.000] \\
0.4035 & 0.000 & [0.000, 0.019] & 0.4260 & 1.000 & [0.981, 1.000] \\
0.4040 & 0.000 & [0.000, 0.019] & 0.4265 & 1.000 & [0.981, 1.000] \\
0.4045 & 0.000 & [0.000, 0.019] & 0.4270 & 1.000 & [0.981, 1.000] \\
0.4050 & 0.000 & [0.000, 0.019] & 0.4275 & 1.000 & [0.981, 1.000] \\
0.4055 & 0.000 & [0.000, 0.019] & 0.4280 & 1.000 & [0.981, 1.000] \\
0.4060 & 0.000 & [0.000, 0.019] & 0.4285 & 1.000 & [0.981, 1.000] \\
0.4065 & 0.000 & [0.000, 0.019] & 0.4290 & 1.000 & [0.981, 1.000] \\
0.4070 & 0.000 & [0.000, 0.019] & 0.4295 & 1.000 & [0.981, 1.000] \\
0.4075 & 0.000 & [0.000, 0.019] & 0.4300 & 1.000 & [0.981, 1.000] \\
0.4080 & 0.000 & [0.000, 0.019] & 0.4305 & 1.000 & [0.981, 1.000] \\
0.4085 & 0.000 & [0.000, 0.019] & 0.4310 & 1.000 & [0.981, 1.000] \\
0.4090 & 0.000 & [0.000, 0.019] & 0.4315 & 1.000 & [0.981, 1.000] \\
0.4095 & 0.000 & [0.000, 0.019] & 0.4320 & 1.000 & [0.981, 1.000] \\
0.4100 & 0.000 & [0.000, 0.019] & 0.4325 & 1.000 & [0.981, 1.000] \\
0.4105 & 0.000 & [0.000, 0.019] & 0.4330 & 1.000 & [0.981, 1.000] \\
0.4110 & 0.000 & [0.000, 0.019] & 0.4335 & 1.000 & [0.981, 1.000] \\
0.4115 & 0.000 & [0.000, 0.019] & 0.4340 & 1.000 & [0.981, 1.000] \\
0.4120 & 0.000 & [0.000, 0.019] & 0.4345 & 1.000 & [0.981, 1.000] \\
0.4125 & 0.000 & [0.000, 0.019] & 0.4350 & 1.000 & [0.981, 1.000] \\
0.4130 & 0.000 & [0.000, 0.019] & 0.4355 & 1.000 & [0.981, 1.000] \\
0.4135 & 0.000 & [0.000, 0.019] & 0.4360 & 1.000 & [0.981, 1.000] \\
0.4140 & 0.000 & [0.000, 0.019] & 0.4365 & 1.000 & [0.981, 1.000] \\
0.4145 & 0.000 & [0.000, 0.019] & 0.4370 & 1.000 & [0.981, 1.000] \\
0.4150 & 0.000 & [0.000, 0.019] & 0.4375 & 1.000 & [0.981, 1.000] \\
0.4155 & 0.000 & [0.000, 0.019] & 0.4400 & 1.000 & [0.981, 1.000] \\
0.4160 & 0.000 & [0.000, 0.019] & 0.4600 & 1.000 & [0.981, 1.000] \\
0.4165 & 0.000 & [0.000, 0.019] & 0.4800 & 1.000 & [0.981, 1.000] \\
0.4170 & 0.000 & [0.000, 0.019] & 0.5000 & 1.000 & [0.981, 1.000] \\
\bottomrule\end{tabular}
\end{table}

%% file: Figures/latex_out/12_table_sweep_12000.tex
\begin{table}[p]\centering
\caption{Finite-horizon backlog-growth probability, $C=mk=12000$ (200 replications per point; Wilson 95\% CIs; fluid threshold $d^\star=0.4175$; fluid growth threshold $d^\mathrm{fluid}=0.4217$; DES step $\mu_1\Delta t=0.083$). Grid: coarse spacing $0.02$ plus refined rows at spacing $0.0005$ within $\pm0.02$ of $d^\star$; the grid is split into two side-by-side halves so that every row from $0.3000$ through $0.5000$ prints.}\label{tab:sweep-12000}
\scriptsize\begin{tabular}{@{}rrl@{\hspace{2.2em}}rrl@{}}\toprule
$d_1$ & prob. & 95\% CI & $d_1$ & prob. & 95\% CI \\\midrule
0.3000 & 0.000 & [0.000, 0.019] & 0.4175 & 0.000 & [0.000, 0.019] \\
0.3200 & 0.000 & [0.000, 0.019] & 0.4180 & 0.000 & [0.000, 0.019] \\
0.3400 & 0.000 & [0.000, 0.019] & 0.4185 & 0.000 & [0.000, 0.019] \\
0.3600 & 0.000 & [0.000, 0.019] & 0.4190 & 0.000 & [0.000, 0.019] \\
0.3800 & 0.000 & [0.000, 0.019] & 0.4195 & 0.000 & [0.000, 0.019] \\
0.3975 & 0.000 & [0.000, 0.019] & 0.4200 & 0.000 & [0.000, 0.019] \\
0.3980 & 0.000 & [0.000, 0.019] & 0.4205 & 0.015 & [0.005, 0.043] \\
0.3985 & 0.000 & [0.000, 0.019] & 0.4210 & 0.110 & [0.074, 0.161] \\
0.3990 & 0.000 & [0.000, 0.019] & 0.4215 & 0.410 & [0.344, 0.479] \\
0.3995 & 0.000 & [0.000, 0.019] & 0.4220 & 0.680 & [0.612, 0.741] \\
0.4000 & 0.000 & [0.000, 0.019] & 0.4225 & 0.950 & [0.910, 0.973] \\
0.4005 & 0.000 & [0.000, 0.019] & 0.4230 & 0.970 & [0.936, 0.986] \\
0.4010 & 0.000 & [0.000, 0.019] & 0.4235 & 1.000 & [0.981, 1.000] \\
0.4015 & 0.000 & [0.000, 0.019] & 0.4240 & 1.000 & [0.981, 1.000] \\
0.4020 & 0.000 & [0.000, 0.019] & 0.4245 & 1.000 & [0.981, 1.000] \\
0.4025 & 0.000 & [0.000, 0.019] & 0.4250 & 1.000 & [0.981, 1.000] \\
0.4030 & 0.000 & [0.000, 0.019] & 0.4255 & 1.000 & [0.981, 1.000] \\
0.4035 & 0.000 & [0.000, 0.019] & 0.4260 & 1.000 & [0.981, 1.000] \\
0.4040 & 0.000 & [0.000, 0.019] & 0.4265 & 1.000 & [0.981, 1.000] \\
0.4045 & 0.000 & [0.000, 0.019] & 0.4270 & 1.000 & [0.981, 1.000] \\
0.4050 & 0.000 & [0.000, 0.019] & 0.4275 & 1.000 & [0.981, 1.000] \\
0.4055 & 0.000 & [0.000, 0.019] & 0.4280 & 1.000 & [0.981, 1.000] \\
0.4060 & 0.000 & [0.000, 0.019] & 0.4285 & 1.000 & [0.981, 1.000] \\
0.4065 & 0.000 & [0.000, 0.019] & 0.4290 & 1.000 & [0.981, 1.000] \\
0.4070 & 0.000 & [0.000, 0.019] & 0.4295 & 1.000 & [0.981, 1.000] \\
0.4075 & 0.000 & [0.000, 0.019] & 0.4300 & 1.000 & [0.981, 1.000] \\
0.4080 & 0.000 & [0.000, 0.019] & 0.4305 & 1.000 & [0.981, 1.000] \\
0.4085 & 0.000 & [0.000, 0.019] & 0.4310 & 1.000 & [0.981, 1.000] \\
0.4090 & 0.000 & [0.000, 0.019] & 0.4315 & 1.000 & [0.981, 1.000] \\
0.4095 & 0.000 & [0.000, 0.019] & 0.4320 & 1.000 & [0.981, 1.000] \\
0.4100 & 0.000 & [0.000, 0.019] & 0.4325 & 1.000 & [0.981, 1.000] \\
0.4105 & 0.000 & [0.000, 0.019] & 0.4330 & 1.000 & [0.981, 1.000] \\
0.4110 & 0.000 & [0.000, 0.019] & 0.4335 & 1.000 & [0.981, 1.000] \\
0.4115 & 0.000 & [0.000, 0.019] & 0.4340 & 1.000 & [0.981, 1.000] \\
0.4120 & 0.000 & [0.000, 0.019] & 0.4345 & 1.000 & [0.981, 1.000] \\
0.4125 & 0.000 & [0.000, 0.019] & 0.4350 & 1.000 & [0.981, 1.000] \\
0.4130 & 0.000 & [0.000, 0.019] & 0.4355 & 1.000 & [0.981, 1.000] \\
0.4135 & 0.000 & [0.000, 0.019] & 0.4360 & 1.000 & [0.981, 1.000] \\
0.4140 & 0.000 & [0.000, 0.019] & 0.4365 & 1.000 & [0.981, 1.000] \\
0.4145 & 0.000 & [0.000, 0.019] & 0.4370 & 1.000 & [0.981, 1.000] \\
0.4150 & 0.000 & [0.000, 0.019] & 0.4375 & 1.000 & [0.981, 1.000] \\
0.4155 & 0.000 & [0.000, 0.019] & 0.4400 & 1.000 & [0.981, 1.000] \\
0.4160 & 0.000 & [0.000, 0.019] & 0.4600 & 1.000 & [0.981, 1.000] \\
0.4165 & 0.000 & [0.000, 0.019] & 0.4800 & 1.000 & [0.981, 1.000] \\
0.4170 & 0.000 & [0.000, 0.019] & 0.5000 & 1.000 & [0.981, 1.000] \\
\bottomrule\end{tabular}
\end{table}